\documentclass[11pt]{amsart}

\usepackage{amssymb,amsfonts,amsthm,amsmath,mathrsfs,xspace,hyperref, bm}
\usepackage{graphicx}
\usepackage[francais,english]{babel}
\usepackage{inputenc}
\usepackage{tikz-cd}
\usepackage{tensor}
\usepackage[T1]{fontenc}

\usepackage[centering, margin=3cm]{geometry}

\usepackage{enumerate}
\usepackage{enumitem}

\makeatletter
\def\bign#1{\mathclose{\hbox{$\left#1\vbox to8.5\p@{}\right.\n@space$}}\mathopen{}}
\makeatother

  \newcommand{\R}{\ensuremath{\mathbb{R}}}%
  \newcommand{\Z}{\ensuremath{\mathbb{Z}}}%
  \newcommand{\N}{\ensuremath{\mathbb{N}}}%
\newcommand{\src}{\mathsf{s}}
\newcommand{\rg}{\mathsf{r}}

    \newcommand{\A}{\ensuremath{\mathcal{A}}}%
        \newcommand{\G}{\ensuremath{\mathcal{G}}}%
        \newcommand{\tG}{\ensuremath{\widetilde{\mathcal{G}}}}%
        \newcommand{\tH}{\ensuremath{\widetilde{\mathcal{H}}}}%
                \newcommand{\tI}{\ensuremath{\widetilde{\mathcal{I}}}}%
                                \newcommand{\tK}{\ensuremath{\widetilde{\mathcal{K}}}}%

                                \newcommand{\tP}{\ensuremath{\widetilde{\mathcal{P}}}}%
                                \newcommand{\fm}{\ensuremath{\operatorname{FM}}}

        \newcommand{\T}{\ensuremath{\mathcal{T}}}%

                \newcommand{\Q}{\ensuremath{{Q}}}%
                \newcommand{\Qc}{\ensuremath{\mathcal{Q}}}%

                \newcommand{\F}{\ensuremath{\vec{F}}}%
                \newcommand{\Tk}{\ensuremath{\vec{T}}}%

                                \newcommand{\gk}{\ensuremath{\vec{\gamma}}}%

\renewcommand{\fg}{\ensuremath{\operatorname{FG}}}

                \renewcommand{\P}{\ensuremath{\mathcal{P}}}%
                \newcommand{\Sc}{\ensuremath{\mathcal{S}}}%

        \newcommand{\Hcal}{\ensuremath{\mathcal{H}}}%
        
                \renewcommand{\H}{\ensuremath{\mathcal{H}}}%

                \newcommand{\K}{\ensuremath{\mathcal{K}}}%

    \newcommand{\sym}{\ensuremath{\operatorname{Sym}}}%

  \newcommand{\supp}{\ensuremath{\operatorname{Supp}}}%
				\newcommand{\orb}{\ensuremath{\Gamma}}%

		  \newcommand{\alt}{\ensuremath{\operatorname{Alt}}}%

    \newcommand{\acts}{\ensuremath{\curvearrowright}}%
  \newcommand{\sub}{\ensuremath{\operatorname{Sub}}}%
  \newcommand{\homeo}{\ensuremath{\operatorname{Homeo}}}%
        \newcommand{\asdim}{\ensuremath{\operatorname{asdim}}}%

    	\newcommand{\fix}{\ensuremath{\operatorname{Fix}}}%

    \newcommand{\st}{\ensuremath{\operatorname{St}}}%
    \newcommand{\rist}{\ensuremath{\operatorname{RiSt}}}%
    \renewcommand{\ker}{\ensuremath{\operatorname{ker}}}
    
    \newcommand{\Ff}{\ensuremath{\mathsf{F}}}
        \newcommand{\Af}{\ensuremath{\mathsf{A}}}
                \newcommand{\Df}{\ensuremath{\mathsf{D}}}

    \newcommand{\cay}{\ensuremath{\widetilde{\Gamma}}}%
    \newcommand{\sch}{\ensuremath{\Gamma}}%

		\newcommand{\iet}{\ensuremath{\operatorname{IET}}}

\theoremstyle{definition}
  \newtheorem{defin}{Definition}[section]

\newtheorem{step}{Step}
\newtheorem*{claim}{Claim}

\theoremstyle{plain}
  
  \newtheorem{cor}[defin]{Corollary}
    \newtheorem{main cor}{Corollary}
  
    \newtheorem*{main ex-cor}{Corollary}

  \newtheorem{thm}[defin]{Theorem}
  \newtheorem{main thm}[main cor]{Theorem}
  \newtheorem{prop}[defin]{Proposition}
    \newtheorem{prop-def}[defin]{Proposition-Definition}
      
  \newtheorem*{principle}{Main phenomenon}
    \newtheorem{stand-assump}[defin]{Standing Assumption}

  \newtheorem{lemma}[defin]{Lemma}

\theoremstyle{remark}
  \newtheorem{remark}[defin]{Remark}
    \newtheorem{remarks}[defin]{Remarks}
    \newtheorem{main remarks}[main cor]{Remarks}

  \newtheorem{example}[defin]{Example}
  \newtheorem{examples}[defin]{Examples}

\date{\today}	
\title[Rigidity of full groups]{Rigidity properties of full groups of pseudogroups over the Cantor set}\author{Nicol\'as Matte Bon
}

\let\oldtocsection=\tocsection
 
\let\oldtocsubsection=\tocsubsection
 
\let\oldtocsubsubsection=\tocsubsubsection
 
\renewcommand{\tocsection}[2]{\hspace{0em}\oldtocsection{#1}{#2}}
\renewcommand{\tocsubsection}[2]{\hspace{2em}\oldtocsubsection{#1}{#2}}
\renewcommand{\tocsubsubsection}[2]{\hspace{2em}\oldtocsubsubsection{#1}{#2}}

\begin{document}
\maketitle
\setcounter{tocdepth}{1}

%
%
%

%
\begin{abstract}
 We show that the (topological) full group of a minimal pseudogroup over the Cantor set satisfies various  rigidity phenomena of topological dynamical and combinatorial nature.
Our main result applies to its possible homomorphisms into other groups of homeomorphisms, and implies that arbitrary homomorphisms between the full groups of a vast class of pseudogroups must extend to continuous morphisms between pseudogroups (in particular giving rise to  equivariant maps at the level of spaces).   As applications, we obtain explicit  obstructions to the existence of embeddings between full groups in terms of  invariants of the underlying pseudogroups (the geometry of their orbital graphs,  the complexity function, dynamical homology), and provide a complete descriptions of  all  homomorphisms within various families of groups  including the  Higman-Thompson groups  (more generally full groups of one sided shifts of finite type),  full groups of  minimal $\Z$-actions on the Cantor set,   and  a class of groups of interval exchanges.

We next  consider a combinatorial rigidity property of groups, which formalises the inability of a group to act on any set  with Schreier graphs growing uniformly  subexponentially, or more generally not faster than a given function $f(n)$. For the exponential function this is a well-known consequence of property $(T)$.   We use full groups to provide a source of examples of groups which satisfy this property but satisfy a strong negation of property $(T)$.

A key tool  used in the proofs is the study of the dynamics of the conjugation action of  groups  on their space of subgroups, endowed with the Chabauty topology. In particular we classify the confined and the uniformly recurrent subgroups of full groups.

\end{abstract}

\tableofcontents

\section{Introduction}

\subsection{Extending homomorphisms from groups to pseudogroups} 
The first purpose of this paper is to introduce  a new method  to study homomorphisms between a class of countable groups of dynamical origin, which goes beyond  the classical reconstruction methods used to study isomorphisms between ``rich'' groups of transformation.   This method provides general tools to prove  non-embedding results between  groups, including cases in which  the subgroup structure of the target can be difficult to study. 

The natural framework to state our results is the theory of \emph{pseudogroups} and  the equivalent language of \emph{\'etale groupoids}.  All definitions on pseudogroups and groupoids can be found with full details in  \S \ref{s-pseudogroups}; we only sketch them in this introduction for the reader convenience.

Let $X$ be a compact space.  A \textbf{pseudogroup} of partial homeomorphisms of $X$ is a collection $\tG$ of homeomorphisms $F\colon U\to V$ between open subsets of $X$ which is closed under taking compositions, inverses,  unions  of  families  of arbitrary cardinality (which are compatible on the intersections of their domains), and which contains the identity restricted to every open subset of $X$ (see \S \ref{s-pseudogroups} for a more detailed and more general definition). 

  Given  a partial homeomorphism $F\colon U\to V$ of $X$ and a point $x\in U$, the \textbf{germ} $[F]_x$ is defined as the equivalence class of the pair $(F, x)$, where $(F_1, x_1)$ and $(F_2, x_2)$ are equivalent if $x_1=x_2$ and $F_1$ and $F_2$ coincide on a neighbourhood of $x_1$. A \textbf{groupoid of germs}  over $X$ is the set $\G$ of germs of some pseudogroup, and determines the pseudogroup uniquely (see \S \ref{s-groupoids} for preliminaries on groupoids).  We will denote $\tG$ the pseudogroup associated to  a groupoid of germs $\G$ (so that $\G$ the groupoid of germs associated to a pseudogroup $\tG$). 
  
 Given a groupoid of germs $\G$, the subset of the pseudogroup $\Ff(\G)\subset \tG$ consisting of elements whose domain and range is the whole space $X$ is a group, called the  \textbf{full group}\footnote{This group is usually called the \textbf{topological full group}, to distinguish it from a larger full group acting on $X$ by preserving each $\G$-orbit. In this paper we will only consider topological full groups, and thus  omit the word ``topological''. Note that the orbit full group  can be seen as the topological full group of a larger pseudogroup, thus many of our results can also be applied to it.}  of $\G$  (or of the pseudogroup $\tG$).  This notion is particularly relevant when  the space $X$ is homeomorphic to the  Cantor set (for short a \emph{Cantor space}), and  our main results are mostly concerned with this case.
As an example, let $G$ be a countable group acting on $X$ by homeomorphisms. The {groupoid of germs} of the action $\G:= \G_{G \acts X}$ is the groupoid of all germs that are represented by elements of $G$. Its full group will be denoted $\Ff(G, X)$ and  consists of all  homeomorphisms of $X$ that locally coincide with elements of $G$.    An important class of group actions on the Cantor set are \emph{subshifts}, i.e. translation action $G\acts X\subset A^G$ on a closed invariant subset of the shift over a finite alphabet. More sources of examples of groupoids  over the Cantor set are aperiodic tilings of Euclidean space,  one-sided shifts of finite type,  AF-groupoids associated to Bratteli diagrams, foliations with a Cantor invariant transversal, etc. 

The notion of full group was introduced  by Dye \cite{Dye} in the context of orbit equivalence of measure preserving group actions. In the topological dynamical setting it was first studied in detail by  Giordano-Putnam-Skau \cite{G-P-S-full} 
(for $\Z$-actions on the Cantor set) and  Matui \cite{Mat-simple, Mat-hom}. More recently the theory of groupoids of germs and their full groups  has attracted considerable attention in group theory, leading to important results in the study of amenability and growth of groups, see the papers  \cite{Ju-Mo, J-N-S, Nek-frag} and the  surveys  \cite{Cor-Bou, Mat-survey}. It  also provides a unifying framework to study several families of groups, including the Higman--Thompson groups and many of their generalisations, groups of intermediate growth, groups of interval exchanges, block-diagonal limits of finite alternating groups, etc.

We will often restrict the attention to a relevant normal subgroup  $\Af(\G)\unlhd \Ff(\G)$ introduced by Nekrashevych \cite{Nek-simple}, called the the \textbf{alternating full group}. 
The group $\Af(\G)$ can be defined as  the subgroup of $\Ff(\G)$ generated by 3-\emph{cycles}, i.e. elements of order 3 that act by permuting cyclically 3 disjoint clopen sets, and as the identity elsewhere (se \S \ref{s-alt} for a more detailed definition).  It is shown in \cite{Nek-simple} that if $\G$ is minimal (i.e.~all its orbits are dense) then $\Af(\G)$ is simple, and if $\G$ is \emph{expansive} (e.g.~the groupoid of germs of a subshifts) then $\Af(\G)$ is also finitely generated.

\medskip

The phenomenon highlighted by the main result of this paper can be informally described as follows.

\begin{principle}{ In  many ``nice'' situations,  a homomorphism from the group $\Af(\G)$ to another group of homeomorphism must {extend} to a {continuous morphism} of pseudogroups.}
\end{principle}

Precise statements will be given shortly. Let us first clarify the terminology. 
\begin{defin}
Given pseudogroups $\tG, \tH$  over spaces $X, Y$, a \textbf{continuous morphism}  is a map $\varphi\colon \tG\to \tH$   which preserves compositions, inverses, unions of compatible families, sends the identity over $X$ to the identity over $Y$, and the empty  homeomorphism  (or \emph{zero}) of $\tG$ to the zero of $\tH$.    
\end{defin}

 The word ``continuous'' does not refer to a topology on $\tG,\tH$, but to the fact that a continuous morphism $\varphi\colon \tG\to \tH$ automatically comes with an associated  continuous map $q\colon Y\to X$, called the \textbf{spatial component} of $\varphi$, which intertwines $\varphi$ and the actions of $\tG$ and $\tH$. The existence of this map follows from the definition by a simple ``Stone duality'' argument  (see \S \ref{s-morphisms} for details).
 
 \medskip
 
Given a continuous morphism $\varphi\colon \tG\to \tH$,   its restriction to the full group $\Ff(\G)\subset \tG$,   is  a group homomorphism taking values in $\Ff(\H)$. Conversely given any groups of homeomorphisms $G, H$ of spaces $X, Y$, it is meaningful and natural to ask whether a group homomorphism $\rho\colon G\to H$ extends to  a continuous morphism  $\tG\to \tH$ between pseudogroups that contain $G$ and $H$. As explained, the existence of such an extension implies the existence of a continuous map $q\colon Y\to X$ which intertwines $\rho$ and the actions. It is however a much more precise conclusion:  we shall show along the paper that it allows to use  a rich pool of invariants associated to groupoids (homology, geometry and growths of the leaves, dynamical invariants)  in order to study group homomorphisms and prove non-embedding results.

\medskip

To place our main theorem  into context, let us first state a   consequence of it  which is already well-known. 

\begin{cor}[Rubin--Matui's Isomorphism Theorem;  \cite{Rubin} \cite{Med-isom} \cite{Mat-SFT}]\label{t-intro-isom} \label{c-isom}
For $i=1,2$ let  $\G_i$ be minimal groupoid of germs over a Cantor space $X_i$, and let $G_i$ be a subgroup of $\Ff(\G_i)$ containing $\Af(\G_i)$. Then every group isomorphism $\rho\colon G_1\to G_2$ extends to a continuous isomorphism of pseudogroups $\tG_1\to \tG_2$. 

\end{cor}
The conclusion of this result is usually stated as the existence of a homeomorphisms $q\colon X_2\to X_1$ that realises the isomorphism $\rho$ by conjugacy (in this special case, this is equivalent to the extension to an isomorphisms $\tG_1\to \tG_2$). 
This result belongs to a broad  family of reconstructions results for isomorphisms between ``large'' groups of transformations. Its origin can be traced back to Dye's seminal work \cite{Dye} in the measure theoretical setting. Dye's method was imported by Giordano-Putnam-Skau \cite{G-P-S-full}  to the setting of full groups of $\Z$-actions on the Cantor set, and extended to larger classes of pseudogroups by  Medynets \cite{Med-isom} and  Matui, who  proved  the theorem above \cite[Th. 3.9]{Mat-SFT}.  More general reconstruction results for isomorphisms between groups of homeomorphisms had meanwhile been proven  by Rubin \cite{Rubin}, and imply  this result as a special case  \cite[Th. 0.2]{Rubin}.


\medskip

It is considerably less understood whether any reasonable results for \emph{homomorphisms} between groups of transformations, in the same spirit of this classical reconstruction theory,  hold true. A very general version of  this problem is raised by Rubin in \cite[p. 7]{Rubin} (for groups of homeomorphisms), and  little results in this direction have been obtained since. In the setting of  full groups no such result  was known so far, neither in the topological dynamical, nor in the measure theoretical setting. In the special case of full groups of minimal  $\Z$-actions on the Cantor set, this is asked in \cite[Question (2f)]{Cor-Bou}. A complete understanding of all  homomorphisms appears to be missing also in concrete special cases such as the   Higman-Thompson groups and their relatives. Here we introduce a new method that allows to state and prove such results. 

\medskip

\medskip 

We need some more terminology to state our main extension theorem.  Given $r\ge 1$, we let $X^{[r]}=X^{r}\bign/\sim$ be the compact space obtained from the cartesian product $X^r$ by identifying any  two $r$-tuples $(x_1,\ldots, x_r)$ and $(y_1,\ldots, y_r)$ if the sets $\{x_1,\ldots, x_r\}$ and $\{y_1,\ldots, y_r\}$ are equal.  Every  $F\in \tG$  can be seen as a partial homeomorphisms of $X^{[r]}$, whose domain consists of sets $\{x_1,\ldots, x_r\}$ contained in the domain of definition of $F$. 
The groupoid of germs over $X^{[r]}$ consisting of germs of \ these transformations is called the $r$th \textbf{symmetric power} of $\G$  and is denoted ${\G^{[r]}}$. The natural action of $\Ff(\G)$ on $X^{[r]}$ realises it as a subgroup of $\Ff(\G^{[r]})$. Thus we can view the groups $\Ff(\G)$ and $\Af(\G)$ as contained  in the pseudogroup $\widetilde{\G^{[r]}}$ for every $r\ge 1$.  

The \textbf{support} of a group $G\le \homeo(Y)$ is the closure of the set of points of $Y$ which are moved by at least one element of $G$.  The following is our main theorem in a slightly simplified form (see Theorem \ref{t-main}). \begin{thm}[Extension Theorem] \label{t-intro-main}
Let $\G$ be a minimal groupoid of germs over a Cantor space $X$, and  $\H$ be a groupoid of germs over any compact space $Y$. Let  $\rho\colon \Af(\G)\to \Ff(\H)$ be a non-trivial homomorphism. Denote by $Z\subset Y$ be the support of $\rho(\Af(\G))$. One of the following mutually exclusive cases holds. 
\begin{enumerate}[label=(\roman*)]
\item \label{i-intro-free} There exists $y\in Z$ such that the corresponding germ  map
\[\Af(\G)\to \H, \quad g\mapsto [\rho(g)]_y\]
is  injective.

\medskip
 
\item \label{i-intro-factors} \label{i-intro-factor}The set $Z$ is clopen, and  there exists $r\ge 1$ such that $\rho$ uniquely extends to a continuous morphism of pseudogroups 
\[\widetilde{\rho}\colon \widetilde{\G^{[r]}} \to \widetilde{\H|Z}.\] (In particular there exists a continuous map $q\colon Z\to X^{[r]}$ which intertwines   $\rho$ and the actions.)
\end{enumerate}

%
%
\end{thm}


\begin{remarks} Let us make some comments on the statement.

\begin{enumerate}[label=(\alph*),leftmargin=*]

\item  The notation $\H|Z$ denotes the \textbf{restriction} of the groupoid $\H$ to $Z$, i.e. the set of germs in $\H$ with source and target in $Z$.

\medskip


\item While it is easier to see that case \ref{i-intro-free} is an obstruction to the existence of an extension a in \ref{i-intro-factor}, the main content of the theorem is that it is the only obstruction. 

\medskip

\item Obviously the theorem can be applied to any action $\Af(\G)\to \homeo(Y)$ of the group $\Af(\G)$ on a compact space, by letting $\H$ be the groupoid  of {all} germs of partial homeomorphisms of $Y$. The relevance of considering smaller groupoids will be apparent shortly. 

\medskip

\item If a homomorphism $\rho$ is defined on some overgroup $G$ of $\Af(\G)$ in $\Ff(\G)$, the theorem can of course still be used by considering the  restriction $\rho|_{\Af(\G)}$. In Theorem \ref{t-main} we will provide a more precise conclusion in this case.

\medskip

\item The natural  embedding $\Af(\G)\hookrightarrow \Ff(\G^{[r]})$ does \emph{not} extend to  a continuous pseudogroup morphism $\tG \to \widetilde{\G^{[r]}}$ as soon as $r\ge 2$. (See Proposition \ref{p-diagonal-does-not-extend}, or just notice that there is no ``canonical choice'' of a  continuous map $X^{[r]}\to X$). This shows that it is necessary to pass to a symmetric power $\G^{[r]}$ for the theorem to be true. 
\medskip


\medskip

\item  One example of homomorphism $\Af(\G)\to \homeo(Y)$ which {always} falls in case \ref{i-intro-free}  (an thus does not extend to $\widetilde{\G^{[r]}}$)   is provided by the Bernoulli shift action of  $\Af(\G)$ on $Y=\{0,1\}^{\Af(\G)}$. Taking the diagonal action on $Y=\{0,1\}^{\Af(\G)}\times X^{[r]}$, we obtain an action that does admit a continuous equivariant map to $X^{[r]}$, which however does   not come from a continuous morphism from $\widetilde{\G^{[r]}}$ (this action still falls in Case \ref{i-intro-free}).

\medskip
\end{enumerate}
\end{remarks}
 As mentioned, the Extension Theorem recovers   Rubin--Matui's Isomorphism Theorem \cite{Mat-SFT} (whose derivation is explained  in Corollary \ref{c-isom}). It is also formally a generalisation of  another property of the group $\Af(\G)$ that has already been mentioned: 
  \begin{cor}[ \cite{Nek-simple}]
 Under the same assumptions, the group $\Af(\G)$ is simple. 
 \end{cor}
 (Its derivation  is explained in Example \ref{e-simple}).   In fact while the proof of the Extension Theorem uses a  quite different method from the previous proofs of  reconstruction results for isomorphisms, this method  is explicitly  based on  a generalisation of the simplicity of $\Af(\G)$ (explained in \S \ref{s-intro-confined} below). 
 

%
%
%
%
%

In addition, the Extension Theorem provides indeed a tool study arbitrary {embeddings} between many well-studied classes of full  groups.   The reason is that a homomorphism can \emph{never} fall in case \ref{i-intro-free} provided the  target groupoid $\H$ has a  ``sufficiently nice''  geometry.  In particular we have the following result, whose statement uses Gromov's notion of asymptotic dimension  \cite{Gro-asymptotic} (recalled in  \S \ref{s-coarse}).

\begin{cor}[Automatic Extension for targets with finite dimension; see Theorem \ref{t-asdim}] \label{c-intro-asdim}
Retain the same assumptions on $\G, \H$ and assume further that they are ``compactly generated'' and that every ``{leaf}''  of  $\H$ has finite asymptotic dimension.

Then for every non-trivial homomorphism $\rho\colon \Af(\G)\to \Ff(\H)$, the support $Z$ of $\rho(\Af(\G))$ is a clopen subset of $Y$, and there exists $r\ge 1$ such that $\rho$ uniquely extends to a continuous morphism $\widetilde{\rho}\colon \widetilde{\G^{[r]}}\to \widetilde{\H|Z}$.

Moreover, if every leaf of $\H$ has asymptotic dimension bounded above by $d\in\N$, then the same conclusion holds  for some $r\le d$. 
\end{cor}

Let us explain the terminology. The definition of \textbf{compactly generated} groupoid is due to Haefliger, and roughly means that there exists a suitable finite subset $\mathcal{T}\subset \tH$  which \emph{generates} the groupoid $\H$, in the sense that every germ of $\H$ coincides with the germ of a finite product of elements of $\mathcal{T}$ (the precise definition can be found in \S \ref{s-leaves}).  The \textbf{leaf} of the groupoid $\H$  based at a point $y\in Y$  is defined as the subset $\H_y\subset \H$ consisting of  germs of the form $[F]_y$ for $F\in \tH$ whose domain contains  $y$. For compactly generated groupoids every  leaf $\H_y$ can be endowed with a structure of \textbf{Cayley graph}, whose  quasi-isometry is independent on the choice of the generating set $\T$, and thus has a well-defined asymptotic dimension (detailed definitions can be found in \S \ref{s-coarse} and \S \ref{s-leaves}). 
 \begin{examples} \label{e-intro-asdim}
Corollary \ref{c-intro-asdim} can be applied to homomorphisms between the full groups of several pseudogroups studied in the literature:

\begin{enumerate}[leftmargin=*]
\item \label{i-Z} The groupoid of germs of actions of  $\Z$. The leaves of such groupoids have asymptotic dimension 1. In particular the last sentence of  Corollary \ref{c-intro-asdim} implies that homomorphisms between full groups of minimal $\Z$-actions must extend to continuous morphisms of pseudogroups, without need to pass to a symmetric power $\G^{[r]}$. 

\medskip

\item More generally, the groupoid of germs of any (topologically) free action of a finitely generated group with finite asymptotic dimension. In this case the leaves of the groupoid are quasi-isometric to the acting group, and thus have the same asymptotic dimension. Thus Corollary \ref{c-intro-asdim} applies to  homomorphisms between the full groups of topologically free actions of: finitely generated abelian and nilpotent groups \cite{Be-Dra-poly},   discrete subgroups of connected Lie groups \cite{Ca-Go-asdim, Ji-asdim},   hyperbolic groups \cite{Gro-asymptotic, Roe-hyp},  relatively hyperbolic groups whose parabolic  subgroups have finite dimension \cite{Os-asdim},  mapping class groups \cite{BBF-asdim},  cubulable groups \cite{Wr-asdim}. 

\medskip

\item Groupoids associated to \emph{one-sided shifts of finite type} and to products of them (see \S \ref{s-SFT}).  Their full groups were studied studied by Matui in \cite{Mat-SFT, Mat-prod}, and include as a special case the family of Higman-Thompson's groups $V_{n,k}$ over the Cantor set. The leaves of the corresponding groupoids have  dimension 1 (they are quasi-isometric to trees) thus  also in this case there is no need to pass to a symmetric power $\G^{[r]}$. 
\medskip

\item Groupoids associated to aperiodic tilings and quasicrystals in the euclidean space (such as the Penrose tiling). Such groupoids are well-studied in symbolic dynamics \cite{tiling1, tiling2, tiling3}. 
The study of their  full group is suggested in \cite[\S 6.3]{Nek-simple}. 

\medskip

\item The groupoid of germs of the fragmentation groups introduced by Nekrashevych, whose  full groups are torsion and  have intermediate growth \cite{Nek-frag}. Their leaves have asymptotic dimension 1 (see \cite{Nek-frag}).

\end{enumerate}
\end{examples}


\subsection{Some concrete examples} \label{s-intro-concrete}
Let us  illustrate how the Extension Theorem can be used in practice to study embeddings between groups, by  considering some concrete examples (we refer to \S \ref{s-concrete}-\ref{s-SFT} for a more detailed discussion). 

One of the most  studied classes of full groups are topological full groups of minimal $\Z$-actions, or \emph{Cantor minimal systems}. Given a Cantor minimal system $(X, u)$  (where $u$ is a minimal homeomorphism of a Cantor space) we denote by  $\Ff(X, u)$  its topological full group. See \cite{Cor-Bou} for a survey on full groups of Cantor minimal systems. 

A system $(X_2, u_2)$ \textbf{factors onto} a system $(X_1, u_1)$ if there exists a continuous map $q\colon X_2\to X_1$ such that $u_1\circ q=q\circ u_1$. If this happens, it is easy to see that the group $\Ff(X_1, u_1)$ embeds into $\Ff(X_2, u_2)$ (see \S \ref{s-Z}). The following result says that all embeddings arise from a a natural modification of this construction  
(this answers the question of Cornulier  \cite[Question 2f]{Cor-Bou} mentioned earlier).

\begin{thm}\label{t-intro-Z}
For $i=1,2$ let $(X_i, u_i)$ be a Cantor minimal system. The following are equivalent.
\begin{enumerate}[label=(\roman*)]
\item The group $\Ff(X_1, u_1)$ embeds into $\Ff(X_2, u_2)$.
\item There exist an element $v\in \Ff(X_2, u_2)$ such that  the system $(v, \supp(v))$ factors onto $(X_1, u_1)$. 
\end{enumerate}
\end{thm}
(See Theorem \ref{t-Z-concrete} for a more precise result including a description of the embeddings).

\medskip

In \S \ref{s-iet}, we consider the group $\iet$ of \textbf{interval exchange transformations}. This is defined as the group of  all piecewise continuous bijections of $\R/\Z$ with finitely many discontinuities, which are translations in restriction to each interval of continuity. This group has attracted attention recently, especially in relation to its subgroup structure, which remains however quite mysterious. For example, it is still not known whether the group $\iet$ contains non-abelian free subgroups.

Fix a finitely generated additive subgroup $\Lambda<\R/\Z$, that we assume to be dense (equivalently, to contain at least one irrational element). We let $\iet(\Lambda)$ be the countable subgroup of $\iet$ consisting of elements whose discontinuity points belong to $\Lambda$, and that in restriction to every interval of continuity coincide with a translation by an element of $\Lambda$. This group is isomorphic to the full group of a natural action  $\Lambda\acts X_\Lambda$ on a Cantor space obtained by ``blowing up'' an orbit of the rotation action $\Lambda \acts \R/\Z$ (see \S \ref{s-iet}). Its derived subgroup $\iet(\Lambda)'$ is simple and finitely generated \cite{Mat-SFT, Nek-simple}. Every finitely generated group which embeds into $\iet$ embeds into $\iet(\Lambda)$ for some $\Lambda$, thus it is natural to ask to what extent the subgroup structure of  $\iet(\Lambda)$ depends on $\Lambda$ and its  arithmetic properties.  Very few results in this direction are known \cite{ExtAmen}, and the dependency on $\Lambda$ is only on its rank as an abstract abelian group. The next result shows  a fine obstruction on the possible embeddings between  groups in this family in terms of $\Lambda$.

\begin{thm}
Let $\Lambda, \Delta< \R/\Z$ be finitely generated groups of rotations, and let $\widehat{\Lambda}, \widehat{\Delta}<\R$ be their preimages  under the quotient map   $\R\to \R/\Z$. The following are equivalent.
\begin{enumerate}[label=(\roman*)]
\item There exists a group embedding of   $\iet(\Lambda)$  into $\iet(\Delta)$.
\item There exists $c\in (0, 1]$ such that the homothety  $\R\to \R, x\mapsto cx$ maps $\widehat{\Lambda}$ into $\widehat{\Delta}$. 
\end{enumerate}
\end{thm}
(See Theorem \ref{t-sturmian-embedding} for a more precise result).

\medskip

In \S \ref{s-SFT}, we consider the full groups of groupoids arising from \emph{one sided shifts of finite type} \cite{Mat-SFT}. For simplicity we discuss here only a special case, namely the family of \emph{Higman--Thompson groups} $V_{n}$. For $n\ge 1$ let $X_n=\{1,\ldots , n\}^\N$ be the Cantor set of $n$-ary one sided sequences. Recall that the Higman-Thompson group $V_{n,1}$ (also denoted $V_{n}$) is the group of all self-homeomorphisms $g$ of $X_n$ with the property that there exists finite collections of finite binary words $\{v_i\}_{i=1}^\ell$ and $\{w_i\}_{i=1}^\ell$  such that every infinite binary sequence begins with one of the words $v_i$ as a prefix, and   $g(v_ix)=w_ix$. It is the full group of a natural groupoid, generated by the germs of the one sided shift $T\colon X_n\to X_n$. The group $V:=V_{2, 1}$ is known as Thompson's group $V$, and was the first example of a finitely presented infinite simple group (see \cite{C-F-P-th}). 

It was proven by Higman that the groups $V_{n,1}$ and $V_{m,1}$ are isomorphic if and only if $n=m$, and their automorphisms were studied in \cite{aut-V} using Rubin's theorem. On the other hand all these groups look quite similar to each other ``qualitatively''. It is well-known and easy to see that they can all be embedded into each other, and multiple  constructions of homomorphisms between them and of endomorphisms can be found in the literature (see e.g.  \cite{Bu-Cl-St-dist,  Mat-prod, Gen-V, Don-endV}). But no general classification of their homomorphisms appears to be known. In \S \ref{s-SFT}, we  use the Extension Theorem to provide a complete description of all homomorphisms within the family  $V_{n, 1}$ (more generally, within full groups of one-sided shifts of finite type).   More precisely, we describe an explicit construction of  homomorphisms between them (inspired by  Matui's work \cite{Mat-prod}) and show that all homomorphisms must arise from this construction. We refer to \S \ref{s-SFT} for details.

 Let us state here only a simple corollary of this description. It says that  although  the groups $V_{n, 1}$ all embed into each other, their embeddings verify a ``fixed point property'':
\begin{cor}[See Corollary \ref{c-V-support}]
Assume that $(n-1) \nmid  (m-1)$. Then for every homomorphism $\rho\colon V_{m,1}\to V_{n,1}$, the action of $\rho(V_{m,1}')$ on $X_n$ has a non-empty clopen set of global fixed points. 
\end{cor}
This can be seen as a ``dynamical'' strengthening of the fact that $V_{n,1}$ is isomorphic to $V_{m,1}$ if and only if $n=m$.


There are many generalisations of the Higman--Thompson groups in the literature, to which the Extension Theorem can also be applied. For example in \S \ref{s-product-SFT} we also consider the family of higher dimension Thompson's groups $nV$ defined by Brin \cite{Brin-nV} (not to be confused with $V_{n, 1}$). Using Rubin's theorem, Brin observed that the groups $2V$ and $V$ are not isomorphic \cite{Brin-nV}, and this was later generalised by Bleak and Lanoue \cite{Bl-La-nV} who showed that $nV$ is isomorphic to $mV$ if and only if $n=m$. 
 An application of Corollary  \ref{c-intro-asdim} shows that we have in fact the following.  
\begin{cor}[See Corollary \ref{c-SFT-prod}]
For every $n, m\in \N$, the higher dimensional Brin-Thompson group $nV$ embeds into $mV$ if and only if $n\le m$. 
\end{cor}

\subsection{A rigidity property for  actions with subexponential  growth } \label{s-intro-fg}
We now leave temporary aside the above topological dynamical setting to discuss a rigidity property of groups of combinatorial nature, that we call property $\fg_{f(n)}$. It formalises the inability of a group to act on a set  with orbits growing uniformly  subexponentially (in cardinality).  This property has been considered by several authors,  but always as a consequence of property $(T)$ or of related analytic properties. We show that full groups provide a vast source of examples of groups with this property, but satisfy a strong negation of property $(T)$. 

  To give a formal definition, let $(\Gamma, d_\Gamma)$ be a discrete metric space. Its  \textbf{uniform growth}  function is the function $\overline{b}_\Gamma(n):=\sup_{v\in \Gamma}|B_\Gamma(v, n)|$, where $B_\Gamma(v, n)$ is the ball of radius $n$ around $v$. We assume that $\overline{b}_\Gamma(n) \le \infty $ (i.e. that $\Gamma$ has \textbf{bounded geometry}).
The  \textbf{wobbling group}  $W(\Gamma)$ is the group  of permutations $\sigma$ of  $\Gamma$  such that $\sup_{v} d_\Gamma(v, \sigma(v))<\infty$ \cite{Ce-Gri-Ha}. 

Given two functions $f, g\colon \N\to \N$ we write $f\preceq g$ if there exists a constant $C>0$ such that $f(n)\le Cg(Cn)$ for every $n$, and $f\sim g$ if $f\preceq g\preceq f$. 
\begin{defin} \label{d-fg}
 Let $f\colon \N\to \N$ be a function, with  $n\preceq f(n)\preceq \exp(n)$. A  group $G$ has property $\fg_{f(n)}$ if for every   discrete metric space $\Gamma$  such that   $\overline{b}_\Gamma(n)\nsucceq f(n)$,  every homomorphism $\rho\colon G\to W(\Gamma)$ has finite image.
 
 When $f(n)\sim \exp(n)$,  we will  omit the function $f(n)$ and write simply property $\fg$.\end{defin}
\begin{remarks} Let us comment on this definition.

  \begin{enumerate}[label=(\alph*),leftmargin=*]

 \item If $G$ is finitely generated, it has property $\fg_{f(n)}$ if and only if the Schreier graph $\Gamma=G/K$ of every infinite index subgroup  $K\le G$  satisfies $\overline{b}_\Gamma(n) \succeq f(n)$ (see  Proposition \ref{p-fg}). This can be taken as an alternative definition in this case. The  formulation in terms of wobbling groups has the two-fold advantage to make sense  for non-finitely generated groups, and  to highlight that  property $\fg_{f(n)}$ can be used in practice to study  homomorphisms $G\to H$ to other groups.  The reason is that whenever a finitely generated group $H$ is given by a  faithful transitive action on a set $Y$, it is a subgroup of $W(\Gamma)$ where $\Gamma$ is the  Schreier graph of the action  (with vertex set $Y$ and edges given by the action of generators). Actions  of groups of exponential growth with a subexponentially growing  Schreier graph are abundant in dynamical group theory (for instance among full groups and among groups acting on rooted trees, see \cite{Ba-Gr-Nek-fra, Ba-Gri-spec}). 
 
 \medskip 
 
 \item The notation $\fg$ stands for ``fixed point property for growth'' (this property can also be seen as a fixed point property, see Proposition \ref{p-fg})
  
  \medskip
   \item  It is natural to require a control on the uniform growth $\overline{b}_\Gamma(n)$ rather than  the  growth of balls $|B_\Gamma(n, v)|$ with respect to a fixed vertex  $v\in \Gamma$. 
       For instance, every residually finite group $G$ can be embedded in $W(\Gamma)$ for a connected graph such that $|B_\Gamma(n, v)| \sim n$ for every $v\in \Gamma$ (take $\Gamma$ to be the disjoint union of the Cayley graphs of all finite quotients of $G$, and make it connected by adding very long segments). But many residually finite groups have property $\fg$ as a consequence of property $(T)$ (see below).
       \end{enumerate}
  
  \end{remarks}


Property $\fg$ is a consequence of Kazhdan's property $(T)$. This fact is a variant of the  well-known  observation  that a group with property $(T)$ does not admit an infinite Schreier graph with subexponential growth (which is attributed to Kazhdan  in \cite[Remark 0.5.F]{Gro}).   With the formulation in terms of wobbling groups, an explicit proof of the implication $(T)\Rightarrow \fg$ is  given  in \cite{Ju-Sa-Wobbling} \cite[Theorem 4.1]{Ju-Sa-Wobbling}. This implication has also been used  to observe  the non-existence of infinite property  $(T)$ subgroups in certain groups, such as the group $\iet$ of interval exchanges \cite{D-F-G}.

The implication $(T)\Rightarrow \fg$  follows from a much stronger consequence of property $(T)$, namely the fact that if  $G$ has property $(T)$, then every $G$-action on a set which preserves an \emph{invariant mean} (i.e.~a finitely additive probability measure) must have a finite orbit. On its own right, this fixed point property is called \textbf{property  $\fm$}. It was formalised  and  studied by Cornulier  in  \cite{Cor-FM}, where it is observed explicitly that $\fm\Rightarrow \fg$ \cite[Th. 7.1]{Cor-FM}.  Property $\fm$ also appears earlier (implicitly) in \cite{Gla-Mo} (see also \cite{Dou, Gri-Nek-amen} for related results on the existence of actions with an invariant mean). 

Summing up, we have implications $(T)\Rightarrow \fm\Rightarrow \fg$. There are countable groups that are known to have property $\fm$ but  not property $(T)$, although such examples remain quite rare  (see \cite{Gla-Mo, Cor-FM}).  These can be obtained in two ways: groups constructed from property $(T)$ groups through \emph{ad-hoc} operations that preserve $\fm$ but not $(T)$, and certain groups constructed using (versions of) the small cancellation theory, such Olshanskii's   \emph{Tarski monsters} \cite{Ol-Tarski1, Ol-Tarski2} (it is still not know whether  all Tarski monsters have property $(T)$). 

On the other hand there is  no finitely generated group which is known to have property $\fg$ but not property $\fm$. A tightly related question of existence of such a group is asked by Cornulier \cite{Cor-FM} namely whether there exists a group without $\fm$ and with no   infinite Schreier graphs of subexponential growth.   More generally, most available tools to ensure that \emph{all} infinite Schreier graphs of a given group are ``large''  also seem to imply that they are (uniformly) non-amenable,  yielding property $\fm$. 
 As a consequence, even less is known about property $\fg_{f(n)}$ for functions $f(n)$ which grow subexponentially, apart from the trivial observation that every infinite group  has property $\fg_{f(n)}$ with $f(n)=n$ (and that  this is sharp for many groups).

\medskip 


The next result  is motivated by this discussion.
  Let $\G$ be a minimal  groupoid of germs over a Cantor space $X$. Assume that $\G$ is {compactly generated} by the set  $\T\subset \tG$. For every point $x\in X$, the \textbf{orbital graph} $\Gamma_x(\G, \T)$ is defined as the graph whose vertex set is the orbit of $x$, and for which edges are of the form $(y, T(y))$ for $T\in \T$.  The function $\beta_\G(n)=\overline{b}_{\Gamma_x(\G, \T)}(n)$ is called the \textbf{orbital growth} of $\G$. Its growth type is independent on $\T$ and on the point $x$ (see Lemma \ref{l-growth-graph}). Since   every orbit is dense, the action of $\Ff(\G)$ on each orbit is faithful, and realises $\Ff(\G)$ as a subgroup of $W(\Gamma_x(\G, \T))$. In particular $\Ff(\G)$ is a subgroup of $W(\Gamma)$ for a graph $\Gamma$ such that $\overline{b}_\Gamma(n)=\beta_\G(n)$. This turns out to be sharp:
    
    \begin{thm}[Property $\fg_{f(n)}$ for full groups]\label{t-intro-fg} \label{t-intro-wobbling}
Let $\G$ be a compactly generated minimal groupoid of germs $\G$ over a Cantor space. Then the group $\Af(\G)$ has property $\fg_{\beta_\G(n)}$ (and the function $\beta_\G$ is sharp).
\end{thm}
This produces a source of groups with property $\fg_{f(n)}$ for a vast class of  functions $f$. Moreover, we use it to prove  the following.

\begin{cor}[See Corollary \ref{c-fg-fm}]
There exist finitely generated groups with property $\fg$ but not property $\fm$. 
\end{cor}

Formally, Theorem \ref{t-intro-fg}  can be deduced from the Extension Theorem \ref{t-intro-main} (see Remark \ref{r-wobbling-full}). However, we will instead deduce it more directly from an intermediate and simpler result, namely the classification   in  \S \ref{s-confined} of the  \emph{confined subgroups} of the group $\Af(\G)$. 
\subsection{Quantitative invariants of groupoids as obstructions  to  homomorphisms}
A tantalising problem in the theory of  full groups consists in understanding how various quantitative invariants of  pseudogroups and  groupoids  are reflected in the behaviour of their full groups. Few general results in this direction are known.
The Extension Theorem implies that various explicit quantitative invariants  associated to two  groupoids $\G_1, \G_2$ produce obstructions to the existence of embeddings between the full groups. This  confirms the  natural intuition that a slow growth of this invariants  gives rise to more ``constrained'' full groups in the sense of group embeddings.

For instance,  Theorem \ref{t-intro-wobbling} readily implies  the following for the orbital  growth function.

\begin{cor}
For $i=1, 2$ let $\G_i$ be a compactly generated minimal groupoid of germs over a Cantor space. Assume that their orbital growth functions satisfy  $\beta_{\G_1}(n)\npreceq \beta_{\G_2}(n)$.  Then there is no non-trivial homomorphism $\Af(\G_1)\to  \Ff(\G_2)$.
\end{cor}


The orbital growth function  captures only part of the nature of a groupoid, namely the \emph{geometry} of its orbits. For this reason, in  \S \ref{s-complexity} we consider another invariant of \emph{dynamical} nature which is well-studied in symbolic dynamics, namely the \emph{complexity function}.  For simplicity we discuss here only the case of groupoid of germs of action of the group $\Z^d$ (see \S \ref{s-complexity} for generalisation to arbitrary groupoids).
Let $X\subset A^{\Z^d}$ be a subshift over a finite alphabet, i.e. a closed, translation invariant subset of $A^{\Z^d}$, where $A$ is a finite alphabet. Its \textbf{complexity function} is the function $p_X(n)$ that counts the number of functions $f \colon B_{\Z^d}(0, n)\to A$ which appear as restrictions of elements of $X$. Here $B_{\Z^d}(0, n)$ denotes the ball of radius $n$ in $\Z^d$ with respect to the standard basis. The complexity of a $\Z^d$-subshift is tightly related to its {topological entropy}, which is given by $h_{top}(\Z^d, X)=\lim_{n\to \infty} \frac{1}{|B_{\Z^d}(0, n)|} \log p_X(n)$. 

 It is natural to believe that $\Z^d$-subshift with slow complexity should produce a full group $\Ff(\Z^d, X)$ which is much more ``constrained'' compared to the full group of  a subshift with fast growing complexity (e.g. positive entropy). Some partial results in this direction can be found in \cite{MB-Liouville, Nek-frag} (see also \cite{Nek-complexity} for related results on groupoid algebras), and some evidence comes from the results in \cite{ExtAmen, E-M}. However in general it is not well-understood how the complexity function and the topological entropy are related to the behaviour of the group $\Ff(\Z^d, X)$, even for $d=1$. A special case of our results is the following.

\begin{thm}[See Corollary \ref{c-complexity}] \label{c-intro-complexity}
Let $\Z^d\acts X, \Z^\ell\acts Y$ be minimal subshifts, and assume that   the complexity functions satisfy $p_{X}\npreceq p_{Y}$. Then every homomorphism $\Ff(\Z^d, X)\to \Ff(\Z^\ell, Y)$ has abelian image.  \end{thm}

To put this result into context, we observe that already for $d=\ell=1$ there was no previous result highlighting  a general qualitative asymmetry between the full groups of subshfits with positive topological entropy and those with vanishing topological entropy.

As a concrete example, let us consider again the group $\iet$ of interval exchange transformations  (see \S \ref{s-intro-concrete}).  The actions  on the Cantor set  $\Lambda \acts X_\Lambda$ arising from groups of rotations of the circle have low (polynomial) complexity. This  ``lack of room''  makes it more difficult to construct embeddings of given groups in $\iet$ (compare with \cite{E-M}), but has never been used to prove obstructions on its possible subgroups. Theorem \ref{c-intro-complexity} has the following corollary. 
\begin{cor}
Let $\Z^d\acts X$ be a minimal subshift whose complexity function $p_X$ is not bounded above by a polynomial (e.g. if $h_{top}(\Z^d, X)>0$).  Then every homomorphism $\Ff(\Z^d, X)\to \iet$ has abelian image. 
\end{cor}

\subsection{Confined subgroups and uniformly recurrent subgroups} \label{s-intro-confined} An important ingredient in the proof of the Extension Theorem is the classification   in  \S \ref{s-confined} of the  all subgroups of the  group $\Af(\G)$ satisfying a weak notion of normality, called  \textbf{confined subgroups}.
A subgroup $H\le G$ of a  group $G$ is said to be \textbf{confined} if there exists a finite subset $P\subset G\setminus \{1\}$ which intersects non-trivially all conjugates of $H$. Equivalently, if closure of the conjugacy class of $H$ in the  space of subgroups $\sub(G)$, endowed with the Chabauty topology,  does not contain the trivial subgroup $\{1\}$.   We show that the group $\Af(\G)$ has  very few confined subgroups: roughly speaking, the only ones are  the stabilisers of finite subsets of $X$ (see Theorem \ref{t-confined}). 

The notion of confined subgroup is related to the notion of \emph{uniformly recurrent subgroup} of Glasner and Weiss \cite{Glas-Wei}. Recall that a \textbf{uniformly recurrent subgroup} (URS) of a group $G$ is a closed, minimal conjugation  invariant subset of  $\sub(G)$. A non-trivial URS consists of confined subgroups (but a confined subgroup need not belong to a URS). From the classification of the confined subgroups of the group $\Af(\G)$ we also deduce that is has a unique non-trivial URS, arising from its action on $X$ (see Corollary \ref{c-URS}).

The first uniqueness results on URS's of groups of homeomorphisms were  obtained by  Le Boudec and the author \cite{LB-MB-subdyn}, motivated by applications to the $C^*$-simplicity. We showed in particular that the Thompson's groups $T$ and $V$  admit a unique non-trivial URS. Using this it was also shown there \cite{LB-MB-subdyn} that this implies a rigidity result for minimal actions: if the group $V$ acts minimally on a compact space $Y$ and every $y\in Y$ has a non-trivial stabiliser, then there exists a continuous equivariant map $q\colon Y\to X$.  (An analogous statement was shown for the Thompson group $T$).  For the group $V$, this is in fact a consequence  of the Extension Theorem  in the special case of minimal actions (see Corollary \ref{c-action-minimal}).  This paper  arose from the desire of understanding better this phenomenon.

Confined subgroups are named after Hartley and Zalesskii \cite{Har-Zal},  who observed the relevance of this notion  in the setting of simple locally finite groups, giving rise to an extensive literature in this special case \cite{Har-Zal, Lei-Pug-diag, Lei-Pug-finlin} (the  interpretation in terms of Chabauty topology was  pointed out recently by Thomas \cite{Thomas}). When $\G$ is an \emph{AF-groupoid} (see \S \ref{s-AF}),  our results provide a classification of the confined subgroup of a well-studied class of simple locally finite groups, namely  \emph{block-diagonal limits} of products of finite alternating groups (or \emph{LDA-groups}) \cite{Lav-Nek-LDA, Lei-Pu-LDA}, see \S \ref{s-AF}. Our classification of the confined subgroups of $\Af(\G)$ is new even in this case, and extends   to all LDA groups a previous classification of Leinen-Puglisi for a subclass of the LDA groups  \cite{Lei-Pug-diag}. In turn, in the case of the  LDA groups our classification of  URS's of the group $\Af(\G)$ recovers a recent result of Thomas \cite{Thomas}.

\subsection*{Some  related work} Very satisfactory reconstruction results for homomorphisms have been obtained in the setting of diffeomorphisms and homeomorphisms groups of manifolds by
Hurtado \cite{Hurt-diff}  and Mann \cite{Mann-hom}  (previous results in low dimension were obtained by   Mann \cite{Mann-diff} and Militon \cite{Mil-hom1, Mil-hom2}).  While the motivation has a similar flavour,  this setting quite different  from the one of this paper both  in terms of the formulation of the results and of the methods of   proof. In particular in \cite{Hurt-diff, Mann-hom} the main results are stated as an automatic continuity property of discrete homomorphisms with respect to  the natural topologies on homeomorphisms and diffeomorphisms groups, while our primary focus here is on countable groups. 


The formulation of Theorem \ref{t-intro-main} and the sprit in which it is applied  are reminiscent  of a theorem of Nekrashevych  \cite{Nek-free}, stating that if a free group acts faithfully on a rooted tree, there is a point $x$ in the boundary of the tree such that every group element is uniquely determined by its germ at $x$ (although there is no substantial connection  at the level of proofs).

In his recent paper \cite{LB-urs-lattice}, Adrien Le Boudec illustrates yet another use of  Chabauty methods, by proving results that relate uniformly recurrent subgroups of a discrete group with its lattice envelopes.

After the first version of this paper has appeared, Chaudkhari \cite{Cha-F} obtained  a classification of the confined subgroups of Thompson's group $F$ similar to our Theorem \ref{t-confined}, and deduced that $F$ cannot be embedded in the wobbling group of a graph of uniformly subexponential growth. 


\subsection*{Organisation of the paper}
\S \ref{s-preliminaries} contains  preliminaries on the terminology on group action and coarse geometry.

In \S \ref{s-prel-Chabauty} contains preliminaries on the Chabauty space and confined subgroups. We recall a method from  \cite{LB-MB-subdyn} to study confined subgroups of groups of homeomorphisms, and prove a refinement of one of the results there which will be used later.

\S \ref{s-morphism-pseudogroups} is devoted to the theory of pseudogroups and \'etale groupoids, in particular to their morphisms. It contains both  preliminaries   and some  new results that  may be of independent interest. Its main new contribution  consists in  clarifying the  functoriality of the correspondence between pseudogroups and groupoids, by setting  up a  notion of morphism between \'etale groupoids that turns this correspondence into an equivalence of categories (Theorem \ref{t-non-comm-stone}). This functorial equivalence is used throughout the paper, and  appears  to be different and more general than the previous results in this direction that we were aware of. 

\S \ref{s-prel-alt} collects preliminaries on the alternating full groups, following \cite{Nek-simple}, and collects some technical facts about them.

In \S \ref{s-confined} we establish our characterisation  of the confined subgroups of full groups. This will be a key tool used throughout the paper. 

In \S \ref{s-Extension Theorem} we state and prove the main Extension Theorem and discuss its first corollaries. 

In \S \ref{s-fg} we define and study property $\fg$ and prove Theorem \ref{t-intro-fg}. 

In \S \ref{s-sec-complexity} we prove the application of our results related to the complexity function of subshifts and of \'etale groupoids. 

In \S \ref{s-concrete}-\ref{s-SFT} we use our results to study homomorphisms between three classes of groups: full groups of Cantor minimal systems,  groups of interval exchanges, full groups of one sided SFT's.

\subsection*{Acknowledgements}
I am grateful to Alessandro Sisto for pointing out a property of asymptotic dimension (Proposition \ref{p-asdim}) which  allowed me to improve the statement of Corollary \ref{c-intro-asdim}. This project  started  as a natural continuation of my collaboration with Adrien Le Boudec \cite{LB-MB-subdyn}, and I thank him for many conversations on the topic of the Chabauty space.  I am grateful to  Volodia Nekrashevych for many conversations on the theory of \'etale groupoids from which this paper has benefited. I thank  Laurent Bartholdi, Matt Brin, Yves Cornulier, Adrien Le Boudec, Damien Gaboriau, Rostislav Grigorchuk, Thierry Giordano, Daniel Lenz, Kostya Medynets,  Volodia Nekrashevych for  remarks on preliminary versions of this paper or on the talks in which  these results have been presented, which allowed me to improve the exposition and the formulation of the results.

\section{General preliminaries}\label{s-preliminaries}

\subsection{Notations on stabilisers and graphs of group actions}\label{s:notations-stabilisers} \label{s-notations} 
Throughout the paper, we use the following notations.	If $X$ is a set and $A\subset X$ is a subset, we denote $A^c$ its complement.

Assume that $G$  is a  group acting on $X$.
	\begin{itemize}
	\item   the stabiliser of a point $x\in X$ will be denoted $\st_G(x)$, and the set-wise stabiliser of $A$ will be denoted $\st_G(A)$,
	\item the point-wise stabiliser of $A\subset X$ will be denoted $\fix_G(A)$,
	\item we  call the subgroup $\fix_G(A^c)$  the \textbf{rigid stabiliser} of $A$ and denote it  $\rist_G(A)$ (this terminology is well-established in the setting of Branch groups, see \cite{Gri-branch}).

	\item  If, moreover, $X$ has a topology we  denote $\st^0_G(x)$ and $\st^0_G(A)$   the subgroups consisting of elements that fix point-wise a neighbourhood of $x$ (respectively, $A$), and call it the \textbf{germ-stabiliser} of $x$   (respectively, $A$). Thus $ \st^0_G(A)\le \fix_G(A)\le \st_G(A)$.
	\end{itemize}
Note that $\st^0(x)\unlhd \st_G(x)$. Assume that $X$ is a compact space and that $G$ is countable. Then a simple Baire argument shows that the set of points $x\in X$ such that $\st^0(x)=\st_G(x)$ is a dense $G_\delta$-subset of $X$. A point $x\in X$ satisfying this condition will be said to be a \textbf{regular point}.

Let $G$ be a finitely generated group with finite symmetric generating set $S$. Assume that $G\acts X$ is an action of $G$ on a set. The \textbf{graph of the action} is the graph $\Gamma_X(G, S)$ with vertex set $X$ and edges of the form $(sx, s)$ for every $s\in S$, where every edge is labelled by the corresponding generator. For every point  $x\in X$, the \textbf{orbital graph} $\Gamma_x(G, S)$ is the connected component of  $x$ and has $x$ a distinguished vertex.   It coincides with the \textbf{Schreier graph} of the stabiliser $\st_G(x)$, where the {Schreier graph} of a subgroup $H\le G$ is the orbital graph  $\sch_H(G, S)$ for the action of $G$ on the coset space $G/H$.  If $X$ has a topology, the Schreier graph of the germ stabiliser $\st_G^0(x)$ is  called the \textbf{graph of germs} at $x$ and will be denoted $\widetilde{\Gamma}_x(G, S)$. We omit $G, S$ when they are clear from the context.

 \subsection{Coarse geometry and asymptotic dimension} \label{s-coarse}
When dealing with \'etale groupoids it will be  convenient to use the language of abstract coarse geometry. We refer to Roe's book \cite{Roe-coarse} for a preliminaries.  Let us recall the basic definitions.

\begin{defin} A \textbf{coarse space} is a space $\Gamma$ endowed with a collection  of subsets of $\Gamma\times \Gamma$, called the \textbf{controlled sets}, which contains the diagonal, is closed under taking subsets and finite unions, and has the following properties:
\begin{enumerate}
\item If $E\subset \Gamma \times \Gamma$ is controlled, then so is $E^{-1}:=\{(y,x)\colon (x, y)\in E\}$;
\item if $E, F\subset \Gamma \times \Gamma$ are controlled, then so is $E\circ F:=\{(x, z)\colon (x,y)\in E, (y, z)\in F$.
\end{enumerate}

\end{defin}

\begin{example} The two main examples are the following.

 \begin{enumerate}[label=(\roman*)]

\item \label{i-coarse-metric} Every metric space $(\Gamma, d_\Gamma)$ is a coarse space whose controlled sets are the sets $E\subset \Gamma \times \Gamma$ such that $\sup_{(x, y)\in E} d_\Gamma(x, y)<\infty$. 

\item \label{i-coarse-discrete} Every  group $G$ (not necessarily finitely generated) is a coarse space whose controlled sets are all subsets $E\subset G\times G$ with the property that $|\{gh^{-1}\colon (g, h)\in E\}|<\infty$. When $G$ is finitely generated, this coarse structure coincides with the one induced by any word metric on it.

\end{enumerate}
\end{example}

\begin{remark}
The reader can have in mind these two examples when reading the following definitions. In fact,  all coarse spaces that we will encounter in this paper will  be  graphs (as metric spaces), groups, or the natural generalisation of \ref{i-coarse-discrete} for {\'etale groupoids} (which will be clarified later).  In the setting of  groupoids this language allows to avoid the constant use of some rather heavy notations, which is why we employ it. 
\end{remark}


A subset $B\subset \Gamma$  of a coarse space is said to be \textbf{bounded} if $B\times B$ is controlled. A family $\{B_i, i\in I\}$ of subsets $B_i\subset \Gamma$ is said to be \textbf{uniformly bounded} if $\bigcup_{i\in I} B_i\times B_i$ is controlled. 
A coarse space is said to be \textbf{discrete} if all bounded subsets are finite, and \textbf{uniformly discrete} if  every family   of uniformly bounded subsets has uniformly bounded cardinality. Discrete groups and bounded degree graphs are uniformly discrete, and all coarse spaces that we will consider in this paper will be uniformly discrete.

Let $\Gamma_1, \Gamma_2$ be coarse spaces. A map  $f\colon \Gamma_1\to \Gamma_2$ is said to be \textbf{bornologous} if for every controlled set $E\colon \Gamma_1\times \Gamma_1$, the set $f_*(E):=\{(f(x), f(y))\colon (x, y)\in E\}$ is controlled. When $\Gamma_1, \Gamma_2$ are graphs endowed with the bounded coarse structure (e.g. Cayley graphs of finitely generated groups), this is equivalent to the map $f$ being Lipschitz. 
A map $f\colon \Gamma_1\to \Gamma_2$ is said to be \textbf{proper} if preimages of bounded sets are bounded, and it is said to be \textbf{a coarse  map} if it is proper and bornologous. Two maps $f, g\colon \Gamma_1\to \Gamma_2$ are said to be close if $\{(f(x), g(x))\colon x \in \Gamma_1\}$ is a controlled set in $\Gamma_2$. 

 Two coarse spaces $\Gamma_1, \Gamma_2$ are \textbf{coarsely equivalent} if there exists coarse maps $f\colon \Gamma_1\to \Gamma_2$ and $g\colon \Gamma_2\to \Gamma_1$ such that $f\circ g$ and $g\circ f$ are close to the identity (the maps $f, g$ are then called \textbf{coarse equivalences}). For graphs (and for finitely generated groups), the notion of coarse equivalence coincides with quasi-isometry.

Every subset $\Gamma_0\subset \Gamma$ is naturally a coarse space whose controlled sets are $E\cap (\Gamma_0\times \Gamma_0)$ for $E\subset \Gamma \times \Gamma$ controlled. 
A map $f\colon \Gamma_1\to \Gamma_2$  is said to be a \textbf{coarse embedding} if it  gives rise to a coarse equivalence with $f(\Gamma_1) \subset \Gamma_2$, endowed with the coarse structure  induced from $\Gamma_2$.  Being a coarse embedding is stronger than being an injective coarse map.

We recall the definition of \textbf{asymptotic dimension} $\asdim(\Gamma)$ of a coarse space $\Gamma$. Given a controlled subset $E\subset \Gamma \times \Gamma$, a family  $\mathcal{D}$ of subsets of $\Gamma$ is said to be $E$-\textbf{separated} if for every $A, B\in \mathcal{D}$  we have $\left(A\times B\right)\cap E=\varnothing$. 

\begin{defin}

A coarse space $\Gamma$ satisfies $\asdim(\Gamma)\le n$ if for every controlled set $E$, one can find $E$-separated families  $\mathcal{D}_0,\ldots, \mathcal{D}_n$ of uniformly bounded subsets of $\Gamma$ such that $\mathcal{D}_0\cup \cdots \cup \mathcal{D}_n$ is a cover of  $\Gamma$. The \textbf{asymptotic dimension} $\asdim(\Gamma)$ is defined as the smallest $n$ such that $\asdim(\Gamma)\leq n$ if such an $n$ exists, and to be $\infty$ otherwise.
\end{defin}

Asymptotic dimension is invariant under coarse equivalence and behaves monotonically under coarse \emph{embeddings}. For {uniformly discrete} spaces a  bit more is true: namely it behaves monotonically under  \emph{injective coarse maps} (not necessarily coarse embeddings): \begin{prop} \label{p-asdim} \label{p-sisto}
Let $\Gamma_1, \Gamma_2$ be uniformly discrete coarse spaces, and assume that there exists an injective coarse map $f\colon \Gamma_1\to \Gamma_2$. Then $\asdim(\Gamma_1)\le \asdim(\Gamma_2)$. 
\end{prop}
This fact was pointed out to the author by A. Sisto, and will be used later. Since we could not locate this statement in the literature, we include a proof for completeness (see also a tightly related observation in \cite[\S 6]{Be-Sc-Ti-asdim}). 

\begin{proof}
If $\asdim(\Gamma_2)=\infty$, there is nothing to prove. Therefore we assume that $\asdim(\Gamma_2)=n<\infty$. Let $E\subset \Gamma_1\times \Gamma_1$ be a controlled set, and upon replacing $E$ with $E\cup E^{-1}$ we assume $E=E^{-1}$.  Let $F=f_*(E)$ and $\mathcal{D}_0,\cdots, \mathcal{D}_n$ be $F$-separated families of uniformly bounded subsets of $\Gamma_2$. By uniform discreteness, there exists $N\ge 1$ such that for every $i$ and $A\in \mathcal{D}_i$ we have $|A|\le N$. The families $f^{-1}(\mathcal{D}_i)=\{f^{-1}(A) \colon A\in \mathcal{D}_i\}$ are $E$-separated. Let $\mathcal{C}_i$ be the refinement of $f^{-1}(\mathcal{D}_i)$ obtained by partitioning every set $f^{-1}(A)$ with its $E$-coarsely connected components. The $E$-coarsely connected component of a subset $B\subset \Gamma_1$  are defined as follows: two points $v, w\in B$ are in the same $E$-coarsely connected component of $B$ if there exists points $w=x_0, x_i,\cdots, x_k=v$ of $B$ such that $(x_i, x_{i+1})\in E$ for $i=0,\ldots, k-1$.   This choice guarantees that the families $\mathcal{C}_i$ remain $E$-separated. By injectivity of the map $f$, every $B\in \mathcal{C}_i$ contains at most $N$-points. Since $B$ is $E$-coarsely connected, this implies that $\bigcup_{B\in \mathcal{C}_i} B\times B\subset \bigcup_{j=0}^N E^{\circ j}$. Thus, the families $\mathcal{C}_i$ consist of uniformly bounded subsets. This shows that $\asdim(\Gamma_1)\le n=\asdim(\Gamma_2)$. \qedhere \end{proof}

\section{The Chabauty space and confined subgroups}

\label{s-prel-Chabauty}
		
	\subsection{Definitions} Let $G$ be a discrete group. The set of subgroups $\sub(G)$ is a compact space endowed with the topology induced by the product topology on the set of all subsets $\{0,1\}^G$ of $G$, which is called the \textbf{Chabauty topology}. The conjugation action of $G$ on $\sub(G)$ is by homeomorphisms. 

Given a finite subset $P\subset G\setminus\{1\}$, we denote
\[\mathcal{U}_P=\{H\in \sub(G)\colon H\cap P=\varnothing\},\quad \mathcal{V}_P=\{H\in \sub(G)\colon P\subset H\}.\]
These sets are open in $\sub(G)$, and form a pre-basis of the topology as $P$ varies over finite subsets of $G\setminus \{1\}$. In particular, the sets of the form $\mathcal{U}_P$ form a fundamental system of neighbourhoods of the trivial subgroup $\{1\}$.

If $G$ is finitely generated, an equivalent description of the Chabauty topology can be given in terms of the space of \textbf{marked graphs}. Recall that given an integer $d\geq 0$ and a finite set $S$, the space of oriented  graphs $(\Gamma, v)$ with a distinguished base-point $v\in \Gamma$,  degree bounded by $d$ and edges  labelled by $S$   is naturally a compact metrisable space, where a sequence $(\Gamma_n, v_n)$ converges to $(\Gamma, v)$ if for every $R>0$  the ball $B_{\Gamma_n}(v_n, R)$ is eventually isomorphic to $B_\Gamma(v, R)$ as a rooted labelled graph. For every finitely generated group $G$, the correspondence $H\mapsto  \sch_H(G, S)$ embeds homeomorphically $\sub(G)$ into the space of marked graphs with degree bounded by $2|S|$ and edges labelled by $S$. 

\begin{defin}
A subgroup $H\le G$ is said to be \textbf{confined} if the closure of $\{gHg^{-1}\colon g\in G\}$ in $\sub(G)$ does not contain $\{1\}$. More generally a group $H$ is said to be \textbf{confined by} a subgroup $A\le G$ if the closure of $\{aHa^{-1}\colon a\in A\}$ does not contain $\{1\}$.
\end{defin}
 This property can be thought of as a weak notion of normality. Note that being confined is equivalent to the existence of a finite set $P\subset G\setminus \{1\}$ such that the conjugacy class of $H$ avoids $\mathcal{U}_P$, i.e.  $P\cap gHg^{-1}\neq \varnothing$ for every $g\in G$. Such a set $P$ will be called a \textbf{confining} set.

Confined subgroups are named after Hartley and Zalesskii  \cite{Har-Zal}, who introduced an equivalent property for simple locally finite groups  \footnote{In the original definition, a subgroup of $H<G$ of a locally finite group is said to be confined if there exists a non-trivial finite \emph{subgroup} $P< G$ such that $gHg^{-1}\cap P \neq \{1\}$. Since they assume $G$ locally finite, this is clearly equivalent to the above. This connection was recently pointed out by Thomas, see \cite{Thomas}}.

Another concept that will play an important role in this paper is the lower and upper semicontinuity of various maps taking values in $\sub(G)$.	
Given a compact space $X$, a map $u\colon X \to \sub(G)$ is said to be \textbf{upper} (respectively \textbf{lower}) \textbf{semicontinuous} if for every net $(x_\nu)$ in $X$ converging to $x$, every cluster point $K$ of $u(x_\nu)$ in $\sub(G)$  verifies $K\le u(x)$ (respectively $u(x)\le K$). A characterisation of these properties is given by the following lemma, which readily follows from the definition of the Chabauty topology.
\begin{lemma}\label{l-semicontinuous}
Let $u\colon X \to \sub(G)$ be a map from a compact space to $\sub(G)$.  
\begin{enumerate}[label=(\roman*)]
\item The map $u$ is upper semicontinuous if and only if for every $g\in G$, the set $\{x\in X\colon g\in u(x)\}$ is closed.
\item The map $u$ is lower semicontinuous if and only if for every $g\in G$, the set $\{x\in X\colon g\in u(x)\}$ is open.
\end{enumerate}
In particular, $u$ is continuous if and only if both conditions hold.
\end{lemma}

Two basic example of upper and lower semicontinuous maps are the following. 

\begin{example}\label{e:stab-semicontinuous}\label{e-stab-semicontinuous}
Let $G$ act by homeomorphisms on a  compact space $X$. Then it follows from the lemma above that:
\begin{itemize}
\item the map $X\to \sub(G), x\mapsto \st_G(x)$ is upper semicontinuous;
\item the map $X\to \sub(G), x\mapsto \st^0_G(x)$ is lower semicontinuous.

\end{itemize} 
\end{example}
We will encounter other examples of semicontinuous maps later. 


 We record for later use the following fact.
\begin{lemma}\label{l-semicontinuous-closure}
Assume that $G$ acts by homeomorphisms on a compact space $X$, let and $u\colon X\to \sub(G)$ be a lower semicontinuous equivariant map such that $u(x)\neq \{1\}$ for all $x\in X$. Then the closure of the image of $u$ does not contain $\{1\}$. In particular  $u(x)$ is confined for every $x\in X$.
\end{lemma}	
\begin{proof}
Let $H$ be in the closure of the image of $u$, and let $(x_\nu)\subset X$ be a net such that $u(x_\nu)$ tends to $H$. After extracting a subnet, $x_\nu$ tends to a limit $x\in X$ and by lower semicontinuity we have $u(x)\le H$. It follows that $H\neq \{1\}$.  In particular, for every $x\in X$ the closure of $\{gu(x)g^{-1}\colon g\in G\}=\{u(gx)\colon g\in G\}$ does not contain $\{1\}$, and therefore $u(x)$ is confined.
\qedhere
\end{proof}	
	Recall also that a \textbf{uniformly recurrent subgroup}, or \textbf{URS}, is a closed minimal invariant subset $Z\subset\sub(G)$ \cite{Glas-Wei}. Whenever $G\acts X$ is a minimal action on a compact space, the closure of $\{\st_G(x) \colon x\in X\}$ contains a unique URS, called the \textbf{stabiliser URS} of the action, see \cite{Glas-Wei} (this notion will not play an essential role in this paper). 
%
%
%
%

\subsection{Confined subgroups via rigid stabilisers}
	\label{s-proximal}
	
		Let $G$ be a group acting faithfully by homeomorphisms on a Hausdorff space $X$.  In \cite[Sec. 3]{LB-MB-subdyn}, a method was developed to relate confined and uniformly recurrent subgroups of $G$ to the rigid stabilisers of open sets of such an action. In this subsection we prove a refinement of one of these results (Theorem \ref{t-proximal} below), that will be used later.  	
	
We use the commutator notation $[g,h]=ghg^{-1}h^{-1}$.	

	
We first recall a well-known lemma in the case of normal subgroups. \begin{lemma}\label{l-double-comm}
Let $G$ be a countable group acting faithfully by homeomorphisms on a Hausdorff space $X$. Let $A\le G$ and let $R\le G$ be a non-trivial subgroup normalised by $A$. Then there exists a non-empty open subset $U\subset X$ such that $R$ contains $[\rist_A(U), \rist_A(U)]$.
\end{lemma}	
\begin{proof} This is essentially \cite[Lemma 3.1]{Nek-fp}, apart for the minor difference that it is stated there for $A=G$. \end{proof}

%

The following proposition is a variant of \cite[Proposition 3.8]{LB-MB-subdyn}.
\begin{prop} \label{p:LB-MB-subdyn}
Let $G$ be a group of homeomorphisms of a Hausdorff space $X$, and $A\le G$ be a subgroup.  Let $H\in \sub(G)$ be a subgroup confined by $A$, with confining set $P=\{g_1,\ldots, g_r\}\subset G\setminus \{1\}$. Assume that $U_1,\ldots, U_r\subset X$ are open subsets such that $U_1,\ldots, U_r, g_1(U_1),\ldots, g_r(U_r)$ are pairwise disjoint. Then there exists  $\ell=1,\ldots , r$, a finite index subgroup $\Gamma< \rist_A(U_\ell)$ and a subgroup $K\le H$ such that the following hold:
\begin{enumerate}[label=(\roman*)]
\item the group $K$ leaves $U_\ell$ invariant;
\item \label{i:surjection} for every $\gamma \in \Gamma$, there exists $k\in K$ such that $k|_{U_\ell}=\gamma$;
\item \label{i:support} every $k\in K$ is supported in $U_1\cup\cdots \cup U_r \cup g_\ell^{-1}(U_1)\cup\cdots \cup g_\ell^{-1}(U_r)$. 
\end{enumerate}
\end{prop}
\begin{proof}  The proof is identical to the proof of \cite[Proposition 3.8]{LB-MB-subdyn}, and we only sketch the argument to address the different formulation of the statement. We will use a lemma of B.H. Neumann, stating that an infinite group cannot be written as the union of finitely many cosets of infinite index subgroups \cite{Neumann}. Let $L\cong \rist_A(U_1)\times \cdots \times \rist_A({U_r})$  be the subgroup of $A$ generated by the (pairwise commuting) rigid stabilisers $\rist_A(U_1),\ldots \rist_A({U_r})$.  Since every element of $G$ conjugates at least an element of $P$ inside $H$, we can write $L=\bigcup_{i=1}^r Y_i$ where $Y_i=\{\gamma\in L\colon \gamma g_i\gamma^{-1}\in H\}$. For every $i=1,\ldots, r$ the set $Y_i$ is contained in a coset of the subgroup $L_i=\langle\gamma\delta^{-1}\colon \gamma,\delta\in Y_i\rangle$, and it follows that there exists $\ell=1,\ldots, r$ such that $L_\ell$ has finite index in $L$. Let $\Gamma<\rist_A({U_\ell})$ be the image of $L_\ell$ under the natural projection of $L$ onto   $\rist_A({U_\ell})$, and observe that it has finite index and that it is generated by the restriction to $U_\ell$ of all elements of the form $\gamma \delta^{-1}$ where $ \gamma, \delta \in Y_\ell$. For $\gamma, \delta\in Y_\ell$ let $a_{\gamma, \delta}=(\gamma g_\ell^{-1}\gamma^{-1})(\delta g_\ell \delta^{-1})\in H$ and let $K\le H$ be the subgroup $K=\langle a_{\gamma, \delta} \colon \gamma, \delta\in Y_\ell\rangle$. Writing $a_{\gamma, \delta}=\gamma (g_\ell^{-1}\gamma^{-1}\delta g_\ell) \delta^{-1}$ we see that each element $a_{\gamma, \delta}$ for $\gamma, \delta\in Y_\ell$ is supported in $U_1\cup\ldots \cup U_r\cup g_\ell^{-1}(U_1)\cup\cdots \cup g_\ell^{-1}(U_r)$, leaves $U_\ell$ invariant, and coincides with $\gamma \delta^{-1}$ in restriction to $U_\ell$. It follows that the group $K$ generated by them verifies the desired conclusions. 
\end{proof}

Recall the the action of $G$ on $X$ is said to be \textbf{proximal} if for every pair of points $x, y\in X$, there exists a net $(g_i)$ of elements of $G$ such that $(g_ix)$ and $(g_iy)$ both converge to the same limit.

\begin{thm}\label{t-proximal}
Let $G$ be a countable group acting faithfully by homeomorphisms of a Hausdorff space $X$, and assume $A\le G$ is a subgroup whose action on $X$ is minimal and proximal. Let $H\in \sub(G)$ be confined by $A$. Then there exists a non-empty open subset $U\subset X$  such that $H$ contains $[\rist_A(U), \rist_A(U)]$.
\end{thm}	
\begin{remark}
This is an improvement of \cite[Theorem 3.10]{LB-MB-subdyn}, where the same result was established under the stronger assumption that the action of $G$ on $X$ is \textbf{extremely proximal}, i.e. for every proper closed subset $C\subset X$ there exists a net of elements $(g_i)$ of $G$ and a point $z\in X$ such that $g_iC$ converges to $z$. It would be interesting to know if the conclusion remains true without any assumption on the action.
\end{remark}
\begin{proof}
We can and we shall assume that for every non-empty open set $U\subset X$ the group $\rist_A(U)$ is infinite, otherwise the statement is obvious. In particular, we can assume that $X$ has no isolated points, as these have a trivial rigid stabiliser. 

Assume that $H$ is confined by $A$ with confining set $P=\{g_1,\ldots, g_r\}\subset G\setminus \{1\}$. Since every $g_i, i=1,\ldots r$ is a non-trivial homeomorphisms of  $X$, and $X$ has no isolated point, it follows that every $g_i$ moves infinitely many points of $X$. In particular, we can find $x_1,\ldots , x_r$ such that $x_1,\ldots, x_r, g_1(x_1),\ldots , g_r(x_r)$ are pairwise distinct points. Let us denote $\mathcal{Q}=\{x_1,\ldots , x_r\}\cup \left(\bigcup_{i=}^r g_i^{-1}(\{x_1,\ldots , x_r\}\right)\subset X$.

For every $i=1,\ldots r,$, choose a neighbourhood $U_i$ of $x_i$, small enough so that the sets $U_1,\ldots, U_r,  g_1(U_1),\ldots, g_r(U_r)$ are pairwise disjoint. We can therefore apply Proposition \ref{p:LB-MB-subdyn}.  Let $\ell\in \{1,\ldots, r\}$, $\Gamma< A_{U_\ell}$ and $K<H$ be given by the lemma. 

Using minimality and proximality of the action of $A$ on $X$, we can find $a\in A$ such that $a(\mathcal{Q})\subset U_\ell$. For every $i=1,\ldots, r$ let $U_i'\subset U_i$ be a smaller neighbourhood of $x_i$ to be  determined shortly. Set $W= U'_1\cup\cdots \cup U'_r \cup\left(\bigcup_{i=1}^r g_i^{-1}(U'_1\cup\cdots \cup U'_r)\right)$, and observe that it is a neighbourhood of $\mathcal{Q}$ that shrinks to $\mathcal{Q}$ when all the $U'_i$'s shrink to $x_i$. We choose the sets $U'_i$ small enough so that $a(W)\subset U_\ell$. We now apply again Proposition \ref{p:LB-MB-subdyn}, this time to the group $H'=a^{-1} H a$ (note that $H'$ is still confined by $A$ with confining set $P$)  and to the  open sets $U'_1,\ldots, U'_r$ (that still verify the disjointness assumption in the proposition, since $U'_i\subset U_i$). Let $\ell'$, $\Gamma'<\rist_A({U'_{\ell'}})$ and $K'<H'$ be given by the proposition. Since we are assuming that all rigid stabilisers are infinite, the group $\Gamma'$ is non-trivial, and therefore so is $K'$ by part \ref{i:surjection} of the proposition. By part \ref{i:support}, $K'$ is supported in $W$. It follows that $aK'a^{-1}$ is a non-trivial subgroup of $H$ supported in $U_\ell$.

From this point on, the proof proceeds exactly as the last part of the proof of \cite[Theorem 3.10]{LB-MB-subdyn} (with the minor difference that the proof is written there for $A=G$). We repeat the argument for the convenience of the reader. Let $b\in aK'a^{-1}$ be a non-trivial element. For every $\gamma\in \Gamma$, choose $k_\gamma\in K$ which coincides with $\gamma$ in restriction to $U_\ell$. Observe that, since the support of $b$ is contained in $U_\ell$ we have $k_\gamma b k_\gamma^{-1}=\gamma b \gamma^{-1}\in H$. It follows that the subgroup $R=\langle k_\gamma b k_\gamma^{-1}\colon \gamma \in \Gamma\rangle \le H$ is normalised by $\Gamma$.  By Lemma \ref{l-double-comm} there exists a non-empty open subset $V\subset U_\ell$  such that $[\rist_\Gamma(V), \rist_\Gamma(V)]\le R\le H$. Now $\rist_\Gamma(V)=\Gamma \cap \rist_A(V)$ has finite index in $\rist_A(V)$, and upon replacing it with a smaller finite index subgroup we can assume that it is normal in $\rist_A(V)$. Thus its derived subgroup is also normal in $\rist_A(V)$. Applying Lemma \ref{l-double-comm} again, we deduce that there exists a non-empty open subset $V'\subset V$ such that $[\rist_A(V'), \rist_A(V')]\le [\rist_\Gamma(V), \rist_\Gamma(V)]\le H$, as desired. \qedhere
\end{proof}

\section{Groupoids, pseudogroups, and their morphisms}
This section contains all the necessary ingredients from the theory of \'etale groupoids and of pseudogroups that will be used in the paper, in particular about morphisms between them. While a large part of it consists in recalling definitions and well-known facts,  its main  result is Theorem \ref{t-non-comm-stone}, which provides a categorical formulation of the equivalence between pseudogroups and \'etale groupoids.

\label{s-morphism-pseudogroups}
\subsection{Basic definitions} \label{s-pseudogroups}

In this subsection we recall all definitions and terminology about pseudogroups and groupoids. We will need to work in a slightly more general setting than the definitions given in the introduction. 

\subsubsection{\'Etale groupoids} \label{s-groupoids}

A \textbf{groupoid} $\G$ over a space $X$ (also called the  \textbf{unit space} of the groupoid) is the set of isomorphism of a small category whose underlying set of objects is $X$. Every $\gamma\in \G$  is a morphism between elements of $X$ that we denote respectively by $\src(\gamma),\rg(\gamma)\in X$.  The maps $\mathsf{r}, \src\colon \G \to X$  are called the \textbf{source} and the \textbf{range} map.  The product of two elements $\gamma, \delta\in \G$ is defined if and only if $\src(\gamma)=\rg(\delta)$ and in this case $\src(\gamma\delta)=\src(\delta), \rg(\gamma\delta)=\rg(\gamma)$. We denote by $\gamma^{-1}$ the inverse of $\gamma\in \G$. Note that we have $\src(\gamma^{-1})=\rg(\gamma)$ and $\rg(\gamma^{-1})=\src(\gamma)$. 

Following a common use we identify the set of objects $X$   with a subset of $\G$ by identifying every object $x\in X$ with the  identity isomorphism of $x$.  With this convention, the source and range map are given by  $\src(\gamma)=\gamma^{-1}\gamma$ and $\rg(\gamma)=\gamma\gamma^{-1}$. 

For every $x\in X$ we denote by  $\G_x$ and $\G^x$ the sets $\src^{-1}(x)$ and $\rg^{-1}(x)$. These subsets are called the \textbf{leafs} (or \textbf{fibres}) of the groupoid. For $x, y\in X$ we let $\G_x^y=\G_x\cap \G^y$. Note that the set $\G_x^x$ is naturally a group, called the \textbf{isotropy group} at $x$. The set $r(\G_x)=\{r(\gamma)\colon  \gamma\in \G_x\}$ is called the \textbf{orbit} of $x$. 

Given a subset $A\subset X$ of the unit space, the \textbf{restriction} of $\G$ to $A$ is the subgroupoid $\G|A=\{\gamma \in \G \colon \src(\gamma), \rg(\gamma)\in A\}$.

A \textbf{topological groupoid} is a groupoid endowed with a topology such that the source and range maps, the composition $\mathcal{G}^{*2}\to \mathcal{G}$ and the inversion $\mathcal{G}\to \mathcal{G}$ are continuous. Here $\G^{*2}=\{ (\gamma, \delta) \colon \src(\gamma)=\rg(\delta)\}$ is the set of \textbf{composable pairs},  endowed with the topology induced from the product topology on $\G \times \G$.
The  space $X$ is endowed with the topology induced by the inclusion $X\subset \G$. A topological groupoid $\G$ is said to  be \textbf{minimal} if every orbit is dense in $X$. It  is said to be \textbf{principal} if the isotropy group $\G^x_x$ is trivial for every $x\in X$. It is said to be \textbf{essentially principal} if the set of points with trivial isotropy group is dense in $X$.

An \textbf{\'etale groupoid} is a topological groupoid $\G$ such that the source and range maps are open and are local homeomorphisms. We do not require the topology on $\G$ to be Hausdorff (as many interesting examples are not). However, we do require the unit space $X$ to be Hausdorff.

A  \textbf{bisection} of an \'etale groupoid is an open subset $T\subset \mathcal{G}$ such that $\src|_T\colon T \to \src(T)$ and $\rg|_T\colon  T \to \rg(T)$ are homeomorphisms onto their image. By definition of an \'etale groupoid, bisections form a basis of the topology. 

\begin{examples} \label{e-groupoids}
\begin{enumerate}[label=(\roman*), wide, labelwidth=!]
\item \label{e-i-group} Every  group $G$ can be seen as an \'etale groupoid over the one-point  space $X=\{1_G\}$ and with the discrete topology. 
\medskip

\item \label{e-i-action}  Let $G\acts X$ be a countable group acting on a compact space. The associated \textbf{action groupoid} is $G\times X$. Its unit space is $X\simeq \{1_G\}\times X$ (with the obvious identification), and source and range map are given by $\src((g, x))=x$ and $\rg((g, x))=gx$. The product and inversion are defined by the rules $(g_1, g_2x)(g_2, x)=(g_1g_2, x)$ and $(g, x)^{-1}=(g^{-1}, gx)$. It is an \'etale groupoid if   $G\times X$ is endowed with the product topology, where $G$ has the discrete topology. A groupoid of this form is always Hausdorff.

\medskip

\item \label{e-i-germs} \label{i-germs} Let again $G\acts X$ be a countable group action on a compact space. For every $(g, x)\in G\times X$ denote by $[g]_x$ the \textbf{germ} of $g$ at $x$, i.e. the equivalence class of the pair $(g, x)$ under the equivalence relation that identifies $(g_1, x_1)$ with $(g_2, x_2)$ if $x_1=x_2$ and $g_1, g_2$ coincide in restriction to a neighbourhood of $x_1$ Then the set of germs is naturally a groupoid $\G=\G(G\acts X)$, called the \textbf{groupoid of germs} of the action, with  unit space $X$ identified with the set of germs  $\{[1_G]_x \colon x\in X\}$. Source, range, composition and inversion are given exactly as in the case of the action groupoid by replacing $(g, x)$ with $[g]_x$. The groupoid $\G$ has a natural topology for which  a basis of open sets is given by sets of the form $\mathcal{U}_{g, U}=\{[g]_x\colon x\in U\}$ where $g\in G$ and $U\subset X$ is open, and is \'etale with this topology. Groupoids of germs of group actions are often non-Hausdorff.

\end{enumerate}

\end{examples}

 \subsubsection{Pseudogroups} Let $S$ be a  semigroup, that is a set endowed with an associative binary operation. An \textbf{inverse} of an element $F\in S$ is an element $T$ such that $FTF=F$ and $TFT=T$. An \textbf{inverse semigroup} is a semigroup  in which every element $F$  has a unique inverse, denoted $F^{-1}$. Element of the form $F^{-1}F$ and $FF^{-1}$, for $F\in S$ are precisely the \textbf{idempotents} of $S$, i.e. those elements $U$ such that $U^2=U$.  In an inverse semigroup, any two idempotents commute. 
An inverse semigroup is said to be an \textbf{inverse monoid} if it is a monoid, i.e. if it admits a neutral element (which is necessarily unique). A \textbf{zero} in an inverse semigroup is an element $0\in S$ such that $0F=0$ for every $F\in S$. See the books of Lawson \cite{Law-book} and Paterson \cite{Pat-book} as general references on inverse semigroups .

As main example, let $X$ be a Hausdorff space, that we further assume to be locally compact for simplicity.  We denote by $\tI(X)$ the inverse monoid of all \textbf{partial homeomorphisms} of $X$, that is, homeomorphism $F\colon U\to V$ between open subsets of $X$.  If $F_i\colon U_i\to V_i$ for $i=1,2$ are elements of $\tI(X)$, their product is defined as the restricted composition
\[F_2 F_1 \colon F_1^{-1}(V_1\cap U_2)\to F_2(V_1\cap U_2),\]
and the inverse of $F\colon U\to V$ is just the inverse homeomorphism $F^{-1}\colon V\to U$. The idempotents of $\tI(X)$ are precisely the set of all identity homeomorphisms of open subsets of $X$.
Note that the empty homeomorphism (between empty subsets of $X$) belongs to $\tI(X)$, and is a zero in $\tI(X)$.  

 It is convenient to simply identify the set of idempotents of $\tI(X)$ with the set of open subsets of $U$, by writing just $U$ instead of $\operatorname{Id}_U$, and we shall systematically do so. Thus, for example, for any two idempotents $U, V\in \tI(X)$ the product $UV$ is the intersection $U\cap V$ of open subsets, for every  $F\colon U\to V\in \tI(X)$ we have $F^{-1}F=U$ and $FF^{-1}=V$, and for every $W\subset U$ the element $FW$ is the restriction of $F$ to $W$.

A family $\{F_i\colon U_i\to V_i, i\in I\}\subset \tI(X)$ is said to be \textbf{compatible} if $F_i$ and $F_j$ agree on $U_i\cap U_j$ and $F_i^{-1}\cap F_j^{-1}$ agree on $V_i\cap V_j$ for every $i, j\in I$. In this case, the natural homeomorphism $\bigcup_{i\in I} F_i \colon \bigcup_{i\in I} U_i \to \bigcup_{i\in I} V_i$ is called the \textbf{union} (or the \textbf{join}) of the compatible family.

The notions of compatible family and join can be made sense in an arbitrary inverse semigroup $S$ as follows. Two elements $F, T$ of an inverse semi-group are \textbf{compatible} if both $FT^{-1}$ and $F^{-1}T$ are idempotent. A family $\{F_i, i\in I\}$ of elements of an inverse semigroup is said to be compatible if its elements are pairwise compatible.  An inverse semigroup is naturally endowed with a partial order, denoted $\subset$, where $F\subset T$ if and only if $F=TU$ for some idempotent $U$. A family $\{F_i, i\in I\}$ of elements of an inverse semigroup is said to admit a \textbf{union} (or \textbf{join})  if it admits a least common upper bound for this partial order, denoted $\bigcup_{i\in I} F_i$. If a family admits a union, then it is necessarily compatible. An inverse semigroup is said to be \textbf{complete} if every compatible family admits a union.

\begin{defin}

A \textbf{pseudogroup} over a space $X$ is a complete inverse monoid $\tG$ together with a semigroup homomorphisms $ \tau \colon \tG\to  \tI(X)$   which preserves inverses,  joins of compatible families, and identifies the idempotents of $\tG$ bijectively with  the set of open subsets of $X$.
It is said to be \textbf{effective} if the representation $\tau\colon \tG\to \tI(X)$ is injective. 
 \end{defin}
 
 When $\tG$ is effective, it can be simply identified with its image of $\tI(X)$, and we will always do so. This recovers the definition sketched in the introduction.

 \begin{remark}
 Many papers put effectiveness as part of definition of pseudogroup.  In this paper, our primary interest will also be towards effective pseudogroups, and adding this assumption to most of our results would result in  minor loss of generality. However non-effective pseudogroups will  be useful in the course of the proofs, and appear naturally in the theory that we develop. Thus, we do not put the effectiveness as part of the definition.
 \end{remark}


\subsubsection{From groupoids to pseudogroups and vice versa} Let us recall the well-known equivalence between \'etale groupoid and  pseudogroups.

Let $\G$ be an \'etale groupoid,  and let $\tG$ be the set of all bisections of $\G$. The set $\tG$ is naturally an inverse monoid  with  product and inversion
\begin{equation} \label{e:bisections} TS=\{\gamma\delta\colon \gamma\in T, \delta \in S, \src(\gamma)=\rg(\delta)\}, \quad T^{-1}=\{\gamma^{-1}\colon \gamma \in T\},  \end{equation}
 an with identity given by $X$. Note that the order $\subset $ and the union operation $\bigcup$ in the inverse semi-group $\tG$ coincide with the inclusion and union of  bisections seen as subsets of $\G$. In addition  every  bisection $T$ defines a homeomorphism between the open subset  $\src(T)$ and $\rg(T)$, given by  
\[\tau(T):= \rg|_T \circ \src|_T^{-1}\colon \src(T)\longrightarrow \rg(T).\] 
This provides a representation $\tau\colon \tG\to \tI(X)$. Thus $\tG$ is a pseudogroup, called the pseudogroup associated to $\G$.
The groupoid $\G$ is said to be \textbf{effective}, or a \textbf{groupoid of germs},  if $\tG$ is an effective pseudogroup, i.e. if every bisection is uniquely determined by the associated homeomorphism $\tau(T)$. In this case we will often omit $\tau$, and denote the homeomorphism simply by $x\mapsto T(x)$. 
 For every bisection $T\in \tG$ and every point $x\in \src (T)$, let us define the (abstract) \textbf{germ} of $[T]_x$ to be the unique $\gamma \in T$ such that $s(\gamma)=x$. When $\G$ is effective, it identifies with the germ of the homeomorphism $\tau(T)$. 
Note that germs of bisections satisfy the cocycle relation
\[[FT]_x=[T]_{\tau(F)(x)}[F]_x, \]
for every $F, T\in \tG$ and $x\in \src(F)$ such that $\tau(F)(x)\in \src(T)$.

\medskip

Conversely, every pseudogroup over $X$ is isomorphic to $\tG$ for some \'etale groupoid $\G$.  Let us briefly recall why; for a more complete exposition with full details we refer to the lecture notes of Resende \cite{Resende}.

Let $\tP$ be an pseudogroup  over $X$, with associated representation $\tau\colon \tP\to \tI(X)$. For $F\in \tP$ and  $x\in U= F^{-1}F$ , define the (abstract) \textbf{germ} of $F$ at $x$, denoted $[F]_x$, to be the equivalence class of the pair $(F, x)$, where two pairs $(F_1, x_1)$ and $(F_2, x_2)$  are equivalent if $x_1=x_2$ and  there exists an open subset $W\subset U_1\cap U_2$ such that $F_1W=F_2W$, where we have set $U_i=F_i^{-1}F_i$. When  $\tP$ is effective,  this simply means that the homeomorphisms $F_1$ and $F_2$ coincide on a neighbourhood of $x$, and thus $[F]_x$ is the germ of $F$ at $x$ in the usual sense.  
The set $\G$ of all germs of $\tP$ is naturally a groupoid with source and range map $\src([F]_x)=x, \rg([F]_x)= \tau(F)(x)$, and product 
\[[F_1]_{x_1}[F_2]_{x_2}=[F_1F_2]_{x_2}\]
which is defined if and only if $x_1=\tau(F_2)(x_2)$. Define a topology on $\G$ for which a basis  of open sets of the form $\{[F]_x\colon x\in U\}$, for $F\in \tP$ and $U=F^{-1}F$. This topology turns $\G$ into an \'etale groupoid. Note that sets of this form are precisely the bisections of $\G$. We therefore have a natural identification $\tG \simeq \tP$.

\medskip

\begin{stand-assump} {From now on, all pseudogroups will be assumed to be of the form $\tG$ for some \'etale groupoid $\G$. } \end{stand-assump}

\subsubsection{The topological full group}
\begin{defin}
Let $\G$ be an \'etale groupoid over $X$. Its \textbf{topological full group} (or just \textbf{full group}) is the subgroup $\Ff(\G)$ of the pseudogroup $\tG$ consisting of all elements $g\in \tG$ such that $\src(g)=\rg(g)=X$. 

\end{defin}

The group $\Ff(\G)$ has a natural action on $X$, via the homeomorphism associated to bisections. We will denote by $\Df(\G)$ the derived subgroup of $\Ff(\G)$. 

When $G$ is the groupoid of germs of a group action $G\acts X$ (as in Example \ref{e-groupoids}~\ref{e-i-germs}), these groups will also be denoted by $\Ff(G, X)$ and $\Df(G, X)$.

\subsection{Morphisms of pseudogroups}\label{s-morphisms}
In this subsection we clarify the notion of morphism between pseudogroups  that will be extensively use in the sequel. We first recall a well-known elementary lemma.

\begin{lemma}[Stone duality for Hausdorff spaces] \label{l-stone}
\label{l-Stone}
Let $X, Y$ be Hausdorff spaces, and let $\varphi$ be a map from the set of all open subsets of $X$ to the set of open subsets of $Y$ which preserves arbitrary unions and finite intersections and such that $\varphi(\varnothing)=\varnothing$ and $\varphi(X)=Y$. 

Then there exists a unique continuous map $q\colon Y\to X$ such that $\varphi(U)=q^{-1}(U)$ for every open subset $U\subset X$. 
\end{lemma}

 Throughout the subsection, let $\G, \H$ be \'etale groupoids over spaces $X, Y$.
\begin{defin}
A \textbf{continuous morphism} of pseudogroups $\varphi\colon \tG\to \tH$ is a semigroup homomorphism which preserves unions of compatible families, and preserves zero and unit (i.e. $\varphi(\varnothing)=\varnothing, \varphi(X)=\varphi(Y)$). 
\end{defin}

A continuous morphism $\varphi\colon \tG\to \tH$ maps idempotents of $\tG$ to idempotents of $\tH$. Thus $\varphi$ provides a map between open subsets of $X$ to open subsets of $Y$ which verifies the conditions of Lemma \ref{l-stone}. We deduce that there exists a unique continuous map $q\colon Y\to X$, such that $\varphi(U)=q^{-1}(U)$ for every open subset $U\subset X$. 

\begin{defin}
The map $q\colon Y\to X$ described above will be called the \textbf{spatial component} of the continuous morphism $\varphi\colon \tG\to \tH$. 
\end{defin}

The spatial component is equivariant with respect to the natural representations  $\tau\colon \tG\to \tI(X)$ and $\tau\colon \tH\to \tI(Y)$, in the following sense:

\begin{prop}[Equivariance of the spatial component] \label{p-spatial-equivariance} \label{p-spatial-equivariant}
Let $\varphi\colon \tG\to \tH$ be a continuous morphism, and $q\colon Y\to X$ be its spatial component. For every $F\in \tG$ with $U=\src(F)$ and $V=\rg(F)$ we have $\src(\varphi(F))= q^{-1}(U)$ and $\rg(\varphi(F))=q^{-1}(V)$, and the homeomorphisms $\tau(F), \tau(\varphi(F))$ verify $\tau(F)\circ q=q\circ(\tau(\varphi(F))$.  \[\begin{tikzcd}
q^{-1}(U)\arrow[r, "\tau(\varphi(F))", "\widetilde{}"'] \arrow[d, "q"] & q^{-1}(V) \arrow[d, "q"] \\
U\arrow[r, "\tau(F)", "\widetilde{}"'] & V.
\end{tikzcd}\] 
\end{prop}
\begin{proof}
 Let $F\in \tG$ and $U=\src(F)=F^{-1}F$. Since $U$ is idempotent, we have that $q^{-1}(U)=\varphi(U)=\varphi(F)^{-1}\varphi(F)=\src(\varphi(F))$, as desired. The same conclusion follows for $V=\rg(F)$  in the same way. Let us show that the the diagram commutes, i.e. that we have $\tau(F)\circ q=q\circ \tau({\varphi(F)})$. To see this let $W\subset V$ be an open subset. Note that  $\tau(F)^{-1}(W)$ is equal to the idempotent $F^{-1}WF$. It follows that $(\tau(F)\circ q)^{-1} (W)=q^{-1}(\tau(F)^{-1}(W))=q^{-1}(F^{-1}WF)=\varphi(F^{-1})\varphi(W)\varphi(F)$.  Arguing similarly we have that $(q\circ \tau({\varphi(F)}))^{-1}(W)=\tau({\varphi(F)})^{-1}(q^{-1}(W))=\tau({\varphi(F)})^{-1}(\varphi(W))=\varphi(F)^{-1}\varphi(W)\varphi(F)$. We conclude that for every open subset $W$ of $V$ we have $(\tau(F)\circ q)^{-1} (W)= (q\circ \tau({\varphi(F)}))^{-1}(W)$. By the uniqueness part of Lemma \ref{l-Stone} this implies that the two maps $q\circ \tau({\varphi(F)})$ and $\tau(F)\circ q$ are equal, as desired. \qedhere

\end{proof}
A direct consequence of the proposition is the following.
\begin{cor} \label{c-continuous-morphism-full}
Let $\varphi\colon \tG\to \tH$ be a continuous morphism, and $q\colon Y\to X$ be its spatial component. Then its restriction to $\Ff(\G)$ is a group homomorphism taking values in $\Ff(\H)$, and the action of $\Ff(\G)$ on $Y$ induced by $\varphi$ factors onto its natural action on $X$ via the map $q$.  
\end{cor}

Injectivity of a continuous morphism $\varphi\colon \tG\to \tH$ is tightly related to the surjectivity of the spatial component $q\colon Y\to X$.

\begin{prop} \label{p-inj-surj}
Let $\varphi\colon \tG\to \tH$ be a continuous morphism. If $\varphi$ is injective, then its spatial component $q\colon Y\to X$ is surjective. The converse holds  if $\tG$ is effective.

\end{prop}
\begin{proof}
Assume that $\varphi$ is injective. Towards a contradiction, let $x\in X$ be a point which does not belong to the image of $q$, and let $U$ be a neighbourhood of $x$. We have $\varphi(U)=q^{-1}(U)=q^{-1}(U\setminus\{x\})=\varphi(U\setminus\{x\})$ contradicting injectivity of $\varphi$. 
Conversely assume that $\tG$ is effective and that $q$ is  surjective. 
Let $F, T\in \tG$ be such that $\varphi(F)=\varphi(T)$. Proposition \ref{p-spatial-equivariance} together with the surjectivity of the map $q$ implies that $\src(F)=\src(T), \rg(F)=\rg(T)$ and that the homeomorphisms $\tau(F)$ and $\tau(T)$ coincide. Thus $F=T$. \qedhere

\end{proof}

\begin{example}\label{e-restriction-morphism}
Assume that $C\subset X$ is a  $\G$-invariant subset. We then have a natural continuous morphism $r_C\colon \tG\to \widetilde{\G|C}$, called the \textbf{restriction morphism}, given by $F\mapsto F|_C:= F\cap (\G|C)$. Clearly $F|_C$ is a bisection of $\G|C$ and this defines a continuous morphism, whose associated spatial component is  the inclusion $C\hookrightarrow X$.
\end{example}

The following simple proposition allows in many cases to reduce the study of continuous morphism to those with surjective spatial component.
\begin{prop} \label{p-morphism-restriction}\label{p-restriction-morphism}
Let $\varphi\colon \tG\to \tH$ be a continuous morphism, with spatial component $q\colon Y\to X$. Then $q(Y)$ is a $\G$-invariant subset of $X$, and there exists a unique continuous morphism $\psi\colon \widetilde{\G|q(Y)}\to \tH$ such that $\varphi$ splits as the composition of the restriction morphism $\tG\to \widetilde{\G|q(Y)}$ followed by $\psi$.  
\end{prop} 

Although the proof is not difficult, we postpone it to the end of \S \ref{s-functor-equivalence}, as it will become tautological after introducing a suitable interpretation of continuous morphisms  in terms of the underlying groupoids. 

\subsection{Actions of groupoids} 
 \label{s-actions}\label{s-groupoid-actions}

Let $\G$ be an \'etale groupoid over a space $X$.   A ({continuous}) \textbf{left action}  of $ \G $ on $ Z$ over a continuous map $q\colon Z\to X$ is a  continuous map
\[\alpha\colon  \G\tensor[_s]{\times}{_q} Z\to Z, \quad (\gamma, z)\mapsto \gamma  \cdot z\]
defined on the {fibre product}  $ \G\tensor[_s]{\times}{_q} Z=\{(\gamma, z)\in \G\times Z\colon \src(\gamma)=q(z)\}$, 
which verifies  $q(\gamma \cdot z)=\rg(\gamma)$ and  $(\delta\gamma) \cdot z=\delta \cdot ( \gamma \cdot z)$ for every
$\delta, \gamma$ such that $\src(\delta)=\rg(\gamma)$ and $\src(\gamma)=q(z)$. We will   write $\alpha \colon \G\acts Z$ to indicate that $\alpha$ is a left action of $\G$ on $(Z, q)$ over some map $q\colon Z\to X$.

Similarly a  \textbf{right action} of $\alpha \colon Z\curvearrowleft \G$ over a continuous map $q\colon Z\to X$ is a map
\[\alpha\colon  Z\tensor[_q]{\times}{_r} \G\to Z, \quad (z, \gamma)\mapsto   z\cdot \gamma \]
such that $q(z\gamma)=\src(\gamma)$ and $(z\cdot \gamma)\cdot \delta=z\cdot (\gamma\delta)$ for every $\gamma, \delta\in \G$ with $\src(\gamma)=\rg(\delta)$.


Actions of $\G$ on $Z$ are essentially the same thing as continuous morphisms from $\tG$ to the pseudogroup $\tI(Z)$:
\begin{prop}\label{prop-morphism-actions}
Left actions of $\G$ on $Z$ are in natural one-to-one correspondence with continuous morphisms  $\varphi \colon \tG\to \tI(Z)$. Under this correspondence, a continuous morphism with spatial component $q\colon Z\to X$ gives rise to an action over the same map $q$, and vice versa. 
\end{prop}
\begin{proof}
Let $\alpha\colon  \G\tensor[_s]{\times}{_q} Z\to Z, \quad (\gamma, z)\mapsto \gamma z$ be an action over a map $q$ and let us construct a continuous  morphism $\varphi\colon \tG\to \tI(Z)$. On idempotents $U\in \tG$ we define $\varphi(U)=q^{-1}(U)$. For every $F\in \tG$ with $\src(F)=U$ and $\rg(F)=V$, define $\varphi(F)\in \tI(Z)$ to be the homeomorphism $\varphi(F)\colon q^{-1}(U)\to q^{-1}(V)$ given by  $\varphi(F)(z)=[F]_{q(z)} \cdot  z$. (Recall $[F]_{q(z)}$ denotes the germ of $F$ at $q(z)$). This defines a continuous morphism $\varphi\colon \tG\to \tI(Z)$. 

Conversely let $\varphi\colon \tG\to \tI(Y)$ be a continuous morphism with spatial component  $q\colon Z\to X$, and let us define an action $\alpha \colon  \G\tensor[_s]{\times}{_q} Z\to Z$. Let $(\gamma, z)\in \G\tensor[_s]{\times}{_q} Z$. Choose any bisection $F$ of $\G$ such that $\gamma \in F$, and set $U=\src(F)$.  Since $q(z)=\src(\gamma)\in U$, we have  $z\in q^{-1}(U)$ which implies  that $z$ is contained in the domain of the partial homeomorphism $\varphi(F)$. Define $\gamma \cdot z=\varphi(F)(z)$. Let us check that this definition does not depend on the choice of $F$. Let $F'$ be another bisection containing $\gamma$, with $U'=\src(F')$.  Since $F\cap F'$ is a neighbourhood of $\gamma$ in $\G$ and bisections form a basis for the topology,  there exists a bisection $T$ such that $\gamma \in T\subset F\cap F'$. Setting $W=\src(T)$, we have $W\subset U\cap U'$. Moreover $T=FW=F'W$, and applying $\varphi$ we see that $\varphi(F)|_q^{-1}(W)=\varphi(FW)=\varphi(T)=\varphi(F'W)=\varphi(F')|_{q^{-1}(W)}$. Since $z\in q^{-1}(W)$ this implies that $\varphi(F)(z)=\varphi(F')(z)$ as desired. Thus the action $\alpha$ is well defined. The fact that it is continuous easily follows from the construction. \qedhere 
 \end{proof}
 
 \subsection{An equivalence of categories between pseudogroups and  groupoids} \label{s-functor-equivalence}
In this subsection we study the functoriality of the correspondence $\G \to \tG$ between \'etale groupoids and pseudogroups.

\begin{remark} \label{r-cocycles} The more straightforward notion of morphism between \'etale groupoids are \emph{cocycles}. A \textbf{cocycle} is a  a map $c\colon \G_1\to \G_2$ between groupoids which commutes with source and range maps and preserves  all operations. In general, a continuous cocycle $c\colon \G\to \H$ does not give rise to a continuous morphism of pseudogroup (neither covariantly, nor contravariantly), but it does under some assumptions on $c$. This has been studied by Lawson and Lenz \cite{Law-Le-Stone}, who show that if one restrict the attention to a class of cocycles (named \emph{covering functors} in \cite{Law-Le-Stone}) and to a special class of continuous morphisms of pseudogroups (named \emph{callitic} in \cite{Law-Le-Stone}), the correspondence $\G\to \tG$ becomes a contravariant equivalence of categories. However, we will need to work with arbitrary morphisms (which need not be callitic). We therefore  take a different point of view. \end{remark}

\medskip

Observe that if $G$ and $H$ are just countable groups, a group homomorphism $G\to H$ can equivalently be thought as a left {action}  of $G$ on $H$ which commutes with the natural right action of $H$ on itself. The following definition generalises this second point of view (rather than the first) to \'etale groupoids.   

\begin{defin}
Let $\G, \H$ be \'etale groupoids over spaces $X, Y$. 
A left action $\alpha \colon \G\acts (\H, p)$   is said to be a \textbf{translation action} if it commutes with the natural right action  $(\H,\src) \curvearrowleft \H$ given by the groupoid operation $(\gamma, \delta)\mapsto \gamma\delta$ (over the source map $\src\colon \H\to Y$). 

 Here a left action $\G\acts (Z, p_1)$ and a right action $(Z, p_2)\curvearrowleft \H$ are said to \textbf{commute} if for every $\gamma\in \G, z\in Z, \delta \in \H$ such that $\src(\gamma)=p_1(z)$ and $p_2(z)=\rg(\delta)$ (so that $\gamma\cdot z$ and $z\cdot \delta$ are both defined),  we have $p_1(z\cdot \delta)=p_1(z), p_2(\gamma \cdot z)=p_2(z)$ and  $(\gamma \cdot z)\cdot \delta=\gamma\cdot (z\cdot \delta)$. 
\end{defin}

\begin{remark}
It follows from the definition that for every translation action $\alpha \colon \G\acts (\H, p)$,  the map $p\colon \H\to X$ satisfies $p(\delta)=p(\delta\delta^{-1})=p(\rg(\delta))$. Thus,  it  is of the form $p=q\circ \rg$ for some map between the spaces $q\colon Y\to X$  (namely the restriction of $p$ to $Y$). We will freely use this observation in what follows. 
\end{remark}

Let $\G, \H, \K$ be \'etale groupoids over spaces $X, Y, Z$, and let $\alpha\colon \G \acts (\H, q\circ \rg)$ and $\beta \colon \H\acts (\H, p\circ \rg)$ be translation actions. Their \textbf{composition}
is the translation action $\xi \colon \G\acts (\K, q\circ p\circ \rg)$ which for every $\gamma\in \G$ and $\delta \in \K$ such that $\src(\gamma)=q\circ p\circ \rg(\delta)$ is given by the formula
\[\gamma\cdot_\xi \delta=\left(\gamma\cdot_\alpha p((\rg(\delta))\right)\cdot_\beta \delta.\]
\begin{lemma}
The composition of two translation actions is a well-defined translation action.
\end{lemma}
\begin{proof}
With the notations as above,  using that $\alpha$ is an action, we have
\[(\gamma_1\gamma_2)\cdot_\xi \delta=((\gamma_1\gamma_2)\cdot_\alpha p(\rg(\delta))\cdot_\beta \delta= (\gamma_1\cdot_\alpha (\gamma_2\cdot_\alpha p(\rg(\delta))) \cdot_\beta \delta.\]
Setting $\eta=\gamma_2\cdot_\alpha p(\rg(\delta))$ and using that $\alpha, \beta$ are translation actions, we obtain  
\begin{multline*}(\gamma_1\gamma_2)\cdot_\xi \delta=((\gamma_1\cdot_\alpha \eta\eta^{-1}) \eta )\cdot_\beta \delta= ((\gamma_1\cdot_\alpha \rg(\eta))\eta) \cdot_\beta \delta=(\gamma_1\cdot_\alpha \rg(p(\eta \cdot \beta \delta ))\eta) \cdot_\beta \delta=\\ (\gamma_1\cdot_\alpha \rg(p(\eta \cdot \beta \delta ))\cdot_\beta(\eta \cdot_\beta \delta)= \gamma_1\cdot_\xi(\eta \cdot_\beta \delta)=\gamma_1 \cdot_\xi(\gamma_2\cdot_\xi \delta)\end{multline*}
showing that $\xi$ is indeed an action. The verification that it is a translation action is similar and even simpler. \qedhere\end{proof}

 This allows to define a category whose objects are \'etale groupoids, and whose morphisms are translation actions. In this category, the identity morphism over $\G$ is given by the left action of $\G$ on itself $\G\acts (\G, r)$ by the groupoid operation. 
 

%
%
%
%
%
%
%
%
%
%
%
%
%
%
%
%
%

\begin{thm}[Functorial equivalence between pseudogroups and groupoids] \label{t-non-comm-stone} 
Let $\G$ and $\H$ be \'etale groupoids. There is a natural one-to-one correspondence between continuous morphisms of    pseudogroups $ \tG\to \tH$ and translation actions of $\G$ on $\H$.  Under this correspondence, a continuous morphism $\varphi\colon \tG\to \tH$ with spatial component $q\colon Y\to X$ corresponds to a translation $\alpha \colon \G\acts \H$  over the map $q\circ \rg\colon \H\to X$. 

Furthermore, this correspondence is functorial, and establishes  an equivalence between the category of pseudogroups with continuous morphisms, and the category of \'etale groupoids with translation actions as morphisms.  


\end{thm}
Explicitly, the correspondence claimed in the statement is as follows. Let $\varphi\colon \tG\to \tH$ be a continuous morphism of pseudogroups with spatial component $q\colon Y\to X$. Let us define a translation action  $\G\acts (\H, q\circ \rg)$. Let $\gamma\in \G, \delta\in \H$ be such that $q(\rg(\delta))=\src(\gamma)$. Choose a bisection $F\in \tG$ such that $\gamma\in F$. Define
 \begin{equation} \gamma\cdot \delta:= [\varphi(F)]_{\rg(\delta)} \delta.   \label{e-action-morphism} \end{equation}
  We claim that the formula \eqref{e-action-morphism} defines a translation action, and establishes the desired functorial one-to-one correspondence. The proof includes a series of elementary verifications, that we spell out for completeness:

\begin{proof}

  Let $\varphi \colon \tG\to \tH$ be a continuous morphism, and let us show that the formula \eqref{e-action-morphism} defines a translation action of $\G$ on $(\H, q\circ \rg)$. 
We first check that $\gamma \cdot \delta$ is well defined, i.e. does not depend on the choice of $F$. Assume that $T$ is another bisection containing $\gamma$. Then we can find a smaller bisection $S\subset T\cap F$ containing $\gamma$. This means that we have $S=TU$ and $S=FU$ for some neighbourhood $U=\src(S)$ of $y$. Applying $\varphi$ we deduce that $\varphi(S)=\varphi(T)q^{-1}(U)$ and $\varphi(S)= \varphi(T)q^{-1}(U)$. Since $\rg(\delta)=[q^{-1}(U)]_{\rg(\delta)}$  is a unit, we have $[\varphi(F)]_{\rg(\delta)}=[\varphi(T)]_{\rg(\delta)}=[\varphi(S)]_{\rg(\delta)}$. This shows that $\gamma \cdot \delta$ is a well-defined element of $\H$.  Moreover the map $(\gamma, \delta)\mapsto \gamma\cdot \delta$ is continuous: for $(\gamma', \delta')$ close enough to $(\gamma, \delta)$, we can choose the same bisection $F$ to define $\gamma'\cdot \delta'$, and  the map $\delta'\mapsto   [\varphi(F)]_{\rg(\delta')} \delta'$ is continuous. 
The fact that the formula \eqref{e-action-morphism} defines an action is a direct consequence of the fact that $\varphi$ is a semigroup morphism, and it is apparent from \eqref{e-action-morphism} that this action commutes with the right action of $\H$ on itself. Finally let us check that this construction is functorial. Assume that $\K$ is another \'etale groupoid over a space $Z$. Let  $\varphi\colon \tG\to \tH$ and $\psi\colon \tH\to \widetilde{\K}$ be continuous morphisms, and $q\colon Y\to X$ and $p\colon Z\to Y$ be their spatial component. The spatial component of $\psi\circ \varphi$ is obviously $q\circ p$. Let $\alpha\colon \G\acts (\H, q\circ \rg), \beta\colon \H\acts (\K, p\circ \rg)$ and $\xi \colon\G\acts (\K, q\circ p\circ \rg)$ be the translation actions associated to $\varphi, \psi$ and $\psi\circ\varphi$ respectively, and let us show that $\xi$ coincides with the composition of $\alpha $ and $\beta$.  Let $\gamma \in \G$ and $\delta\in \K$ be such that $\src(\gamma)=q(p(\rg(\delta))$, and let $F$ be a bisection containing $\gamma$.  Since $p(\rg(\delta))$ is a unit of $\H$, we have 
\[\gamma\cdot_\alpha p(\rg(\delta))=[\varphi(F)]_{p(\rg(\delta))}p((\rg(\delta))=[\varphi(F)]_{p(\rg(\delta))}\in \varphi(F)\]
thus, $\varphi(F)$ is a bisection that contains $(\gamma \cdot_\alpha p(\rg(\delta))$ and we can use it as a representative in the definition of the action $\beta$, obtaining 
\[(\gamma\cdot_\alpha p(\rg(\delta)))\cdot_\beta \delta= [\psi(\varphi(F))]_{\rg(\delta)}\delta=\gamma \cdot_\xi \delta\]
showing that $\xi$ coincides with the composition of $\alpha$ and $\beta$. 
We have shown that the formula \eqref{e-action-morphism} defines a functor from pseudogroups with continuous morphisms to \'etale groupoids with translation actions. It remains to show that this is in fact an equivalence of categories. 

To this end we first prove the following lemma.
\begin{lemma} \label{l-open-map}
If $\alpha\colon \G\acts (\H, q\circ \rg)$ is a translation action, the map \begin{align*}\G\tensor[_s]{\times}{_q} Y\to \H \\ (\gamma, y)\mapsto \gamma\cdot y \end{align*} 
is open and is a local homeomorphism.
\end{lemma}
\begin{proof}
First observe that the projection $\pi_2\colon \G\tensor[_s]{\times}{_q} Y \to Y$ is open and is a local homeomorphism. To see this, note that a basis for the topology of $ \G\tensor[_s]{\times}{_q} Y $ is given by sets of the form $F\tensor[_s]{\times}{_q} V$ for $F$ a bisection of $\G$ and $V\subset Y$ open. The restriction of $\pi_2$ to $F\tensor[_s]{\times}{_q} V$ is a homeomorphism with image  $\pi_2( F\tensor[_s]{\times}{_q} V)=\{y\in V\colon \exists \gamma\in F, \src(\gamma)=q(y)\}=V\cap q^{-1}(\src(F))$, which is therefore open, and inverse given by $y\mapsto ([F]_{q(y)},y)$. 

Now fix $(\gamma, y)\in  \G\tensor[_s]{\times}{_q} Y$. Let $T$ be a bisection of $\H$ containing $\gamma\cdot y$. By the continuity of $(\gamma, y)\mapsto \gamma\cdot y$ we can find a neighbourhood  $W$ of $(\gamma, y)$ which is mapped into $T$. Note that for every $(\gamma', y')\in W$ we have $\pi_2(\gamma', y')=y'$, and $\src(\gamma'\cdot y')=y'$. Thus on $W$, the map $(\gamma', y')\mapsto \gamma'\cdot y'$ coincides with the composition of $\pi_2\colon W\to Y$, followed by the local inverse $\src|_T^{-1}\colon \src(T)\to T$. Since both maps are open and local homeomorphisms, this proves the lemma. \qedhere

\end{proof}

Going back to the proof of the theorem, let $\alpha \colon \G\acts(\H,  q\circ \rg)$ be a translation action, and let us define a continuous morphism $\varphi\colon \tG\to \tH$. For every bisection $F\in \tG$ we set
\[\varphi(F)=\{\gamma \cdot y \medskip \colon \medskip \gamma \in F, y\in Y, \src(\gamma)=q(y)\}.\]
Note that $\varphi(F)$ is the image of $F\tensor[_s]{\times}{_q} Y$ under the map $(\gamma, y)\mapsto \gamma \cdot y$. Thus, Lemma \ref{l-open-map} implies that  $\varphi(F)$ is open.  To see that it is a bisection note that we have $\src(\varphi(F))=q^{-1}(\src(F))$, and the restriction of the source map to $\varphi(F)$ is invertible with inverse $y\mapsto [F]_{q(y)} \cdot y$. Similarly one checks that $\rg(\varphi(F))=q^{-1}(\rg(F))$ and that the restriction of $\rg$ to $\varphi(F)$ is invertible. Thus $\varphi(F)$ is a bisection. Let us check that $\varphi$ is a semigroup homomorphism.
Fix $F_1, F_2\in \tG$. Let us show that $\varphi(F_1F_2)\subset \varphi(F_1)\varphi(F_2)$. An element  of $\varphi(F_1F_2)$ has the form $(\gamma_1\gamma_2)\cdot y$ for some  $\gamma_1\in F_1, \gamma_2\in F_2$ and $y\in Y$  such that $ \src(\gamma_1)=\rg(\gamma_2)$, and $q(y)=\src(\gamma_1)$. Put $z=\rg(\gamma_2\cdot y)$. Since $z(\gamma_1\cdot y)=\gamma_1\cdot y$, using that the action is a translation action we obtain $(\gamma_1\gamma_2)\cdot y=\gamma_1\cdot (\gamma_2\cdot y)= \gamma_1\cdot (z(\gamma_2\cdot y))=(\gamma_1\cdot z)(\gamma_2\cdot y)\in \varphi(F_1)\varphi(F_2)$, proving that $\varphi(F_1F_2)\subset \varphi(F_1)\varphi(F_2)$. To see the converse, let $\gamma_1\cdot y_1\in \varphi(F_1)$ and $\gamma_2\cdot y_2\in \varphi(F_2)$ be such that $\rg(\gamma_2\cdot y_2)=\src(\gamma_1\cdot y_1)=y_1$. Then $(\gamma_1\cdot y_1)(\gamma_2\cdot y_2)=\gamma_1\cdot (y_1(\gamma_2\cdot y_2))=\gamma_1\cdot (\gamma_2\cdot y_2)=(\gamma_1\gamma_2)\cdot y_2\in \varphi(F_1F_2)$. This shows that $\varphi$ is a semi-group homomorphism. 

The fact that $\varphi$ preserves unions of compatible families, zero and identity, are clear from its  definition. Moreover it is not difficult to check that the action associated to $\varphi$ via the formula \eqref{e-action-morphism} is the action $\alpha$ we started with. 
 The proof is complete. \qedhere

\end{proof}

%
\begin{remark}\label{r-Lawson-Lenz}
Let us explain how this equivalence  includes the one described by Lawson and Lenz \cite{Law-Le-Stone}, mentioned in Remark \ref{r-cocycles}. In their terminology, a cocycle $c\colon \H\to \G$ is called a \emph{covering functor} if for every $y\in Y$ the restriction of $c$ to the leaf $\H_y$ is a bijection with the leaf $\G_{c(y)}$. A covering functor gives rise to a translation action $\G\acts (\H, c|_Y\circ \rg)$ as follows: for every $\gamma\in \G$ and $\delta\in \H$ such that $c(\rg(\delta))=\src(\gamma)$, define $\gamma \cdot \delta:= \eta\delta$, where $\eta\in \H_{\rg(\delta)}$ is the unique preimage of $\gamma$ under $c$. One can check that the continuous morphism $\varphi\colon \tG\to \tH$ associated to this translation action is given by $\varphi(F)=c^{-1}(F)$, and that the morphisms arising in this way are precisely those called \emph{callitic} in \cite{Law-Le-Stone}. 
(Note that what we call \emph{pseudogroup} is called \emph{spatial pseudogroup} in \cite{Law-Le-Stone}, the word \emph{pseudogroup} being used in a more general way in \cite{Law-Le-Stone}.)
\end{remark}

As a first application of this theorem, let us prove Proposition \ref{p-restriction-morphism}
\begin{proof}[Proof of Proposition \ref{p-restriction-morphism}]
Let $\varphi\colon \tG\to \tH$ be a continuous morphism with spatial component $q\colon Y\to X$, and set $C=q(Y)$.  We consider the associated translation action $ \G\acts (\H, q\circ \rg)$. It is tautologically the composition of the natural translation action $\G\acts (\G|C, \iota\circ \rg)$  by the groupoid operation (where $\iota \colon C\to X$ is the inclusion) and of a translation action $\G|C\acts (\H, q\circ \rg)$. \qedhere
\end{proof}

\subsection{Coarse structure on leaves and compactly generated groupoids} \label{s-cayley} \label{s-leaves}

In this subsection we assume for simplicity that $\G$ is an etale groupoid over a \emph{compact} space $X$. Preliminaries on coarse spaces where recalled in \S \ref{s-coarse}.

For every $x\in X$, the leaf $\G_x$ is naturally a coarse space, where $E\subset \G_x\times \G_x$ is declared to be a controlled set if there exists a compact subset $K\subset \G$ such that  $\{\gamma\delta^{-1} \colon (\gamma, \delta)  \in E\} \subset K$.  (A similar definition can be given for the leaf $\G^x$)

Note that  the coarse space $\G_x$ is uniformly discrete. To see this assume that $\{B_i\}_{i \in I}$ is a family of uniformly bounded subsets of $\G_x$, and let $K\subset \G$ be a compact set such that $\gamma\delta^{-1}\in K$ for every $\gamma, \delta \in B_i$ and every $i\in I$.  The set $K$ can be covered with finitely many open bisections $F_1,\ldots, F_n$. Thus if we fix some $\delta\in B_i$ we have that every  $\gamma\in B_i$ has the form $\gamma=(\gamma\delta^{-1})\delta=[F_i]_{\rg(\delta)\delta}$ for some $i$. Therefore $|B_i|\le n$ for all $i\in I$.    

In a similar fashion, for every $x\in X$ the orbit $\rg(\G_x)=\{\src(\gamma)\colon \gamma \in \G_x\}$ can be endowed with a structure of a (uniformly discrete) coarse space.

In a relevant special case these structure come from a structure of graph. 
Following Haefliger \cite{Hae} (see also \cite[Sec. 2.3]{Nek-hyp}), recall that  $\G$ is said to be \textbf{compactly generated} if there exists a compact subset $K\subset \G$ such that every element of $\G$ is a product of elements of $K \cup K^{-1}$ (note that the definition  for groupoids with non-compact unit space is more complicated, see \cite{Hae}).

\begin{example} 
If $G\acts X$ is a finitely generated group action on a compact space, its groupoid of germs is compactly generated, a compact generating set being given by the unions of the sets of germs of elements in a finite generating set of $G$.
\end{example}

In terms of the pseudogroup $\tG$, compact generation translates as  follows \cite{Hae}. 

\begin{defin} \label{def-Hae}

 A bisection $F\in \tG$  is said to be \textbf{extendable} if there exists $F'\in \tG$ such that $F\subset F'$ and the closure of $\src(F)$ in $Y$ is contained in $\src(F')$. Note that extendable bisections always form a basis of the topology. The   pseudogroup $\tG$ is said to be \textbf{compactly generated} if there exists a finite set $\mathcal{T}=\{T_1, \ldots, T_r\}\subset \tG$ of extendable bisections such that the set of products $\bigcup_{n\ge 0} (\mathcal{T}\cup \mathcal{T}^{-1})^n$ is a cover of the \'etale groupoid $\G$.   The set $\mathcal{T}$ is called a finite {generating set} of $\tG$. 
\end{defin}
The pseudogroup $\tG$ is compactly generated if and only if the \'etale groupoid $\G$ is so. Namely if $K\subset \G$ is a compact set generating $\G$, every finite cover of $K$ by  extendable bisections is a finite generating set of $\tG$. Conversely let $\T=\{T_1,\cdots , T_r\}$ be a finite generating set of $\tG$, and choose bisections $T'_i\supset T_i$ as in the definition of extendable bisections. Then $K=\bigcup \bar{T}_i$ is a compact generating set of $\G$, where $\bar{T}_i$ refers to the closure in $T'_i$.

If the space $X$ is totally disconnected space (e.g. a Cantor space) the requirement that the bisections ${T}_i$ are extendable  can be replaced by the requirement that they are compact and open. {In this case, we will always consider generating sets consisting of compact open bisections}.

Assume that $\G$ is a compactly generated \'etale groupoid, with finite symmetric generating set $\T=\{T_1,\ldots, T_r\}\subset \tG$. Following Nekrashevych \cite{Nek-hyp}, for every $x\in X$ we define the \textbf{Cayley graph}  $\cay_x(\G, \mathcal{T})$ as the graph whose vertex set is the leaf $\G_x$ and with edges labelled by $\T$, where $\gamma_1, \gamma_2$ are connected by an oriented edge labelled by $T_i\in \mathcal{T}$ if $\gamma_2=[T_i]_{\rg(\gamma_1)} \gamma_1$.

Similarly for every $x\in X$, we also define the  \textbf{orbital graph}  $\orb_y(\H, \mathcal{T})$  as the $\T$-labelled graph whose vertex set is the orbit  of $\rg(\G_x)$, and where $x_1, x_2$ are connected by an  oriented edge for every $i=1, \ldots, r$  such that $x_2=\tau(T_i)(x_1)$. 

%
%
%
%
%
%
%

\begin{example}
Assume that $\G$ is the groupoid of germs  of a continuous action $G\acts X$  of a finitely generated group. Then, for the natural choices of generating sets,  the orbital graphs and the Cayely graphs of $\G$ based at $x$ coincide with the orbital graph and the graph of germs of the action $G\acts X$ as defined in \S \ref{s-notations}. 
\end{example}
%

The following proposition implies that for minimal  groupoids, the quasi-isometry type of the Cayley and orbital graphs is invariant under restriction.
 For a proof we refer to \cite[Cor. 2.3.4]{Nek-hyp}.
\begin{prop}\label{p-compact-generation-stable} \label{p-graph-qi}
Assume that $\G$ is minimal. Let $U\subset X$ be a compact open subset. Then for every $x\in U$ the fibres $(\G|U)_x$ and $\G_x$ are coarsely equivalent spaces. Moreover  $\G|U$ is compactly generated if and only if $\G$ is so.  
\end{prop}

%

It is an obvious but fundamental fact that group embeddings between discrete groups are coarse embeddings. 
The following proposition  will serve as a sufficiently good  replacement of this fact for  pseudogroups.
Its proof  is an application of Theorem \ref{t-non-comm-stone},
\begin{prop}\label{p-morphism-leaves}
Let $\G, \H$ be  \'etale groupoids over compact spaces $X, Y$, and let $\varphi\colon \tG\to \tH$ be a continuous morphism with spatial component $q\colon Y\to X$. Let $x\in q(Y)$, and assume the isotropy group $\G_x^x$ is trivial. Then for every $y\in q^{-1}(x)$, there exist an injective coarse map $\G_x\hookrightarrow \H_y$.
 \end{prop}
  Note that when $\varphi$ is injective, we have $q(Y)=X$  (Proposition \ref{p-inj-surj}), and thus this can be applied to every $x$ with trivial isotropy group.
  \begin{proof}
 By Theorem \ref{t-non-comm-stone}, we can consider a translation action $\G\acts(\H, q\circ \rg), (\gamma, \delta)\mapsto \gamma\cdot \delta$ which corresponds to $\varphi$. For $x, y$ as in the statement we have a map $\G_x\to \G_y$ given by $\gamma \mapsto \gamma \cdot y$. This map is  injective: if $\gamma\neq \gamma' \in \G_x$, the assumption that $\G_x^x$  is trivial implies that $\gamma$ and $\gamma'$ have different range, and thus $q(\rg(\gamma\cdot y))=\rg(\gamma)\neq \rg(\gamma')=q(\rg(\gamma'\cdot y))$, showing that $\gamma\cdot y\neq \gamma' \cdot y$. Let $K\subset \G$ be a compact subset. If $\gamma_1, \gamma_2\in \G_x$ are such that $\gamma_1=\delta\gamma_2$ for $\delta\in K$, then using that the action is a translation action we see that we must have $(\gamma_1\cdot y)=(\delta \rg(\gamma_2\cdot y))\cdot (\gamma_2 \cdot y)$, and thus $(\gamma_1\cdot y)(\gamma_2 \cdot y)^{-1}$ belongs to the set $K\cdot Y=\{\delta \cdot z\colon \delta \in K, z\in Y, \src((\delta)=q(z)\}$, which is a compact subset of  $\H$. This shows that the map is bornologous. Finally it is enough to note that an injective bornologous map between uniformly discrete coarse spaces is automatically a coarse map.  \qedhere
 

 \end{proof}

\subsection{Functoriality of homology} \label{s-homology}
Let us recall the definition of the homology groups of an \'etale groupoid with coefficient in an abelian group $A$. For simplicity, we assume  that $\G$ is an \'etale groupoid over a Cantor space $X$ (we will only use homology in this case).  

For $n\ge 1$, let  $\Delta_n(\tG)$ be the set of  consisting set of all  tuples $(F_1,\cdots , F_n)$ of compact open bisections $F_i\in \tG$ such that $\src(F_i)=\rg(F_{i+1})$ for every $ i=1,\ldots n-1$. For $n=0$ we let $\Delta_0(\tG)$ be the set of all clopen subsets of $X$. 
Define  $C_n(\tG)$ to be the abelian group with generating set $\Delta_n(\tG)$ and  relations of the form 
\begin{equation} \label{e-chain-relation}(F_1,\ldots, F_n)= (T_1,\ldots, T_n) +(S_1,\ldots, S_n)    \end{equation}
whenever $F_i=T_i\sqcup S_i$ for $i=1,\cdots, n$. We let $C_n(\tG, A)=A \otimes_\Z C_n(\tG)$ (where both $A$ and $C_n(\tG)$ are seen as $\Z$-modules).

For $n\ge 2$ and  $i=0,\ldots ,n$ define maps $d_i\colon \Delta_n(\tG)\to \Delta_{n-1}(\tG)$ by
\[d_i(F_1,\cdots F_n)=\left\{\begin{array}{lr} (F_2,\ldots, F_n) & \text{ if } i=0\\ (F_1,\ldots,F_iF_{i+1},\ldots, F_n) & \text{ if } 1\le i\le n-1\\ (F_1, \cdots, F_{n-1}) & \text{ if } i=n\end{array} \right. \]
 If $n=1$,  set $d_0(F)=\rg(F)$ and $ d_1(F)=\src(F)$. The maps $d_i$  preserve the relations \eqref{e-chain-relation}, and thus extend  to group homomorphisms 
 \[d_i\colon C_n(\tG)\to C_{n-1}(\tG).\] 
  Setting $\partial_n=\sum_{i=0}^n(-1)^i d_i$, we obtain a chain complex
  \[\begin{tikzcd}
0 & C_0(\tG, A) \arrow[l, "0"] & C_1(\tG, A) \arrow[l, "\operatorname{Id}_A\otimes \partial_1"] &  C_2(\tG, A) \arrow[l, "\operatorname{Id}_A\otimes \partial_2"] & \arrow[l, "\operatorname{Id}_A \otimes \partial_3"] \cdots 
\end{tikzcd}\]
  The \textbf{homology}  of $\tG$ with coefficient in $A$ is defined as the homology of this complex, namely $H_n(\tG, A)=\operatorname{Ker}(\operatorname{Id}_A\otimes \partial_n)/\operatorname{Im}(\operatorname{Id}_A\otimes \partial_{n+1})$.  If $A=\Z$, we simply write $H_n(\tG)$.

%

\begin{remark} This definition is equivalent to the definition given by Matui in \cite{Mat-hom} (extended in \cite{Nek-simple} to non-Hausdorff groupoids). To see this, observe that the group $C_n(\tG, A)$ admits  the following interpretation. Let 
\[\G^{*n}=\{(\gamma_1,\ldots, \gamma_n)\colon \rg(\gamma_i)=\src(\gamma_{i+1}), i=1,\ldots n\}\]
be the set of composable $n$-tuples of the groupoid $\G$. For $n=0$ we set $\G^{*0}=X$. For every $(F_1,\cdots, F_n)\in \Delta_n(\tG)$ and every $a\in A$, the  element $ a\otimes (F_1,\cdots, F_n) $  defines a function $\G^{\ast n} \to A$ which maps $(\gamma_1,\ldots, \gamma_n)$ to $a$ if $\gamma_i\in F_i$ for every $i=1,\ldots, n$, and to $0$ otherwise. The group $C_n(\tG, A)$ is isomorphic to the subgroup of the group of all functions from $\G^{*n}$ to $A$ generated by all functions  of this form. If $\G$ is Hausdorff, this is the same as the group of continuous, compactly supported functions $\G^{*n} \to A$. 
For $n=0$ and for any $\G$ (Hausdorff or not), the group $C_0(\tG, A)$ coincides with the group of continuous functions $C(X, A)$. Using these identifications, it is clear  that the definition of the groups $H_n(\G, A)$ coincides with the definition given in \cite{Mat-hom, Nek-simple}.
\end{remark}

Groupoid homology behaves functorially with respect to continuous morphisms of pseudogroups. Namely assume that $\G_1, \G_2$ are \'etale groupoids over Cantor spaces $X_1, X_2$, and that $\varphi \colon \tG_1\to \tG_2$ is a continuous morphisms. Then $\varphi$ induces a map, still denoted $\varphi\colon \Delta_n(\tG_1)\to \Delta_n(\tG_2)$, by $\varphi(F_1,\ldots, F_n)=(\varphi(F_1),\dots, \varphi(F_n))$ and thus induces a homomorphism \[\operatorname{Id}_A \otimes \varphi\colon C_n(\tG_1, A)\to C_n(\tG_2, A),\]
 which is clearly a morphism of chain complexes.  
 We deduce:

\begin{prop} \label{p-morphism-homology} Every continuous morphism of pseudogroups $\varphi\colon \tG_1\to \tG_2$ induces maps between the homology groups
\[\varphi_\ast \colon H_n(\G_1, A)\to H_n(\G_2, A).\]

\end{prop}

This will be used in our applications of the Extension Theorem, in particular to groups of interval exchanges (\S \ref{s-iet}) and to full groups of one sided shift of finite type (\S \ref{s-SFT}). 

\subsection{The symmetric power of an \'etale groupoid} \label{s-reduced}

Let $\G$ be an \'etale groupoid over a space $X$, and $r \ge 1$. Let us give a more formal definition of the symmetric power $\G^{[r]}$  (whose definition was sketched in the introduction for $\G$ effective). This will be needed  later in our main theorem in \S \ref{s-Extension Theorem}. We anticipate it  to this section since  it also provides  some instructive examples concerning pseudogroups morphisms.

Given a (Hausdorff) space $X$ and an integer $r\geq 1$, let $X^{[r]}$ be the  space of non-empty finite subsets of $X$ of cardinality at most $r$. It is endowed with the quotient of the product topology on  the cartesian power $X^r$ via the surjection $X^r\to X^{[r]}, (y_1,\ldots, y_r)\mapsto \{y_1,\ldots y_r\}$. 

\begin{defin}\label{d-reduced} \label{d-reduced-power} Let $\G$ be an \'etale groupoid over $X$, and $r\geq 1$. The \textbf{symmetric power}   $\G^{[r]}$  is the groupoid over $X^{[r]}$, consisting of all non-empty finite subsets $\Qc\subset \H$ of cardinality at most $r$ such that every two distinct elements $\gamma, \delta\in \Qc$ have different source and different range. 

\end{defin}

The groupoid structure on $\G^{[r]}$ is defined as follows. Source and range maps $\src, \mathsf{r}\colon \G^{[r]}\to X^{[r]}$ are given by the image of of the source and range maps of $\G$.    It follows that for every $\Qc_1, \Qc_2\in \G^{[r]}$ such that $\src(\Qc_1)=\rg(\Qc_2)$, and every $\gamma_2\in \Qc_2$, there exists a unique $\gamma_1\in \Qc_1$ such that $\src(\gamma_1)=\rg(\gamma_2)$. The product $\Qc_1\Qc_2$ is defined as the set-wise product in $\G$:
\[   \Qc_1\Qc_2=\{\gamma_1\gamma_2\colon \gamma_1\in \Qc_1, \gamma_2\in \Qc_2, \src(\gamma_1)=\rg(\gamma_2)].   \]

To define a topology on $\G^{[r]}$, let $\mathcal{E}\subset \G^r$ be the set of tuples $(\gamma_1,\cdots, \gamma_r)$ with the property that for every $i, j$ such that $\gamma_i\neq \gamma_j$, we have $\src(\gamma_i)\neq \src(\gamma_j)$  and $\rg(\gamma_i)\neq \rg(\gamma_j)$. Note that $\mathcal{E}$ is open in $\G^r$, and that we have a surjective map $\mathcal{E}\to \G^{[r]}, (\gamma_1,\cdots, \gamma_r)\mapsto \{\gamma_1,\cdots, \gamma_r\}$. We consider on $\G^{[r]}$ the quotient of the topology on $\mathcal{E}$, under this surjective map. 

With this topology, $\G^{[r]}$ is \'etale. Explicitly a basis of bisections of $\G^{[r]}$ can be described as follows.  Let $T_1,\cdots , T_s\in \tG$ be  compatible bisections, with $1\le s\le r$.  Then the set 
 \[\mathcal{U}_{T_1,\cdots, T_s}=\{\Qc\in \G^{[r]} \colon \Qc\subset T_1\cup\cdots \cup T_s, \Qc\cap T_i\neq \varnothing\forall i=1,\cdots, s\}\]
 is a bisection of $\G^{[r]}$, and sets of this form are a basis of the topology of $\G^{[r]}$.

The sets $X^{[r]}$ are naturally nested: 
  \[X=X^{[1]}\subset X^{[2]}\subset X^{[3]}\subset \cdots X^{[r-1]}\subset X^{[r]} \]
  If $\ell< r$, the set $X^{[\ell]}$ is a closed  $\G^{[r]}$-invariant subset of $X^{[r]}$. This shows in particular that $\G^{[r]}$ is never minimal if $r \ge 2$ (unlike the usual power $\G^r$, which is minimal if and only if $\G$ is). 
  
  For every $\ell\le r$ we have $\G^{[r]}|X^{[\ell]}=\G^{[\ell]}$. Thus, the restriction morphisms (Example \ref{e-restriction-morphism}) provide us with a sequence of continuous morphisms
  \[\widetilde{\G^{[r]}}\to \widetilde{\G^{[r-1]}}\to \cdots \to \widetilde{\G^{[2]}}\to \widetilde{\G}.\]

  \begin{lemma}\label{l-closed-power}
  Assume that $\G$ is minimal. Then for every $r\ge 1$, the only closed $\G^{[r]}$-invariant subsets of $X^{[r]}$ are the sets $X^{[\ell]}, \ell=1,\cdots, r$.
  
  In particular assume that $\H$ is an \'etale groupoid over a compact space $Y$, and let $\varphi\colon \tG^{[r]} \to \widetilde{\H}$ be a continuous morphism. Then there exists a unique $\ell \le r$ such that $\varphi$ splits as the composition of the restriction homomorphism $\widetilde{\G^{[r]}}\to \widetilde{\G^{[\ell]}}$ and a continuous morphism $\widetilde{\G^{[\ell]}}\to \tH$ whose spatial component $q\colon Y\to X^{[\ell]}$ is surjective.  
  \end{lemma}
  \begin{proof}
  Let $C\subset X^{[r]}$ be a closed $\G^{[r]}$-invariant subset, and let $Q:=\{x_1,\cdots, x_\ell\}\in C$ be such that $\ell:=|Q|$ is maximal. Let us show that $C=X^{[\ell]}$. We have $C\subset X^{[\ell]}$ by maximality of $\ell$. Let $P=\{y_1,\ldots, y_\ell\}\in X^{[\ell]}$, and let $U_1,\cdots U_\ell$ be neighbours of $y_1,\ldots, y_\ell$. By minimality we can find elements $\gamma_i\in \G$ such that $\src(\gamma_i)=x_i, \rg(\gamma_i)\in U_i$, and $\rg(\gamma_1),\ldots , \rg(\gamma_\ell)$ are pairwise distinct. This shows that $P$ is in the closure of the $\G^{[r]}$-orbit of $Q$, and thus $P\in C$. Since $P\in X^{[\ell]}$ was arbitrary we have $X^{[\ell]}\subset C$.
  
  For the second part, it is enough to observe that by compactness  $q(Y)$ is  closed, thus equal to $X^{[\ell]}$ for some $\ell\le r$, and to apply Proposition \ref{p-restriction-morphism}.  \qedhere  
  \end{proof}
  
  We now give an example of continuous morphism of pseudogroups. 
  
 \begin{example}  The symmetric power $\G^{[r]}$ is related to the usual (cartesian) power $\G^r$ by the existence of a  continuous morphism $\widetilde{\G^{[r]}}\to\widetilde{\G^r}$ whose  spatial component is the surjection $X^r\to X^{[r]}, (x_1,\ldots, x_r)\mapsto \{x_1,\ldots, x_r\}$. 
  
Too see this, let $q\colon X^r\to X^{[r]}$ be the  surjection. We define a translation action $\G^{[r]}\acts (\G, q\circ \rg)$ as follows. Let $\Qc\in \G^{[r]}$ and $\delta\in \G^r$ such that $q(\rg(\delta))=\src(\Qc)$. Write $\delta=(\delta_1,\ldots, \delta_r)$,  $\rg(\delta)=(x_1,\ldots, x_r)$, so that $ \src(\Qc)=\{x_1,\ldots, x_r\}$. For every $i$, there exists a unique $\gamma_i\in \Qc$ such that $s(\gamma_i)=x_i$. Define $\Qc\cdot \delta=(\gamma_1\delta_1,\ldots, \gamma_r\delta_r)$. It is straightforward to verify that this is a translation action. 
  
 \end{example}
 
 We now give a \emph{non}-example of continuous morphism of pseudogroups. 
 \begin{example}For every bisection $T\in \tG$,  the set $T^{[r]}=\{\Qc\in \G^{[r]}\colon \Qc\subset T\}$ is a bisection of $\G^{[r]}$, and thus an element of $\widetilde{\G^{[r]}}$. It is not difficult to check that the map $T\mapsto T^{[r]}$ is a semigroup homomorphism $\tG\to \widetilde{\G^{[r]}}$. 
 However as soon as $r\ge 2$, this homomorphism does not preserve unions of compatible families, and thus is \emph{not} a continuous morphism of pseudogroups. To see this, write $X=U\cup V$ as the union of two open, non-empty proper subsets. Not every subset $\{x_1,\ldots, x_r\}\subset X$ is contained either in $U$ or in $V$, and thus $U^{[r]}\cup V^{[r]}$ is strictly smaller than $X^{[r]}$.

An analogous thing happens with the cartesian power $\G^r$, namely we have a semigroup homomorphism $\tG\to \widetilde{\G^r}, T\mapsto T^r$, which is not a continuous morphism for the same reason.

 \end{example}

\section{Preliminaries  on the alternating full group} \label{s-prel-alt}

In this  section,  we recall the definition of the alternating full group $\Af(\G)$ as defined by Nekrashevych in \cite{Nek-simple}, and collect some technical facts on the formalism of \emph{multisections} developed there.

\medskip

\emph{Throughout the section, $\G$ is an \'etale groupoid over a space $X$ which is assumed to be either a Cantor space, or locally compact, non-compact Cantor space (i.e. a space  homeomorphic to any open, non-closed subset of a Cantor space).}

 \subsection{Expansive groupoids and the group $\Af(\G)$} \label{s-alt}
 \begin{defin}
A \textbf{multisection} of {degree} $d$ of $\G$ is a map $\vec{F}$ from $\{1,\cdots, d\}^2$ to the pseudogroup $\tG$ such that the following conditions are satisfied:
\begin{enumerate}[label=(\roman*)]
\item  for every $i=1\ldots, d$,  the set $\vec{F}(i,i)\subset X$ is a compact open subset of $X$ ;
\item we have $\F(i,i)\cap \F(j, j)=\varnothing$ if $i\neq j$;
\item  for every $i,j,\ell=1,\ldots d$ we have $\vec{F}(i, j)\F(j,\ell)=\F(i,\ell)$. \end{enumerate}

Note that it follows from this definition that $\src(\F(i, j))=\F(j, j)$ and $\rg(\F(i, j))=\F(j, j)$ and that $\F(i, j)$ is a compact open bisection of $\G$. 
The set $\bigcup_{i=1}^d \F(i,i)\subset X$ is called the \textbf{domain} of the multisection and the subsets $\F(i,i)$ are called the \textbf{components} of the domain. 

\end{defin}

A multisection $\F$ of degree $d$ induces an embedding from the symmetric group $\sym(d)$ into ${\Ff(\G)}$, still denoted $\F\colon \sym(d)\to {\Ff(\G)}$. For every $\sigma\in \sym(d)$,  the element $\F(\sigma)$ is given by the bisection
\[ {\F(\sigma)}=\left(\bigcup_{i=1}^d \F(\sigma(i), i)\right)\sqcup  (X\setminus \bigcup_{i=1}^d\F(i, i)).\]
The associated homeomorphism permutes the components $\F(i,i)$ in a way prescribed by  $\sigma$, and acts within each component according to the bisection $\F(\sigma(i),i)$.  We denote $\Af(\F)$ the image of the alternating group $\alt(d)$  under $\F$.
\begin{defin}\label{def-alt}
The  \textbf{alternating full group} of $\G$ is the subgroup ${\Af(\G)}$ of ${\Ff(\G)}$  generated by $\bigcup_{\F} \Af(\F)$, where $\F$ varies over multisections of  degree $d\geq 3$.
\end{defin}

\begin{remark} \label{r:direct-limit}
Note that since we required the component of the domain of a multisection to be compact, the group $\Af(\G)$ acts on $X$ by compactly supported homeomorphisms. In particular, when the unit space $X$ is  the locally compact, non-compact Cantor set, the group ${\Af(\G)}$ is the direct limit of the groups ${\Af({\G|{U}})}$, were $U\subset X$ varies among compact open subsets of $X$ (and hence $U$ is a  Cantor space).
\end{remark}

In \cite{Nek-simple}, Nekrashevych proves the following two theorems. See also \cite{MB-full} for an independent result equivalent to the first (but with a different definition of $\Af(\G)$). 
\begin{thm}[\cite{Nek-simple}]\label{t:simple}\label{t-simple}
 Assume that $\G$ is  a minimal effective \'etale groupoid over a (locally compact) Cantor space  $X$.
 Then every non-trivial subgroup of ${\Ff(\G)}$ normalised by ${\Af(\G)}$ contains ${\Af(\G)}$. In particular, ${\Af(\G)}$ is simple and contained in every non-trivial normal subgroup of ${\Ff(\G)}$.
\end{thm}
\begin{remark}
In \cite{Nek-simple} this result  is stated there only  when  $X$ is a Cantor space, but the case of the locally compact Cantor set reduces to this by Remark \ref{r:direct-limit}. 
\end{remark}
\begin{thm}[\cite{Nek-simple}] \label{t-fg}
Assume that $\G$ is an \emph{expansive} \'etale groupoid over a  Cantor space $X$,  and that every $\G$-orbit contains at least 5 points. Then the alternating full group ${\Af(\G)}$ is finitely generated.
\end{thm}
\begin{defin}[\cite{Nek-simple}] \label{d-expansive}
A compactly generated groupoid $\G$  over a Cantor space $X$ is said to be \textbf{expansive} if it admits a compact generating set $\mathcal{T}\subset \tG$  such that $\bigcup_{r=0}^\infty (\mathcal{T}\cup \mathcal{T}^{-1})^r$ is a basis of the topology of $\G$. Such a set is called an \textbf{expansive generating set} of $\tG$. 
\end{defin}
 The groupoid of germs of an action of a finitely generated group on a Cantor space is expansive if and only if the action is expansive, if and only if it is conjugate to a subshift, see \cite[Prop. 5.5]{Nek-simple}.

\subsection{General facts on multisections and the group $\Af(\G)$} 
We now list some useful elementary facts about multisections and the group $\Af(\G)$, that will be used in the sequel.  We use the following terminology.

\begin{defin}
A \textbf{multigerm} of degree $d$ of $\G$ is  an injective map $\gk\colon \{1,\ldots , d\}^2\to \G$ such that $\gk(i, i)\in X$ for every $i=1,\ldots, d$ and for every $i,j,\ell=1,\ldots d$ we have $\gk(i, j)\gk(j, \ell)=\gk(i, \ell)$.
\end{defin}
Note that it follows that $\src(\gk(i, j))=\gk(j, j)$ and $\rg(\gk(i, j))=\gk(i, i)$, and that $\gk(i, i), i=1\ldots , d$ all lie in the same orbit.

If $\gk$ is a multigerm and $\F$ a multisection of the same degree, we will write $\gk\in \F$ to mean that $\gk(i, j)\in \F(i, j)$ for every $i, j$.

\begin{lemma} \label{l-multigerm} 
Let $\gk$ be a degree $d$ multigerm.
\begin{enumerate}[label=(\roman*)]

\item  \label{i:multigerm1} Assume $x_{d+1},\ldots, x_{d'}$ for $d'>d$ are  units in the same orbit of $\gk(1,1)$ such that the points  $\gk(1,1),\ldots , \gk(d,d), x_{d+1}, \ldots , x_{d'}$ are pairwise distinct. Then there exists a degree $d'$ multigerm $\gk'$ such that $\gk'|_{\{1,\ldots, d\}^2}=\gk$  and $\gk'(i, i)=x_i$ for $i=d+1,\ldots d'$. 
\item \label{i:multigerm2}   Assume that $T_{i, j}\in \tG$ is a neighbourhood  of $\gk(i, j)$ for every $i, j=1\ldots d$. Then there exists a degree $d$ multisection $\F$ such that $\gk\in \F$ and $\F(i, j)\subset T_{i,j}$ for every $i,j$. \end{enumerate}
\end{lemma}
Part \ref{i:multigerm2} allows as to talk about ``sufficiently small'' multisections containing $\gk$, and we will sometimes do so without mention.

\begin{proof}
To prove (i),  set set $\gk'(i, j)=\gk(i, j)$ for $1\le i, j\le d$. Choose arbitrarily for  every $d+1 \le j\le d'$ an element $\gk'(j, 1)\in \G$ such that $\src(\gk'(j, 1))=\gk(1, 1)$ and $\rg(\gk'(j, 1))=x_j$, and set $\gk'(1, j)=\gk'(j, 1)$. In the remaining cases, let $\gk'(i, j)=\gk'(i,  1)\gk'(1, j)$. 

Let us prove (ii).  First, for every $i, j$ we can assume that $T_{i, j}=T_{j, i}^{-1}$, upon  replacing all the $T_{i, j}$  by $T_{i, j}\cap T^{-1}_{j, i}$. Choose a decreasing basis $(W_n)$ of compact open neighbourhoods of $\gk(1, 1)$ in $X$. If $n$ is large enough, we can assume $W_n\subset T_{1, 1}$ , that  $W_n\subset \src(T_{j, 1})$ for every $j=2, \ldots, d$, and that $W_n, \tau(T_{2,1})(W_n), \ldots, \tau(T_{d, 1})( W_n)$ are pairwise disjoint (recall that $\tau(T)$ denotes the homeomorphism associated to a bisection).  Define a sequence of multisections $\F_n$ by setting $\F_n(j, 1)=T_{j, 1}W_n$ and $\F_n(1, j)=\F_n(j, 1)^{-1}$ for $j=1, \ldots d$ and $\F_n(i, j)= \F_n(i, 1)\F_n(1, j)$ in the remaining cases. Then for every $i, j$ the sequence of bisections $\F_n(i, j)$ is a basis of neighbourhoods of $\gk(i, j)$, and it follows that $\F_n(i, j)\subset T_{i, j}$ for all $i, j$ if $n$ is large enough.  \qedhere
\end{proof}

\begin{lemma}
Assume that every $\G$-orbit contains at least $d$ points for some $d\geq 2$. Then for every compact open  bisection $T\in \tG$ such that $\src(T)\cap \rg(T)=\varnothing$, one can find finitely many degree $d$ multisections $\F_1, \ldots ,\F_n$ such that $T=\cup_{i=1}^n \F_i(1, 2)$. 
\end{lemma}
\begin{proof}
Using Lemma \ref{l-multigerm} and the fact that every $\G$ orbit has at least $d$ points, or every $\gamma\in T$, we can find a degree $d$ multigerm $\gk$ such that $\gk(1, 2)=\gamma$, and a degree $d$ multisection $\F\ni \gk$ such that $\F(1, 2)\subset T$. The conclusion follows by compactness.
\end{proof}

The next proposition is \cite[Proposition 3.2]{Nek-simple}. If $\F, \vec{T}$ are multisections of the same degree, we write $\F\subset \vec{T}$ if  $\F(i, j)\subset \vec{T}(i,j)$ for every $i, j$.

\begin{prop}\label{p-nekr-cover}
Let $\F$ be a degree $d$ multisection. Let $\F_\ell\subset \F$ for  $\ell=1,\ldots n$ be multisections such that $\F(i, j)=\bigcup_{\ell=1}^n \F_\ell(i, j)$ for every $i, j\in \{1, \ldots, d\}$. Then $A(\F)$ is contained in the subgroup generated by $A(\F_1)\cup \cdots \cup A(\F_n)$. 
\end{prop}

\begin{lemma}\label{l-proximal}
If $\G$ is minimal, the action of ${\Af(\G)}$ on $X$ is proximal.
\end{lemma}
\begin{proof}
Let $x, y\in X$ and $U$ be an open subset. Using Lemma \ref{l-multigerm} we can find multigerms $\gk, \vec{\delta}$ of degree 3 such that $\gk(1, 1)=x, \vec{\delta}(1, 1)=y$ and $\gk(2, 2), \vec{\delta}(2, 2)\in U$ and the points $\gk(i, i),\vec{\delta}(j, j), i,j=1,\ldots 3$ are pairwise disjoint. For any sufficiently small multisections $\F\ni \gk$ and $\Tk\ni \vec{\delta}$ the element $g=\F((123))\Tk((123))\in {\Af(\G)}$ is such that $g(\{x, y\})\in U$. \qedhere \end{proof}
By a well-known consequence of proximality  \cite[Lemma 3.3]{Gla-book}, we deduce:
\begin{lemma}\label{l-trivial-centraliser}
If $\G$ is minimal, the only continuous $\Af(\G)$ equivariant map $X\to X$ is the identity. In particular $\Af(\G)$ has trivial centraliser in $\Ff(\G))$.
\end{lemma}

\begin{lemma}\label{l-qi-graph}
Assume that $\G$ is a compactly generated \'etale groupoid over a Cantor space $X$ with no orbit of size less than 3. Let $\T\subset \tG$ be a finite generating set of compact open bisections. 
Then there exists a finitely generated subgroup $G\le {\Af(\G)}$ such that for every $x\in X$  the orbital graph $\orb_x(\G, \T)$, is bi-Lipschitz equivalent to the the orbital graph  $\sch_x(G, S)$ of the action $G\acts X$, where $S$ is any finite symmetric generating set  of $G$. Moreover,   the constants in the bi-Lipschitz equivalence are independent of $x\in X$. 
\end{lemma}
Note that if $\G$ is expansive, one can simply choose $G=\Af(\G)$.
\begin{proof}
 We can assume that the  generating set $\T=\{T_1,\ldots, T_d\}$ is such that $\rg(T_i)\cap \src(T_i)=\varnothing$ for every $i=1\ldots d$, see \cite[Lemma 5.1]{Nek-simple}.   Using Lemma \ref{l-multigerm}, choose for every $i=1, \ldots, d$  degree 3 multisections $\F_{i, 1},\ldots \F_{i, {k_i}}$ such that the  sets $\F_{i,j}(1,2), j=1\ldots k_i$ cover $T_i$. Let $G$ be the group generated by $S= \bigcup_{i=1}^d\bigcup_{j=1}^{k_i} \Af(\F_{i, j})$. By construction every $y, z$ in the $\G$-orbit of $x$ that are neighbours in $\orb_x(\G, \T)$ are also neighbours in $\sch_x(G, S)$. Conversely it is easy to see that the identity map on vertices $\sch_x(G, S)\to \orb_x(\G, \T)$ is Lipschitz (for any finitely generated subgroup $G<{\Af(\G)}$ and any finite generating set $S$). 
 \qedhere \end{proof}
For later use we observe the following (the definition of asymptotic dimension was recalled in \S \ref{s-coarse}).

\begin{prop}\label{p-asdim-alt}
For every minimal compactly generated \'etale groupoid $\G$ over a Cantor space, the group $\Af(\G)$ has infinite asymptotic dimension.
\end{prop}
\begin{proof}
Let $U_1,\ldots, U_n \subset X$ be disjoint clopen subsets. The groupoid $\G|{U_i}$ is also compactly generated (see Proposition \ref{p-compact-generation-stable}) and minimal.  For every $i=1,\ldots, n$   find a finitely generated subgroup $K_i<A(\G|{U_i})$ as in Lemma \ref{l-qi-graph}. Since $K_i$ admits Schreier graphs, in particular it is itself infinite.  Let $K=K_1\times \cdots \times K_n\le \Af(\G)$. Since every $K_i$ is infinite, its Cayley graph contains copy of  $\Z$, and therefore there is an injective coarse map $\Z^n\to K$.  It is well-known that $\asdim(\Z^n)=n$ \cite[\S 9.2]{Roe-coarse}.  Since $n$ is arbitrary, Proposition \ref{p-asdim} implies the conclusion.
\end{proof}

\section{Confined subgroups of full groups}
\label{s-confined}\label{s-Chabauty}
The following result characterises the confined subgroups of $\Ff(\G)$ and $\Af(\G)$ when $\G$ is a minimal effective groupoid over a Cantor space.    This will be a key tool used throughout the paper. 

All relevant notations on the various stabilisers for group actions ($\st^0_G(x), \st_G(x)$, etc..) can be found in  \S \ref{s-preliminaries}.

\begin{thm}[Classification of the confined subgroups]\label{t-Chabauty}\label{t-confined}
 Assume that $\mathcal{G}$ is a minimal effective \'etale groupoid over a space  $X$ which is either a Cantor space or a locally compact Cantor space. The following statements on  a subgroup $H\in \sub(\Ff(\G))$ are equivalent:
\begin{enumerate}[label=(\roman*)]
\item \label{i:conjugacy}  the subgroup $H$ is confined;
\item \label{i:AG-conjugacy} the subgroup $H$ is confined by $\Af(\G)$; 
 \item \label{i:alternating} there exists a unique finite subset ${Q}\subset X$ (possibly empty) such that we have \[\st^0_{\Af(\G)}(Q)\le H \le \st_{\Ff(\G)}({Q}).\]
 \end{enumerate}
 Moreover let $H$ satisfy one of these equivalent conditions, and let $P\subset \Ff(\G)\setminus \{1\}$ be a confining set. Then the set $Q$ in \ref{i:alternating} verifies $|Q|\le |P|-1$. 
\end{thm}

Let us first prove prove some lemmas.

\begin{lemma}\label{l-st-finite}
For every finite subset $Q\subset X$, we have $\Af(\G|Q^c)=\st^0_{\Af(\G)}(Q)$. 
\end{lemma}

In the proof we will consider the natural action of the  group $\Ff(\G)$  on the set of multisections by $(g\F)(i, j)=g\F(i, j)g^{-1}$. Similarly we let $\Ff(\G)$ act on the set of multigerms by $(g  \gk)(i, j)=[g]_{\gk(i, i)}\gk(i, j)[g]_{\gk(j,j)}^{-1}$.

\begin{proof}
The inclusion $\Af(\G|Q^c)\le \st^0_{\Af(\G)}(Q)$ is clear, since $\Af(\G|Q^c)$ is generated by elements whose support is compact and contained in $Q^c$. Let $g\in \st^0_{\Af(\G)}(Q)$ and let us show that $g\in \Af(\G|Q^c)$. Since $g\in \Af(\G)$, there exist degree 3 multisections  $\F_1,\ldots, \F_n$  of $\G$ and elements $s_i\in \Af(\F_i)$  such that $g=s_n\cdots s_1$. Note that we can assume that  $s_i$ is of the form $s_i=\F_i((123))$ for every $i=1,\ldots, n$. Let $V$ be a  neighbourhood of $Q$ such that $g$ fixes $V$ point-wise. There are three cases to consider.

\emph{Case 1.} Assume that none of the domains of the multisections $\F_i$ intersects $Q$. In this case we have $\Af(\F_i)\le \Af(\G|Q^c)$ for every $i=1,\ldots n$ and the conclusions follows immediately.

\emph{Case 2.}  Assume that the union of the domains of the multisections $\F_i$ does not cover the whole unit space $X$. Let $U\subset X$ be a clopen set which does not intersect the domain of any of the multisections $\F_i$, and note that $g$ fixes $U$ point-wise. Using Lemma \ref{l-multigerm} and minimality of $\G$,  we can find an element $h\in \Af(\G)$ such that $h(Q)\subset U$ and such that $h$ fixes point-wise the complement of $V\cup U$. In particular, $h$ commutes with $g$ and hence $g=h^{-1}gh=(h^{-1}s_nh)\cdots (h^{-1}s_1h)\in \langle \Af(h^{-1}\F_1), \ldots , \Af(h^{-1} \F_n)\rangle$. Moreover $Q\subset h^{-1}(U)$ and the set $h^{-1}U$ avoids the domain of all the multisections $h^{-1}\F_i, i=1,\ldots, n$. Hence we are reduced to Case 1.

\emph{Case 3.} Assume that the union of the domains of the multisections $\F_i$ covers $X$. Define inductively sets $Q_0, \ldots, Q_n$ by $Q_0=Q$ and $Q_{i+1}=Q_i\cup s_{i+1}(Q_i)\cup s_{i+1}^2 (Q_i)$.  Let $W_0$ be a clopen neighbourhood of $Q$ contained in $V$, and define inductively sets $W_i, i=1,\ldots n$ by $W_{i+1}=W_i\cup s_{i+1}(W_i)\cup s_{i+1}^2(W_i)$. If $W_0$ shrinks to $Q$ we have that $W_n$ shrinks to $Q_n$, and therefore we may assume that $W_n$ is a proper subset of $X$. For every $i=1,\ldots, n$  let $\F'_i\subset \F_i$ be the multisection consisting of all multigerms $\gk\in \F_i$ with $\gk(1, 1)\in W_i$, let $\F''_i\subset \F_i$ be its complement, and set $s'_i=\F'_i((123)), s''_i=\F''_i((123))$ (if $\F'_i$ or $\F''_i$ are empty, set $s'_i=1$ or $s''_i=1$ accordingly). Observe that $s'_i$ and $s''_i$ are commuting elements such that $s_i=s'_is''_i$.  Set $g'=s'_n\cdots s'_1$ and $g''=s''_n\cdots s''_1$. Since the domain of $\F''_i$ avoids $W_i\supset Q$ we have  $g''\in \Af(\G|Q^c)$. Moreover for $i<j$ we have that each $s'_i$ is supported in $W_i\subset W_j$ and $s''_j$ is supported in the complement of $W_j$, and therefore $s'_i$ and $s''_j$ commute. We deduce that $g=g'g''$. Since for every $i$ the domain of $\F'_i$ is contained in $W_i\subset W_n$ which is a  proper subset of $X$, we have $g'\in \Af(\G|Q^c)$ by Case 2. Therefore $g\in \Af(\G|Q^c)$. \qedhere

\end{proof}

\begin{lemma}\label{l-alt-elementary}
Let $d\geq 1$. Then the alternating group $\alt(3d)$ is generated by $\alt(\{d+1,\ldots 3d\})\cup \left(\bigcup_{i=1}^d\alt(\{i, d+i, 2d+i\})\right)$.
\end{lemma}
The proof is elementary and we omit it. 

\qedhere
\begin{proof}[Proof of Theorem \ref{t-Chabauty}] To see that \ref{i:alternating}$\Rightarrow$\ref{i:conjugacy}, it is enough to check that for every finite set $\Q\subset X$, the subgroup $\st^0_{\Af(\G)}(\Q)$ is confined. Let $r=|\Q|$ and choose $r+1$ degree 3 multisections   with disjoint domains $\F_1,\ldots , \F_{r+1}$ (for example using Lemma \ref{l-multigerm}). For every $g\in {\Ff(\G)}$, there exists $i$ such that the set $g(\Q)$ does not intersect the domains of $\F_i$, and therefore $A(\F_i)\le g\st^0_{\Af(\G)}(\Q)g^{-1}$. For every $i=1,\ldots, r+1$ choose a non-trivial element $g_i\in A(\F_i)$. Then the set  $P=\{g_1,\ldots , g_{r+1}\}$ is confining. 

The implication \ref{i:conjugacy}$\Rightarrow$\ref{i:AG-conjugacy} is obvious. 

It remains to be shown  that \ref{i:AG-conjugacy}$\Rightarrow$\ref{i:alternating}, which is the main part of the statement.
 To this end, let $H\le {\Ff(\G)}$ be confined by $\Af(\G)$ and  fix  a confining set $P\subset G\setminus \{1\}$. Write $r=|P|$. 
We divide the proof into a sequence of steps. 
 \begin{step}\label{s-highly-minimal}
 Let $D_1\ldots D_r\subset X$ be disjoint finite subsets, and let $V\subset X$ be an open set. Then there exists $h\in H$ and $j=1, \ldots r$ such that $h(D_j)\subset V$.
 \end{step}

 \begin{proof} Let $g_1,\ldots, g_r$ be the elements of $P$. Since $aPa^{-1}\cap H \neq \varnothing$ for every $a\in {\Af(\G)}$, it is enough to observe that there exists an element $a\in {\Af(\G)}$ such that $ag_ia^{-1}(D_i)\subset V$ for every $i=1,\ldots, r$. We now explain how to construct such an element. Let $U_1,\ldots, U_r\subset X$ be open sets such that $U_1,\ldots, U_r, g_1(U_1), \ldots, g_r(U_r)$ are pairwise disjoint (such sets exist for any finite family of non-trivial homeomorphisms of $X$,  see e.g. \cite[Lemma 3.1]{LB-MB-subdyn}). Set $D=D_1\cup \cdots \cup D_r$. Since $\mathcal{\G}$ is minimal, for every $k=1,\ldots, r$ and every $x\in {D}_k$ we can find an element ${\gamma}_x$ such that $x=\src(\gamma_x)\neq \rg(\gamma_x)\in U_k \setminus \left(D \cup g_1^{-1}(D)\cup \cdots \cup g_r^{-1}(D)\right)$. Moreover, we can find such elements in such a way that the points $\rg(\gamma_x), x\in D_k, k=1,\ldots r$ are pairwise distinct (for instance, this can be done by ordering all points of $D$ in a list and choosing each $\gamma_x$ at a time in this order; note that at each step only finitely many new points are forbidden for $\rg(\gamma_x)$). For every $k=1,\ldots, r$ let ${D}'_k=\{\rg(\gamma_x)\colon \in D_k\}\subset U_k$, and set $D'=D'_1\cup\cdots \cup D'_r$. Let also $D''_k=g_k({D}'_k)\subset g_k(U_k)$ and $D''=D''_1\cup\cdots \cup D''_r$. Note that, by the choices made, the sets $D, D', D''$ are pairwise disjoint.  Proceeding in a similar way, choose for every $x\in D''$ an element $\gamma_{x} \in \G$ such that $x=\src(\gamma_{x})\neq \rg(\gamma_{x})\in V\setminus \left(D\cup D'\cup D''\right)$, and the points $\rg(\gamma_{x})$ for $ x\in D''$ are pairwise distinct. For $k=1,\ldots, r$  let $D'''_k=\{  \rg(\gamma_{x}), x\in D''_k\}$ and $D'''=D'''_1\cup\cdots \cup D'''_k$. Finally using Lemma \ref{l-multigerm} find for every $x\in D\cup D''$ a degree 3 multigerm $\gk_x$ such that $\gk_x(1, 2)=\gamma_x$ for every $x\in D\cup D''$ and the units $\gk_x(3, 3)$ for $ x\in D\cup  D''$ are pairwise distinct and do not belong to $D\cup D'\cup D''\cup D'''$. Using Lemma \ref{l-multigerm} again, we can find for every $x\in D\cup D''$ sufficiently small multisections $\F_x\ni \gk_x$ that have pairwise disjoint domains. Let $\sigma=(123)\in \alt(3)$ and let $a=(\prod_{x\in D}\F_x(\sigma^{-1}))(\prod_{x\in D''} \F_x(\sigma))$. For every $k=1\ldots r$ we have $a^{-1}(D_k)=D'_k$ and $a(D''_k)=D_k'''$ and therefore $ag_ka^{-1}(D_k)=D'''_k\subset V$. This concludes the proof of the Step \ref{s-highly-minimal}. \qedhere
 \end{proof}

\begin{step} \label{s-setC}
Among  closed $H$-invariant proper subsets of $X$ there is  a unique maximal one with respect to inclusion (possibly empty), denoted $Q$. The set $Q$  is finite and verifies $|Q|\le r-1$.  This will be the set in the statement of the theorem. \end{step}

\begin{proof} Let $Z\subset X$ be a proper closed $H$-invariant subset. Assume by contradiction that $x_1,\ldots, x_r$ are distinct points in $Z$. By Step \ref{s-highly-minimal} applied to the sets $D_i=\{x_i\}, i=1,\ldots r$, and to the open set $V=X\setminus Z$, there exists $j$ and $h\in H$ such that $h(x_j)\notin Z$, contradicting the invariance of $Z$. It follows that every closed $H$-invariant proper subset $Z\subset X$ is finite and satisfies $|Z|\le r-1$. Let $(Z_n)_{n\geq 0}$ be a sequence of closed $H$-invariant proper subsets. Since for every $m\geq 0$ the union $\bigcup_{n=0}^m Z_n$ is also a proper closed invariant subset, it has cardinality at most $r-1$, and it follows that it has to stabilise for $m$ large enough and therefore the sequence $Z_n$ takes only finitely many values. Therefore there are only finitely many closed $H$-invariant proper subsets and their union $Q$ verifies the conclusion.\qedhere
\end{proof}
 \begin{step}  \label{s:LB-MB-subdyn} There exists a clopen subset $U\subset X$ such that $H$ contains ${\Af({\G|{U}})}$. \end{step}
 \begin{proof}
Since the action of ${\Af(\G)}$ on $X$ is minimal and proximal (Lemma \ref{l-proximal}), we can apply Theorem \ref{t-proximal}. We obtain an open set $U\subset X$  that $H$ contains $[\rist_{\Af(\G)}(U), \rist_{\Af(\G)}(U)]$.  Since ${\Af({\G|{U}})}<\rist_{{\Af(\G)}}(U)$ is infinite and simple, it is contained in $[\rist_{\Af(\G)}(U), \rist_{\Af(\G)}(U)]$. \qedhere \end{proof}

\begin{step} \label{s-small}
Let $d\geq 3$ and let $\gk$ be a degree $d$ multigerm of $\G|{Q^c}$. Then for every multisection $\Tk\ni \gk$  there exists a multisection $\F$ with $\gk\in \F\subset \Tk$ and  $A(\F)\le H$.  
\end{step}
\begin{proof}
First observe that if there exists an element $h\in H$ such that $h\gk\in \G|U$, then for every sufficiently small bisection $\F\ni \gk$ we have that $h\F$ is a multisection of $\G|U$ and $hA(\F)h^{-1}=A(h\F)\le H$ by Step \ref{s:LB-MB-subdyn}. 
In general, since for every $i=1,\ldots, d$ the unit $\gk(i, i)$ does not belong to $Q$, it has a dense $H$-orbit by the definition of $Q$. Therefore there exist elements $h_1,\ldots, h_d$ such that $h_i\gk(i, i)\in U$ for revery $i=1,\ldots, d$. For every $i=1,\ldots d$ let $V_i$ be a neighbourhood of $\gk(i, i)$ such that $h_i(V_i)\subset U$. By minimality of $\G$ we can choose disjoint sets $D_1,\ldots D_r$ that consist of points that lye in the same orbit of $\gk(1, 1),  \ldots, \gk(d, d)$ and are disjoint from these and from $Q$, and such that $|D_j\cap V_i|=2$ for every $j=1,\ldots r$ and every $i=1,\ldots, d$. In particular $|D_j|=2d$ for every $j=1,\ldots r$. By Step \ref{s-highly-minimal}, there exists $k\in H$ and $\ell=1,\ldots r$ such that $kD_\ell\subset U$. Fix such $\ell$ and $h$ and set $D_\ell=\{x_{d+1},\ldots, x_{3d}\}$ where the numbering is chosen in such a way that $x_{d+i}, x_{d+2i}\in V_i$ for every $i=1,\ldots d$. By Lemma \ref{l-multigerm} we can find a degree $3d$ multigerm $\gk'$ such that $\gk'|_{\{1,\ldots, d\}^2}=\gk$ and $\gk'(j,j)=x_j$ for $j= d+1,\ldots , 3d$. Observe that we have $k\gk'|_{\{d+1,\ldots 3d\}^2}\in \G|{U}$ and $h_i\gk'|_{\{i, d+i, 2d+i\}^2}\in \G|{U}$ for every $i=1,\ldots, d$. It follows that if $\F'$ is any sufficiently small multisection containing $\gk'$, we have $A(\F'|_{\{d+1,\ldots, 3d\}^2})\le H$ and $A(\F'|_{\{i, d+i, 2d+i\}})\le H$ for every $i=1,\ldots d$. By Lemma \ref{l-alt-elementary}, we deduce that $A(\F')\le H$. Setting $\F=\F'|_{\{1,\ldots d\}^2}$ we therefore have $A(\F)\le H$. By Lemma \ref{l-multigerm} we can choose $\F'$ sufficiently small so  that $\F\subset \Tk$ as desired. \qedhere 
\end{proof}
\begin{step}\label{s-contains}
$H$ contains $\st^0_{\Af(\G)}(Q)$.
\end{step}
\begin{proof}
Let $\Tk$ be a degree $d\geq 3$ a multisection of $\G|{Q^c}$. By Step \ref{s-small} and by compactness of $\Tk$ we can cover $\Tk$ with finitely many degree $d$ multisections $\F$ such that $\Af(\F)\le H$. By Proposition \ref{p-nekr-cover}, we have $\Af(\Tk)\le H$. It follows that ${\Af(\G|\Q^c)}\le H$ since $\Tk$ was arbitrary. The conclusion follows from Lemma \ref{l-st-finite}. \qedhere
\end{proof}
\begin{step}
The set $Q$ satisfies  \ref{i:alternating} in the statement of the theorem.
\end{step}
\begin{proof}
The inclusion $H\le \st_{\Ff(\G)}(Q)$ follows by the definition of $Q$ in Step \ref{s-setC}, and the inclusion $\st^0_{\Af(\G)}(Q)\le H$ was proven in Step \ref{s-contains}.  Only uniqueness is left to prove. Assume that $\Q_i, i=1,2$ satisfy the condition. In particular we have $\st^0_{\Af(\G)}(Q_1)\le H\le  \st_{{\Ff(\G)}}(\Q_2)$. Note that every $x\notin \Q_1$ has an infinite dense orbit under the action of $\st^0_{\Af(\G)}(\Q_1)$ and therefore we deduce that $\Q_2\subset \Q_1$. Exchanging the roles, $\Q_1=\Q_2$. \qedhere
\end{proof}
Finally note that the set $Q$ verifies $|Q|\le |P|-1$ by its definition in   Step \ref{s-setC} (recall that $r=|P|$). This concludes the proof of the theorem. \qedhere
\end{proof}
In practice, we will use the following variant of  Theorem \ref{t-Chabauty}.
\begin{cor} \label{c-Chabauty-G} \label{c-Chabauty}
Let $G$ be a group such that $\Af(\G)\le G\le \Ff(\G)$. A subgroup $H\in \sub(G)$ is confined if and only if there exists a unique finite subset $\Q\subset X$ such that $\st^0_{\Af(\G)}(Q)\le H\le \st_G(\Q)$. Moreover we have $|\Q|\le |P|-1$, where $P\subset G \setminus \{1\}$ is a confining set for $H$. 
\end{cor}
\begin{proof}
If $H\in \sub(G)$ is confined, then it is confined by $\Af(\G)$ and hence the theorem applies. \qedhere
\end{proof}
\subsection{Case of uniformly recurrent subgroup}

Theorem \ref{t-confined} provides the following classification of the uniformly recurrent subgroups of the group $\Af(\G)$ (note that it will not be used elsewhere in the paper).
We say that a uniformly recurrent subgroup of a group $G$ is \textbf{trivial} if it is $\{\{1\}\}$ or $\{G\}$.

 \begin{cor}[Classification of URS's of the group $\Af(\G)$.]
 \label{c-URS}\begin{enumerate}[label=(\roman*)]
\item If $X$ is a Cantor space, the only non-trivial URS's of the group $\Af(\G)$ is the stabiliser URS associated to the action $\Af(\G)\acts X$. 

\item If $X$ is a locally compact, non-compact Cantor space, then the group $\Af(\G)$ has no non-trivial URS's. 

\end{enumerate}
\end{cor}

\begin{proof}
Let $Y\subset \sub(\Af(\G))$ be a uniformly recurrent subgroup with $Y\neq \{\{1\}\}$, and let $H\in Y$. Since $H$ is confined, there exists $Q$ such that $\st^0_{\Af(\G)}(Q)\le H\le \st_G(\Q)$. Let $\mathcal{P}_r(X)$ be the compact space of finite subsets of $X$ of cardinality at most $r$. If $X$ is not compact,  we can find a sequence $(g_n)\in \Af(\G)$ such that $g_n(Q)\to \varnothing$ in $\mathcal{P}_r(X)$. Let $K\in Y$ be any cluster point of $g_n Hg_n^{-1}$. By the lower semicontinuity of $Q\mapsto \st^0_{\Af(\G)}(Q)$ (Lemma \ref{l-semicontinuous}), we obtain that $K$ contains $\st^0_{\Af(\G)}(\varnothing)=\Af(\G)$ and hence $Y=\{K\}=\{\Af(\G)\}$, a contradiction. Otherwise, by proximality and minimality we can  find for every $x\in X$, a sequence $(g_n)$ as above such that $g_n(Q)\to \{x\}$. Using again semi-continuity we get that every cluster point $K$ of $g_n H g_n^{-1}$ satisfies  $ \st^0_{\Af(\G)}(x)\le K \le  \st_{\Af(\G)}(x)$. Choosing $x$ to be a regular point, we get $K=\st_{\Af(\G)}(x)\in Y$, and therefore by minimality $Y$ must coincide with the stabiliser URS associated to $\Af(\G)\acts X$. \qedhere
\end{proof}

\subsection{Example: AF-groupoids and  confined subgroups of the  LDA groups} \label{s-AF}
To give an example of our classification of the confined subgroup, let us consider the case of   \emph{AF-groupoids} associated to a \emph{Bratteli diagram}. These are an important class of \'etale groupoids whose full groups form  a well-studied class of simple locally finite groups \cite{Lei-Pu-LDA, Lav-Nek-LDA, Thomas}.
We point out to the reader that the reading of this subsection is not necessary to read the rest of the paper.

\begin{defin} A \textbf{Bratteli diagram} is a graph $B=(V, E)$, possibly with multiple edges, together with a map $d\colon  V\to \N$, called the \textbf{labelling}, such that the following hold

\begin{enumerate}
\item there are partitions $V=\sqcup_{i\geq 1} V_i$ and $E=\sqcup_{i\geq 1} E_i$ of the set of vertices and of edges into non-empty finite subsets, and we have $|E_i|\geq 2$ for every $i$;
\item  for every $i\geq 0$ there are maps $o\colon E_i\to V_i$ and and $t\colon E_i\to V_{i+1}$, such that every edge $e\in E_i$ has endpoints  $o(e)\in V_i$ and $t(e)\in V_{i+1}$;  
\item for every $v\in V_i, i\geq 0$ there exists at least an edge $e\in E_i$  with $o(e)=v$;
\item \label{i:labeling} for every $v\in V$, we have $k(v):=d(v)- \sum_{e\colon t(e)=v}(d(o(e))\geq 0$;
\item we have $d(v)\geq 1$ if $v\in V_1$.
\end{enumerate}
 \end{defin}
 
 The Bratteli diagram $B$ is said to be \textbf{simple} if for every $i$ there exists $j>i$ such that every vertex of $V_i$ is connected to every vertex of $V_j$. It is said to be \textbf{unital} if $k(v)=0$ for all but finitely many  vertices $v\in V$.

 For every vertex $v\in V$, choose a finite alphabet $A^v=\{a_1^v, a_2^v\ldots, a_{k(v)}^v\}$ such that $A^v\cap A^w=\varnothing$ if $v\neq w$ (if $k(v)=0$, set $A^v=\varnothing$).
 Define the \textbf{path space} $X_B$ of the Bratteli diagram $B$ to be the set of infinite formal words  of the form $a_j^ve_ie_{i+1}e_{i+2}\ldots$ where   $v\in V$, $1\le j\le k(v)$, $o(e_i)=v$ and $t(e_\ell)=o(e_{\ell+1})$  for every $\ell\geq i$.  The space $X_B$ is endowed with the natural product topology, and it is compact if and only if $B$ is unital.
 A \textbf{finite path} is a finite word $\gamma$ which is a prefix of an element of $X_B$. If $\gamma$ is a finite  path of the form $\gamma=a_j^ve_ie_{i+1}\cdots e_k$, its \textbf{terminal vertex is} the vertex $t(\gamma)=t(e_k)\in V_k$, and if $\gamma$ is of the form $a_j^v$ we set $t(\gamma)=v$. Let $\Delta_v$ be the set of finite paths that have $v$ as terminal point. 
 Given a finite path $\gamma$, we denote by $U_\gamma\subset X_B$ the set of  paths that have $\gamma$ as a prefix. These sets form a basis of clopen sets for the topology on $X_B$. 
 
 Fix a vertex $v$. For every $\gamma_1, \gamma_2\in \Delta_v$, we have a natural  homeomorphism $\tau_{\gamma_1, \gamma_2}\colon U_{\gamma_1}\to U_{\gamma_2}$ given by $\gamma_1e_ke_{k+1}\cdots \mapsto \gamma_2e_k e_{k+1}\cdots$. These maps generate a pseudogroup of partial homeomorphisms of $X_B$. The corresponding groupoid of germs, with it natural \'etale topology, will be denoted $\mathcal{G}_B$. The groupoid $\G_B$ is minimal if and only if the diagram $B$ is simple. 
 \begin{defin}
 An \'etale groupoid is called an \textbf{AF-groupoid} if it is of the form $\G_B$ for some simple Bratteli diagram $B$. 
 \end{defin} 
 
If $\G_B$  is an AF-groupoid,  the group $\Af(\G_B)$ is a simple locally finite group, that can be expressed as the ascending union $\bigcup_{i\geq 1}\prod _{v\in V_i} \alt(\Delta_v)$ of finite alternating groups, where each alternating group $\alt(\Delta_v)$ acts on $U_v=\sqcup_{\gamma\in \Delta_v} U_\gamma$ by permuting initial finite prefixes, and acting trivially outside $U_v$. In fact, in this case the group $\Af(\G)$ coincides with the derived subgroup of $\Df(\G)$ of the topological full group. For more details on   topological full groups of AF-groupoids, see \cite[\S 3]{Mat-simple}.

By a result of Lavrenyuk--Nekrashevych \cite{Lav-Nek-LDA}, the class of simple locally finite groups arising in this way are precisely the \textbf{strongly diagonal limits} of products of alternating groups, or \textbf{LDA-groups} (with the single exception of the infinite finitary alternating group $\alt_f(\N)$, which is an LDA-group but does not arise in this way).

Confined subgroups of simple locally finite groups have been extensively studied \cite{Har-Zal, Lei-Pug-diag, Lei-Pug-finlin} . Confined subgroups of the LDA-groups were studied by  Leinen--Puglisi in  \cite{Lei-Pug-diag}, who provided a characterisation of the confined subgroups in the special case of diagonal limits of alternating groups (the case of $\alt_f(\N)$ had ben previously treated in \cite{Se-Za-alt}).  Theorem \ref{t-Chabauty} extends this characterisation to all the LDA-groups. 
  \begin{cor} \label{t-LDA} Let $G$ be an $LDA$-group, with $G\neq \alt_f(\N)$, and let $B$ be a simple Bratteli diagram such that $G\simeq \Af({\G_B})$. A subgroup $H<G$ is  confined if and only if $H$ contains the point-wise fixator of a finite set $\Q\subset X_B$ as a subgroup of finite index.  
\end{cor}
\begin{proof}
By Corollary \ref{c-Chabauty}, $H$ is confined if and only if there exists a unique set $Q$ such that $\st^0_{\Af(\G_B)}(Q)\le H\le \st_{\Af(\G_B)}(Q)$. But an AF-groupoid is principal, hence  $\st^0_{\Af(\G_B)}(Q)=\fix_{\Af(\G_B)}(Q)$ and the latter has finite index in $\st_{\Af(\G_B)}(Q)$. 
\end{proof}

Corollary \ref{c-URS} recovers a recent result of Thomas  \cite[Theorem 2.8]{Thomas}.
\begin{cor}
Let $B$ be a simple Bratteli diagram, and assume that $Y\subset \sub(\Af(\G_B))$ is a non-trivial uniformly recurrent subgroup. Then $B$ is unital, and $Y=\{\st_{\Af_{\G_B}}(x)\colon x\in X_B\}$. 
\end{cor}

\section{The Extension Theorem}  \label{s-dichotomy}\label{s-homomorphisms}

\label{s-Extension Theorem}

In this section we state and prove the main Extension Theorem (Theorem \ref{t-main} below),  and discuss the first corollaries of it.

\subsection{Statement and first corollaries} 


 Let $\G$ be a minimal effective groupoid over the Cantor space, and $r\ge 1$ be an integer. We refer to \S \ref{s-reduced} for the definition of the symmetric power $\G^{[r]}$ and its basic properties. The natural action of $\Ff(\G)$ on $X^{[r]}$ has germs in $\G^[r]$, and thus provides an embedding $\Ff(\G)\hookrightarrow \Ff(\G^{[r]})$. More formally this embedding is given by:
  
\[g\mapsto g^{[r]}:= \left\{\{\gamma_1,\ldots, \gamma _r\}\in \G^{[r]} \quad \colon \quad \gamma_i\in  g, i=1,\ldots, r \right\}.\]
In what follows we implicitly consider this embedding as an identification by viewing $\Ff(\G)$ as a subgroup of $\Ff(\G^{[r]})$, thus also as a subset  the pseudogroup $\widetilde{\G^{[r]}}$.

Let now $\H$ be an \'etale groupoid over a compact space $Y$. Given an element $g\in \Ff(\H)$, we define the \textbf{support} of $g$ to be the set $\supp(g)=\{y\in Y\colon [g]_y\neq y\}$. (Recall from \S \ref{s-pseudogroups} that  the notation $[g]_y$ refers to the (abstract) germ at $y$ in the pseudogroup $\tH$, i.e. the unique element  $\gamma\in g\subset \H$ such that $\src(\gamma)=y$). The support of an element is automatically closed in $Y$ since it is the complement in $Y$ of of $g\cap Y$ (and the latter is open). If $\H$ is effective, the support of $g$ coincides with the usual notion of support of the associated homeomorphism, i.e.  $\supp(g)=\overline{\{y\in Y\colon g(y)\neq y\}}$. The support of  a subgroup $K \le \Ff(\H)$ is defined as $\supp(K)=\overline{\bigcup_{k\in K} \supp(k)}$.
 
 Finally if $Z\subset Y$ is a clopen subset, we identify the group  $\Ff(\H|Z)$ with the subgroup of $\Ff(\H)$ consisting of elements whose support is contained in $Z$.

After all these clarifications, we  can  state:

\begin{thm}[Extension Theorem] \label{t-main}
Let $\mathcal{G}$ be a minimal effective \'etale groupoid over a Cantor space $X$, and let $\Hcal$ be an \'etale groupoid over any compact space $Y$. Let $G$ be a subgroup of  $\Ff(\G)$   that contains $\Af(\G)$. Let $\rho\colon G\to \Ff(\H)$  be a homomorphism, and let $Z\subset Y$ be the support of $\rho(\Af(\G))$. Exactly one of the following possibilities holds. 
\begin{enumerate}[label=(\roman*)]

\item \label{i-not-free} The set $Z$ is empty (ant thus $\rho$ factors through the quotient  $G/\Af(\G)$).

\item \label{i-free}   There exists a point $y\in  Z$ such that the map $G\to \H_y,  g\mapsto [\rho(g)]_y$ is injective.


 \item  \label{i-factor} The set $Z$ is a clopen in  $Y$, and there exist $r\ge 1$ such that $\rho|_{\Af(\G)}$ uniquely extends to a  continuous  morphism of  pseudogroups  
 \[ \widetilde{\rho}\colon\widetilde{\G^{ [r]}}\to \widetilde{\H|Z}.\]  
 Moreover, $\widetilde{\rho}|_G$ and $\rho$ are related as follows: there exists a group homomorphism $\psi\colon G/\Af(\G)\to \Ff(\H)$ whose image centralises $\widetilde{\rho}(G)$, and such that for every $g\in G$ we have $\rho(g)=\widetilde{\rho}(g)\psi\left(g\Af(\G)\right)$. 
 
%
%

  \end{enumerate}

\end{thm}

The proof is postponed to \S \ref{s-proof-ext}

\begin{remarks} \label{r-comments}
Let us make some comments on the statement.
\begin{enumerate}[label=(\roman*)]

\item The group $\Af(\G)$ being a normal subgroup of $\Ff(\G)$, is also normal in  $G$.

\item The most significant case of Theorem \ref{t-main} is when $G=\Af(\G)$, for which the unpleasant  last sentence in case \ref{i-factor} can be ignored (and case \ref{i-not-free} becomes irrelevant). 
For more general $G$, it may not be true that the obtained extension $\widetilde{\rho}$ coincides with ${\rho}$ on the whole group $G$. For instance in many cases  the group $G=\Ff(\G)$ surjects onto $\Z$ \cite{Mat-simple, Mat-SFT}. This can often be used to ``twist'' a given homomorphism $\rho\colon \Ff(\G)\to \Ff(\H)$ with a homomorphisms $\Z\to \Ff(\H)$ whose image centralises the image of $\rho$ (e.g. with disjoint support). 

\item \label{i-comments-G-equivariant} If case \ref{i-factor} holds, let  $q\colon Z\to X^{[r]}$ be the spatial component of $\widetilde{\rho}$. By Corollary \ref{c-continuous-morphism-full}, the map $q$ is equivariant for the $\rho$-action of the group $\Af(\G)$ on $Z$ with its natural  action on $X^{[r]}$. However, in fact one can check that  the set $Z$ is invariant under the $\rho$-action of the whole group $G$, and the map $q$ remains equivariant with respect to this action. 
This is true a-priori  only for the $\widetilde{\rho}$-action of $G$, but  it remains true for the $\rho$-action  using that the image of $\psi$ centralises $\widetilde{\rho}(G)$  (we omit details as this will also be apparent from the  the proof, see Proposition \ref{p-free-factor} below).

\item \label{i-comments-surjective}\label{i-comments-injective} Assume that $\rho$ falls in case \ref{i-factor}. Then by the general properties of the symmetric power $\G^{[r]}$,  upon  taking a smaller $r$ we can assume without loss of generality that the spatial the map  $q\colon U\to X^{[r]}$ is surjective (equivalently that $\widetilde{\rho}$ is injective, by Proposition \ref{p-inj-surj}), see Lemma \ref{l-closed-power}.   We will often use this observation.

\end{enumerate}

\end{remarks}

Before turning to the proof of Theorem \ref{t-main}, let us illustrate its statement by analysing it in some simple special cases. Perhaps the simplest is the following.

\begin{example}[Simplicity of the group $\Af(\G)$] \label{e-simple}
Let us consider Theorem \ref{t-main} in  the degenerate case when $\H=H$ is just a countable group, seen as an \'etale groupoid over the one-point space $Y=\{1_H\}$.(see Example \ref{e-groupoids} \ref{e-i-group}). Note that the full group $\Ff(H)$ is isomorphic to $H$ itself, and the pseudogroup $\widetilde{H}$ is just $H$ with the adjunction of a zero element $\varnothing$. 
Let $\rho\colon G\to \Ff(H)=H$ be a homomorphism. 
First of all note that Case \ref{i-factor} in the statement of Theorem \ref{t-main} cannot happen. If it did, by Remarks \ref{r-comments}~\ref{i-comments-surjective} we would obtain a surjective map $q\colon \{1_H\}\to X^{[r]}$ for some $r\ge 1$, which is absurd. We are left with cases \ref{i-not-free} and \ref{i-free}. Assume that Case \ref{i-free} happens. Then we have necessarily $y=1_H$, and the germ map $G\to H, g\mapsto [g]_{1_H}$ is simply $\rho$ itself. Thus case \ref{i-free} exactly says that $\rho$ is injective. On the other hand case \ref{i-not-free} exactly says that $\Af(\G)\subset \ker \rho$. Summing up, in this case the theorem says that for every  group $G\le \Ff(\G)$ containing $\Af(\G)$, every homomorphism $\rho\colon G\to H$ to a countable group $H$ verifies either $\ker{\rho}=\{1\}$ or $\Af(\G)\le\ker{\rho}$. For $G=\Af(\G)$ this is equivalent to the fact that $\Af(\G)$ is simple, and for $G=\Ff(\G)$ to the fact that $\Af(\G)$ is contained in every non-trivial normal subgroup of $\Ff(\G)$. Thus in this special case, Theorem \ref{t-main} reduces to  Theorem \ref{t-simple}. 
\end{example}


%

If one knows \emph{a-priori} that the image of $G$ acts minimally on  $Y$, the statement simplifies as follows: 

\begin{cor}[Case of image with minimal action] \label{c-action-minimal} 
Retain the  assumptions of Theorem \ref{t-main}, and assume further that $\rho(G)$ acts minimally on $Y$. Then exactly one of the following holds.

\begin{enumerate}[label=(\roman*)] 
\item The homomorphism $\rho$ factors through the quotient $G/\Af(\G)$.
\item For every $y\in Y$ the map $G\to \H_y, G\mapsto [\rho(g)]_y$ is injective. 
\item  The restriction $\rho|_{\Af(\G)}$ uniquely extends to a continuous morphism $\widetilde{\rho}\colon \tG\to \tH$. Moreover there exists a group homomorphism $\psi\colon G/\Af(\G)\to \Ff(\H)$ whose image centralises $\widetilde{\rho}(G)$ and such that for every $g\in G$ we have $\rho(g)=\widetilde{\rho}(g)\psi(g\Af(\G))$.
\end{enumerate}
\end{cor}
\begin{proof}
 Let us show that each case in Theorem \ref{t-main} corresponds its homologous here. For case \ref{i-not-free} this is obvious. For case \ref{i-free}, it is enough to observe that the set of points $y\in Y$ satisfying the conclusion must is closed and $\rho(G)$-invariant, thus by minimality it is either empty or equal to $Y$. In case \ref{i-factor}, observe first of all that since $\Af(\G)$ is normal in $G$, the support $Z$ of $\rho(\Af(\G))$ is invariant under $\rho(G)$, and thus by minimality we have $Z=Y$. Thus we obtain an integer $r\ge 1$ and a unique continuous extension $\widetilde{\rho}\colon \widetilde{\G^{[r]}}\to \tH$. By remark \ref{r-comments}~\ref{i-comments-surjective}, we can choose $r$ such that the associated spatial component $q\colon Y\to X^{[r]}$ is surjective. By minimality  this implies that $r=1$, otherwise $X=X^{[1]}$ would sit in $X^{[r]}$ as a proper closed $G$-invariant subset, and its preimage $q^{-1}(X)\subset Y$ would be a closed proper ${\rho(G)}$-invariant subset  (using Remark \ref{r-comments}~\ref{i-comments-G-equivariant}).  \qedhere
\end{proof}

The following is Rubin--Matui's  Isomorphism Theorem restated in the language of this paper.
\begin{cor} \label{c-isom} \label{c-isomorphisms}
For $i=1,2$ let $\G_i$ be a minimal effective \'etale groupoid over a Cantor space $X_i$.  Let $G_i$ be groups such that $\Af(\G_i)\le  G_i\le \Ff(\G_i)$. And assume that $G_1$ and $G_2$ are isomorphic. 

Then every group isomorphism $\rho\colon G_1\to G_2$  extends to a continuous isomorphism of pseudogroups $\widetilde{\rho}\colon \tG_1 \to \tG_2$.
In particular the pseudogroup $\tG_1$ and $\tG_2$ (equivalently, the \'etale groupoids $\G_1$ and $\G_2$) are isomorphic.
\end{cor}

\begin{proof}
Let $\rho\colon G_1\to G_2$ be an isomorphism, and apply Corollary \ref{c-action-minimal}. Since $\rho$ is injective, it cannot fall in case \ref{i-not-free}. Using Lemma \ref{l-multigerm}, we see that for every $x\in X_2$ there exists $g\in \Af(\G_2)\le G_2, g\neq 1$ such that $[g]_x=x$. Setting $h=\rho^{-1}(g)$, we have $[\rho(h)]_x=x$, and therefore $\rho$ cannot fall in case \ref{i-free}  either. Therefore it falls in case \ref{i-factor}. Let $\widetilde{\rho}\colon \tG_1\to \tG_2$ and $\psi\colon G_1/\Af(\G_1)\to \Ff(\G_2)$ be provided by the statement. Since $\Af(\G_i)$ is  characterised as the intersection of all non-trivial normal subgroups of $G_i$ (Example \ref{e-simple}), we have $\rho(\Af(\G_1))=\Af(\G_2)$. Thus $\widetilde{\rho}(G_1)$ contains $\Af(\G_2)$ and therefore has trivial centraliser (see Lemma \ref{l-trivial-centraliser}). We deduce that $\psi$ is trivial, and thus that $\rho$ extends to a continuous morphism $\widetilde{\rho}\colon \tG_1\to \tG_2$. Applying now the same reasoning to $\nu=\rho^{-1}\colon G_2\to G_1$ we obtain an extension $\widetilde{\nu}\colon \tG_2\to \tG_1$. Now note that $\widetilde{\nu}\circ \widetilde{\rho}\colon \tG_1\to \tG_1$ extends the identity isomorphism $G_1\to G_1$, and by the uniqueness of the extension in Theorem \ref{t-main}, we obtain that $\widetilde{\nu}\circ\widetilde{\rho}$ is the identity of $\tG_1$. By symmetry, we conclude that $\widetilde{\rho}$ is an isomorphism. 
\qedhere
\end{proof}

\subsection{Automatic Extension for targets with finite dimension} \label{s-asdim}

Beyond the case of isomorphisms, the Extension Theorem implies that for $\G, \H$ in a broad class of groupoids,  {all} homomorphisms between full groups must extend to  continuous pseudogroups morphisms on a symmetric power $\widetilde{\G^{[r]}}$. For simplicity we focus here on homomorphisms defined $G=\Af(\G)$. It is enough to find conditions on $\G, \H$ that ensure \emph{a-priori} that the germ map $\Af(\G)\to \H_y, g\mapsto [\rho(g)]_y$ can never be injective. This can be done by comparing the coarse geometry of the leaves $\H_y$ with the coarse geometry of the group $\Af(\G)$, and showing that there is ``no room'' for this to happen.  For many interesting examples of groupoids, the leaves have a quite simple structure and can be described explicitly, while the group $\Af(\G)$ usually has a quite complicated geometry (even for very ``small'' $\G$). Thus, this is often an easy task in practice.

In particular, we have the following (preliminaries on coarse geometry and on asymptotic dimension were given  in \S \ref{s-coarse})
\begin{thm}[Automatic Extension for targets with finite asymptotic dimension] \label{t-asdim}
Let $\G$ be a compactly generated effective minimal \'etale groupoid over a Cantor space $X$. Let $\H$ be an \'etale groupoid over a compact space $Y$, and assume that for every $y\in Y$ the leaf $\H_y$ has finite asymptotic dimension. Let  $\rho\colon \Af(\G)\to \Ff(\H)$ be a non-trivial homomorphism.

Then the support of $\rho(\Af(\G))$ is a  clopen subset $Z\subset Y$ and there exists $r \ge 1$ such that $\rho$  uniquely extends to a continuous morphism $\widetilde{\rho}\colon \widetilde{\G^{[r]}}\to \widetilde{\H|Z}$.

If moreover $d:=\sup_{y\in Y}\asdim(\H_y)<\infty$,  then the above extension exists for some $r\le d$. 
\end{thm}

The assumptions of this criterion are verified by a vast class of groupoids  (see  Examples \ref{e-intro-asdim} in the introduction).
\begin{proof}
Let $\rho \colon \Af(\G)\to \Ff(\H)$ be a non-trivial homomorphism, and apply Theorem \ref{t-main}. Case \ref{i-not-free} cannot happen here, because we assume that $\rho$ is non-trivial. Assume by contradiction that case \ref{i-free} holds. We obtain an injective  map  $\iota\colon \Af(\G)\mapsto \H_y, g\mapsto [\rho(g)]_y$. Let us check that this map is bornologous with respect to the natural coarse structures on $\Af(\G)$ and on $\H_y$.  Let $E\subset \Af(\G)\times \Af(\G)$ be a controlled set, so that  $F=\{gh^{-1} \colon (g, h)\in E\}$ is finite. The set $K=\bigcup_{f\in F} \rho(f)$ is a compact subset of $\H$, and for $(g, h)\in E$ we have $[\rho(g)]_y[\rho(h)]_y^{-1}=[\rho(gh^{-1})]_{\rho(h)(y)}\in K$, showing that $\iota_*(E)=\{([\rho(g)]_y, [\rho(h)]_y) \colon (g,h)\in E\}$ is bornologous.  We deduce that $\iota$ is an injective coarse map. But by Proposition \ref{p-asdim-alt} we have $\asdim(\Af(\G))=\infty$, which together with Proposition \ref{p-sisto} implies that $\asdim(\H_y)=\infty$, contradicting our assumption. Thus, we must be in case \ref{i-factor}, showing the first part of the theorem. To prove the las sentence, let us first  prove the following lemma.
\begin{lemma} \label{l-asdim-reduced-power}
Let $Q=\{x_1,\cdots, x_s\}\in X^{[r]}$, where $1\le s\le r$ and $x_1,\cdots, x_s$ are distinct.  Then the leaf of $\G^{[r]}$ based at $Q$ satisfies $\asdim(\G^{[r]}_Q)\ge \asdim(\G_{x_1}\times \cdots \times \G_{x_s})$. 
\end{lemma}
\begin{proof}
Let $U_1\ni x_1,\cdots, U_s\ni x_s$ be disjoint clopen subsets. By Proposition \ref{p-graph-qi}, we have $\asdim\left( (\G|U_i)_{x_i}\right)=\asdim(\G_{x_i})$. We have a map $(\G|U_1)_{x_1} \times \cdots \times (\G|U_s)_{x_s} \to \G^{[r]}_Q$ given by $(\gamma_1,\ldots, \gamma_s)\mapsto \{\gamma_1,\ldots, \gamma_s\}$, which is injective and coarse. Thus the conclusion follows from Proposition \ref{p-sisto}. 
\end{proof}
To conclude the proof of the theorem, let $q\colon Z\to X^{[r]}$ be the spatial component of $\widetilde{\rho}$. By Remark \ref{r-comments}~ \ref{i-comments-surjective} we can assume that $q$ is surjective. Since  we assume that $\G$ is compactly generated, it can be covered by countably many bisections homeomorphic to a subset of $X$, and  therefore $\G$ is second countable. A Baire argument shows that a second countable  effective \'etale groupoid is always essentially principal (see \cite[Prop. 2.5]{Nek-simple}). Thus we can choose distinct points $x_1,\cdots, x_r\in X$ lying in distinct $\G$-orbits and with trivial isotropy group, which implies that $Q=\{x_1,\cdots, x_r\}$ has trivial isotropy group in $\G^{[r]}$. Using Proposition \ref{p-morphism-leaves}, we get that there is a coarse injective map $\G_Q^{[r]}\to \H_y$ for every $y\in q^{-1}(Q)\subset Y$, and by Proposition \ref{p-asdim} and Lemma \ref{l-asdim-reduced-power} this implies that $r\le \asdim(\G_{x_1}\times \cdots \times \G_{x_r})\le \asdim(\H_y) \le \sup_{y\in Y} \asdim(\H_y)$ \qedhere
\end{proof}
\begin{remark}
The proof of Lemma \ref{l-asdim-reduced-power} shows more generally that  the space $\asdim(\G^{[r]}_Q)$ receives an injective coarse map from a space coarsely equivalent to  $\asdim(\G_{x_1}\times \cdots \times \G_{x_s})$. Actually, one can also show that $\G^{[r]}_Q$ admits an injective coarse map into $\G_{x_1}\times \cdots \times \G_{x_s}$, and thus we have $\asdim(\G^{[r]}_Q)=\asdim(\G_{x_1}\times \cdots \times \G_{x_s})$.  However, we point out that $\G^{[r]}$ and $\G_{x_1}\times \cdots \times \G_{x_s}$ need \emph{not} be coarsely equivalent. In fact, $\G^{[r]}$ may not even be compactly generated, even if $\G$ is. 
\end{remark}

In section \ref{s-concrete}-\ref{s-SFT} we will provide various  applications of Theorem \ref{t-asdim} to the study of homomorphisms between concrete examples of groups. 
Let us also record the following  consequence of the argument in its proof. 
\begin{cor} \label{c-asdim-obstruction}
Under the same assumptions of Theorem \ref{t-asdim}, assume that 
\[\sup_{y\in Y} \asdim(\H_y)< \inf_{x\in X} \asdim(\G_x).\] Then there is no non-trivial homomorphism $\Af(\G)\to \Ff(\H)$. 
\end{cor}
\begin{proof}
Combine Theorem \ref{t-asdim} with Lemma \ref{l-asdim-reduced-power} and with Proposition \ref{p-morphism-leaves}. \qedhere

\end{proof}

\subsection{Proof of the Extension Theorem} \label{s-proof-ext} We now turn to the proof of  Theorem  \ref{t-main}.

\medskip

\emph{Throughout this subsection, we  let  $\G, \H, G$ and $\rho\colon G\to \Ff(\H)$ be as in the statement of Theorem \ref{t-main}, and let $Z\subset Y$ be the support of $\rho(\Af(\G))$.}

\medskip

  For every $y\in Y$, we denote $\st_G(y)$ the stabilisers with respect to the action $G\acts Y$ induced by $\rho$. We further set
\[\st_G^{\rho}(y)=\{g\in G\colon [\rho(g)]_y=y\},\]
and note that it is a subgroup of $G$, and that $\operatorname{ker}(\rho)=\bigcap_{y\in Y} \st_G^{\rho}(y)$. In fact, for every $y\in Y$ the subgroup $\st_G^{\rho}(y)$ is a normal subgroup of the stabiliser $\st_G(y)$. Note that when $G$ is effective, the subgroup $\st^{\rho}_G(y)$ is simply the germ stabiliser $\st^0_G(y)$ for the action $G\acts Y$ induced by $\rho$. 

The proof proceeds by studying the family of subgroups $\st^{\rho}_G(y), y\in Y$ in the Chabauty space $\sub(G)$.

Let us rephrase  Case \ref{i-free} in Theorem \ref{t-main}.
\begin{lemma}\label{l-free}
For a point $y\in Y$, the following are equivalent.
\begin{enumerate}
\item The map $G\to \H_y, g\mapsto [\rho(g)]_y$ is injective.
\item We have $\st^\rho_G(y)=\{1\}$. 
\end{enumerate}

\end{lemma}
\begin{proof}
If $\st^\rho_G(y)\neq \{1\}$ every non-trivial element $g\in \st^\rho_G(y)$ satisfies $[\rho(g)]_y=y=[\rho(1)]_y$, thus the map $g\mapsto [\rho(g)]_y$ is not injective. Conversely assume that $g\neq h$ are such that $[\rho(g)]_y=[\rho(h)]_y=\gamma$. Then $[\rho(g^{-1}h)]_y=[\rho(g^{-1})]_{\rg(\gamma)}[\rho(h)]_y=\gamma^{-1}\gamma=\src(\gamma)=y$, thus $g^{-1}h$ is a non-trivial element of $\st^\rho_G(y)$.  \qedhere \end{proof}

We begin by explaining why the three cases in Theorem \ref{t-main} are mutually exclusive. 
\begin{prop} \label{p-exclusive}
The three cases in Theorem \ref{t-main} are mutually exclusive.
\end{prop}

\begin{proof}
Let us first prove the following Lemma.  
\begin{lemma}\label{l-exclusive}
Assume that $Z$ is clopen and that  $\rho|_{\Af(\G)} \colon \Af(\G)\to \Ff(\H)$ extends to a continuous morphism $\widetilde{\G^{[r]}}\to \widetilde{\H|Z}$ for some $r\ge 1$. Let $q\colon Z\to X$ be the spatial component of $\widetilde{\rho}$. Then for every $y\in Y$ we have $\st^0_{\Af(\G)}(q(y))\le \st^\rho_{G}(y)$. \end{lemma}
\begin{proof}
For $y\in Z$, let $g\in \st^0_{\Af(\G)}{(q(y))}$. This means that we can find  $V$ be a neighbourhood of $q(y)$ in $X^{[r]}$ which is pointwise fixed by $g$. Seeing $V$ as an idempotent  of the pseudogroup $\widetilde{\G^{[r]}}$,   we have the equality $gV=V$. Applying $\widetilde{\rho}$ we obtain $\rho(g)q^{-1}(V)=q^{-1}(V)$, that is $q^{-1}(V)\subset \widetilde{\rho}(g)$. Since $y\in q^{-1}(V)$, this implies that $\widetilde{\rho}(g)\in \st^\rho_G(y)$. Since $g\in \st^0_{\Af(\G)}{(q(y))}$ was arbitrary, we have $\st^0_{\Af(\G)}{(q(y))}\le \st^{\rho}_G(y)$.  \qedhere \end{proof}

Going back  to the proof of the proposition, it is clear that case \ref{i-not-free} in Theorem \ref{t-main} is exclusive with \ref{i-free} and \ref{i-factor}. If case \ref{i-factor} happens,  Lemma \ref{l-exclusive} implies that we have $\st^\rho_G(y)\neq \{1\}$ for every $y\in Z$ (since it is easy to see that $\st^0_{\Af(\G)}(Q)\neq \{1\}$  for every finite subset $Q\subset X$). Moreover if $y\notin Z$, then by definition of the support we have $\st^{\rho}_G(y)\supset \Af(\G)$, thus also $\st^{\rho}_{\Af(\G)}(y)\neq \{1\}$. Combining with Lemma \ref{l-free} we deduce that case \ref{i-free} and \ref{i-factor} cannot happen simultaneously. \qedhere
\end{proof}

The following Proposition is not used in the proof of the Extension Theorem, but shows that passing to a symmetric power $\G^{[r]}$ is indeed necessary for its statement to be true.
\begin{prop} \label{p-diagonal-does-not-extend}
For $r\ge 2$, the natural embedding $\Af(\G)\hookrightarrow \Ff(\G^{[r]})$ does not extend to a continuous morphism $\widetilde{\G^{[\ell]}}\to \widetilde{\G^{[r]}}$ for any $\ell \le 1$. 
\end{prop}

\begin{proof}
This follows from Lemma \ref{l-exclusive}, by observing that if $Q\in X^{[\ell]}$ and $P\in X^{[r]}$ with $|P|=r$ we cannot have $\st^0_{\Af(\G)}(Q)\le \st^{0}_{\Af(\G)}(P)$. \qedhere
\end{proof}

We now show that if Case \ref{i-factor}  in Theorem \ref{t-main} happens, then the extension is unique. 
\begin{prop}
Assume that $Z$ is clopen and that for some given $r\ge 1$, there exists continuous morphism $\widetilde{\rho}\colon \widetilde{\G^{[r]}}\to \widetilde{\H|Z}$ and $\widetilde{\rho}'\colon \widetilde{\G^{[r]}}\to \widetilde{\H|Z}$ that extend $\rho|_{\Af(\G)}$. Then $\widetilde{\rho}=\widetilde{\rho}'$.   
\end{prop}
\begin{proof}
Let us first prove the following lemma.
\begin{lemma} \label{l-power-germ}
Let $r\geq 1$. Then for every $\Qc\in \widetilde{\G^{[r]}}$ there exists $g\in \Af(\G)$ such that $\Qc\subset g$. 

\end{lemma}
\begin{proof}
 Assume at first that $\src(\Qc)$ and $\rg(\Qc)$ are disjoint. Let $\gamma_1,\ldots, \gamma_s$ for $s\le r$ be the distinct elements of $\Qc$. Using Lemma \ref{l-multigerm} and the fact that every $\G$-orbit is infinite and dense, we can find degree 3 multigerms $\gk_1,\ldots, \gk_d$ such that $\gk_i(1,2)=\gamma_i$ for $i=1,\ldots, s$ and such that the points $\gk_i(3, 3), i=1,\ldots, s$ are all disjoint and disjoint from $\src(\Qc)$ and $\rg(\Qc)$, and degree 3 multisections $\F_1\ni \gk_1,\ldots \F_s\ni \gk_s$ with disjoint support. Then the element $g=\F_1((123))\cdots \F_\src((123))$ verifies the desired conclusion. If $\src(\Qc)$ and $\rg(\Qc)$ intersect non-trivially, using again minimality we can write $\Qc=\Qc_1\Qc_2$ where $\Qc_1, \Qc_2\in \G^{[r]}$ fall in the previous case. Hence we can choose $g_1, g_2\in \Af(\G)$ such that $\Qc_i\subset {g_i}$ for $i=1,2$, and $g=g_1g_2$ verifies the conclusion. \qedhere
\end{proof}

Going back to the proof of the Proposition, let $q\colon Z\to X^{[r]}$ and $q'\colon Z\to X^{[r]}$ be the spatial components. Let us first show that $q=q'$. Assume by contradiction that there exists $y\in Z$ such that $\{x_1,\ldots, x_r\}=:q(y)\neq q'(y):=\{x_1',\ldots, x_r'\}$. We can assume without loss of generality that $x_1\notin \{x_1', \ldots, x_r'\}$. Choose $\gamma_1\in \G$ such that $\src(\gamma_1)=x_1$ and $\rg(\gamma_1)\notin q(y)\cup q'(y)$. By Lemma \ref{l-power-germ} we can find $g\in \Af(\G)$ such that $\{\gamma_1, x_1',\ldots, x_r'\}\subset g$. This $g$ belongs to $\st^0_{\Af(\G)}(q'(y))$, which by Lemma \ref{l-exclusive} is contained in $\st^{\rho}_G(y)\le \st_G(y)$, and thus $\rho(g)$ fixes $y$. Since the map $q$ is equivariant, we deduce that it also fixes $q(y)$, which is in contradiction with the fact that $\gamma_1\in g$. We conclude that $q=q'$. 

By Theorem \ref{t-non-comm-stone}, $\widetilde{\rho}$ and $\widetilde{\rho}'$ give rise to translation actions $\alpha\colon \G^{[r]} \acts \H|Z$ and $\alpha ' \colon \G^{[r]} \acts \H|Z$ via the formula \eqref{e-action-morphism}. But by Lemma \ref{l-power-germ}, the representative bisection $F\in \widetilde{\G^{[r]}}$ which appears in \eqref{e-action-morphism} can always be chosen to be in $\Af(\G)$. Since $\widetilde{\rho}|_{\Af((\G)}=\widetilde{\rho}'|_{\Af(\G)}$, we deduce that $\alpha=\alpha'$. Thus, $\widetilde{\rho}=\widetilde{\rho}'$, by Theorem \ref{t-non-comm-stone}. \qedhere
\end{proof}

We now turn to the core of the proof of the Theorem, which is to show that at least one of the three cases must  hold. An important role in the proof will be played by the upper and lower semicontinuity of various maps taking values in the Chabauty space $\sub(G)$. For the definition and  all other preliminaries on the Chabauty topology, we refer back to   \S \ref{s-preliminaries}.
\begin{lemma} \label{l-ker-semicontinuous}
The map $Y\to \sub(G), y\mapsto \st_G^{\rho}(y)$ is lower semicontinuous. \end{lemma}
\begin{proof}
We apply Lemma \ref{l-semicontinuous}. For every $g\in G$ we have $\{y\in Y\colon g\in \st_G^{\rho}(y)\}={\rho(g)}\cap Y$ which is open in $Y$. \qedhere
\end{proof}

\begin{lemma}\label{l-dense-normalizes}
Let $U\subset X$ be a dense open set. Then every non-trivial subgroup of ${\Ff(\G)}$ normalised by ${\Af(\G|{U})}$ contains ${\Af(\G|{U})}$.
\end{lemma}
\begin{proof}
Let $N\le {\Ff(\G)}$ be a non-trivial subgroup normalised by ${\Af({\G|{U}})}$. By Lemma \ref{l-double-comm}, there exists a non-empty open subset $V\subset X$ such that $N$ contains $[\rist_{\Af(\G|{U})}(V), \rist_{\Af(\G|{U})}(V)]$. Since $U$ is dense, the intersection $U\cap V$ is non-empty, and $\rist_{\Af(\G|{U})}(V)$ contains $\Af(\G|{U\cap V})$. The latter is a perfect group, and therefore $N$ contains $\Af(\G|{U\cap V})$, and thus it contains the normal closure of $\Af(\G|{U\cap V})$ in ${\Af(\G|{U})}$. Since the group ${\Af(\G|{U})}$ is simple by Theorem \ref{t:simple}, we get the conclusion. \qedhere

\end{proof}

The next proposition is the core of the argument in the proof of the Extension Theorem, and  makes use of the classification of the confined subgroups given in Theorem~\ref{t-confined}.  We denote by $\mathcal{P}_r(X)=X^{[r]}\sqcup\{\varnothing_X\}$ the space of {all} finite finite subsets of $X$ with cardinality at most $r$,  where   $\varnothing_X$ is the empty subset (the index $_X$ is put here to stress its role as a point in $\mathcal{P}_r(X)$). Note that  $\varnothing_X$ is the unique isolated point of $\P_r(X)$.
\begin{prop}\label{p-map} \label{p-free-factor}
Assume that for every $y\in Y$ we have $\st^\rho_G(y)\neq \{1\}$. Then there exists $r\geq 0 $ and a continuous  $G$-equivariant map $p\colon Y\to \P_r(X)$,   which is uniquely determined by the condition that $\st^0_{\Af(\G)}(p(y))\le \st^\rho_G(y)\le \st_G(p(y))$ for every $y\in Y$. 
\end{prop}


\begin{proof} Since $\st_G^{\rho}(y)\neq \{1\}$ for every $y\in Y$, we deduce from  Lemma \ref{l-ker-semicontinuous} and Lemma \ref{l-semicontinuous-closure} that the subgroup $\{1\}$ does not belong to the closure of $\{\st_G^{\rho}(y), y\in Y\}$ in $\sub(G)$. Hence, the image of the map $\st_G^{\rho}$ avoids a neighbourhood of $\{1\}$ of the form $\mathcal{U}_P=\{H\in \sub(G) \colon H \cap P=\varnothing\}$ for some non-empty finite set $P\subset G\setminus \{1\}$. We deduce that for every $y\in Y$, the subgroup $\st_G^{\rho}(y)$ is confined in $G$, with confining set $P$.  Set $r=|P|-1$.
 
By Corollary \ref{c-Chabauty-G}, for every $y\in Y$ there exists a unique finite set $p(y)\in \mathcal{P}_r(X)$ such that $\st^0_{\Af(\G)}(p(y))\le \st_G^{\rho}(y)\le \st_G(p(y))$. We therefore obtain a map $p\colon Y \to \mathcal{P}_r(X)$.  This map is clearly equivariant with respect to the action $G\acts Y$ induced by $\rho$ and to the action $G\acts \mathcal{P}_r(X)$. The non-trivial thing left to show is that it  is continuous.

The equivariance of $p$ automatically implies that we have the inclusion $\st_G(y)\le \st_G(p(y))$ for every $y \in Y$. We therefore have for every $y\in Y$ the chain of subgroups:
\[\st^0_{\Af(\G)}(p(y))\le \st_G^{\rho}(y)\le \st_G(y)\le \st_G(p(y)).\]
We will use the facts that the maps $\mathcal{P}_r(X)\to \sub(G), \Q\mapsto \st^0_{\Af(\G)}(\Q)$ and $Y\to \sub(G), y\mapsto \st_G^{\rho}(y)$ are lower semicontinuous (Example \ref{e-stab-semicontinuous} and Lemma \ref{l-ker-semicontinuous}), while the maps $y\mapsto \st_G(y)$ and $\Q\mapsto \st_G(\Q)$ are upper semicontinuous (see Example \ref{e-stab-semicontinuous}).
Let $(y_i)\subset Y$ be a net converging to a limit $y$. Let $\Q$ be a cluster point of $(p(y_i))$ in $\P_r(X)$, and let us show that $\Q=p(y)$. Up to taking a subnet, we can assume that $(p(y_i))$ converges to $\Q$ and that the four nets $\left(\st^0_{\Af(\G)} (p(y_i) )\right), \left(\st_G^{\rho}(y_i)\right), \left(\st_G(y_i)\right), \left(\st_G(p(y_i))\right)$ all converge in $\sub(G)$ to limits denoted $H_1,\ldots , H_4$ respectively. First, passing to the limit the inclusions $\st_G^{\rho}({y_i})\le \st_G({p(y_i)})$ and using  semicontinuity we obtain $\st_G^{\rho}(y)\le H_2\le H_4 \le \st_G(\Q)$, and therefore $\st_G^{\rho}(y)\le  \st_G(\Q)$. Next, passing to the limit the inclusion $\st^0_{\Af(\G)}(p(y_i))\le \st_G({y_i})$ and using semicontinuity in the same way we obtain $\st^0_{\Af(\G)}(\Q)\le  \st_{G}(y)$. This implies that  $\st^0_{\Af(\G)}(\Q)$ normalises $\st_G^{\rho}(y)$, since the latter is a normal subgroup of $\st_G(y)$. Since $\st^0_{\Af(\G)}(Q)=\Af(\G|\Q^c)$ (Lemma \ref{l-st-finite}) and we are assuming that $\st_G^{\rho}(y)$ is non-trivial,  Lemma \ref{l-dense-normalizes} implies that we actually have  $\st^0_{\Af(\G)}(\Qc)\le \st_G^{\rho}(y)$. We have proven that the set $\Q$ satisfies $\st^0_{\Af(\G)}(\Q)\le \st_G^{\rho}(y)\le \st_G(\Q)$ and therefore $p(y)=\Q$ by the uniqueness of $\Q$ in Theorem \ref{t-Chabauty}. Since $Q$ was an arbitrary cluster point of the net $(p(y_i))$, we deduce that $p(y_i)$ is actually converging to $p(y)=Q$. This completes the proof of the continuity of $p$, and thus the proof of the proposition. \qedhere

\qedhere \end{proof}

We now explain how to deduce the Extension Theorem from Proposition \ref{p-map}. This is based on formal manipulations with groupoids, and on the equivalence of categories established in Theorem \ref{t-non-comm-stone}.

\begin{proof}[Proof of Theorem \ref{t-main}] We assume that  case \ref{i-free} does not hold, i.e. $\st^\rho_{G}(y)\neq \{1\}$ for every $y\in Y$. and let $p\colon Y\to \P_r(X)$ be the map given by Proposition \ref{p-map}. 

We first observe that   the support $Z$ of $\rho(\Af(\G))$ is precisely $p^{-1}(X^{[r]})$. Namely since the map $p$ is equivariant and $\Af(\G)$ has no global fixed points in $X^{[r]}$, it follows that $\rho(\Af(\G))$ moves every point of $p^{-1}(X^{[r]})$ and thus $p^{-1}(X^{[r]})\subset Z$. Conversely, if $y\notin p^{-1}(X^{[r]})$, we have $p(y)=\varnothing_X$, and thus by Proposition \ref{p-map} we have  $\Af(\G)=\st^{0}_{\Af(\G)}(\varnothing_X)\le \st^{\rho}_G(y)$, showing that $y\notin U$. Thus $Z\subset p^{-1}(X^{[r]})$.  If  $Z$ is empty, we are in  case \ref{i-not-free}. Thus, we will assume that $p^{-1}(X^{[r]})\neq \varnothing $ and show that case \ref{i-factor} holds. Note that since $X^{[r]}$ is clopen in $\P_r(X)$ and the map $p$ is continuous, the set $Z$ is clopen in $Y$. 


Denote by $\P_r(\G)$  the groupoid over $\P_r(X)$ consisting of \emph{all} finite subsets of $\G$ of cardinality at most $r$ (including the empty one). In other words  $\P_r(\G)=\G^{[r]} \sqcup \{\varnothing_X\}$, where $\{\varnothing_X\}$ is seen as a trivial groupoid, and $\G^{[r]}=\P_r(\G)|X^{[r]}$. Note that $\Ff(\P_r(\G))$ is supported in $X^{[r]}$ and is isomorphic to $\Ff(\G^{[r]})$ as a group. 
We consider the quotient group $G/\Af(\G)$ as a groupoid with a one point unit space (see Example \ref{e-groupoids}~\ref{e-i-group}). Let $\K=\P_r(\G) \times \left(G/\Af(\G)\right)$ be the product groupoid over the space $\P_r(X)\times \{1_{G/\Af(\G)}\}$, that we identify with with $\P_r(X)$. Its  full group $\Ff(\K)$ contains a natural  copy of the direct product $\Ff(\P_r(\G))\times \Ff(\G/\Af(\G))=\Ff(\G^{[r]})\times \left(G/\Af(\G)\right)$.

We embed the group $G$ into $\Ff(\K)$ via a diagonal embedding $\delta \colon G \hookrightarrow \Ff(\G^{[r]})\times \left(G/\Af(\G)\right) \le \Ff(\K)$, defined by the usual inclusion $G\hookrightarrow \Ff(\G^{[r]})$ and by the projection to the quotient $G\to G/\Af(\G)$. 
\begin{claim}
There exists a continuous morphism of pseudogroup ${\eta} \colon \tK\to \tH$ such  that ${\eta}\circ \delta =\rho$, i.e. such that following diagram commutes.

\[\begin{tikzcd}
G\arrow[r, hook, "\delta"] \arrow[dr, "\rho" '] & \Ff(\K) \arrow[r, hook] & \tK \arrow[d, dotted, "{\eta}"] \\
 & \Ff(\H) \arrow[r, hook] & \tH
\end{tikzcd}\]

\end{claim}

 Before proving the claim, let us explain why it implies the desired conclusion. The restriction of $\eta$ to $\G^{[r]}=\widetilde{\P_r(\G)|X^{[r]}}$ provides us with a  continuous morphism $\widetilde{\rho}\colon \tG^r\to \tH|p^{-1}(X^{[r]})=\widetilde{\H|Z}$, while the restriction of $\eta$ to $G/\Af(\G)$ provides the desired homomorphism $\psi\colon G/\Af(\G)\to \Ff(\H)$.  Since $\Ff(\G^{[r]})$ and $G/\Af(\G^{[r]})$  centralise each other in $\Ff(\K)$, so do their images $\eta((\Ff(\G^{[r]}))=\widetilde{\rho}(\Ff(\G^{[r]}))\supset \widetilde{\rho}(G)$ and $\eta(G/\Af(\G))=\psi(G/\Af(\G))$. Finally, we have $\rho(G)=\eta(\delta(g))=\widetilde{\rho}(g)\psi(g\Af(\G))$, by construction.

Let us now prove the claim.  For $g\in G$ and $P\in \P_r(X)$, we denote by $[g]_P=\{[g]_x\colon x\in P\}$, i.e. the germ of $g$ seen as an element of $\Ff(\P_r(\G))$. On the other hand the notation $[\delta(g)]_P$ refers to the germ of $\delta(g)$ seen as an element of $\Ff(\K)$, which  is given by $[\delta(g)]_P=([g]_P, g\Af(\G))$. We need a lemma.
\begin{lemma}
For every $\left(\Qc, h\Af(\G)\right) \in\K$ there exists $g\in G$ such that $[\delta(g)]_{\src(\Qc)}=\left(\Qc, h\Af(\G)\right)$.
\end{lemma}
\begin{proof}
Set $\Qc'=[h^{-1}]_{\rg(\Qc)}\Qc$. By Lemma \ref{l-power-germ}, there exists $k\in \Af(\G)$ such that $\Qc'\subset k$. Then the element $g=hk$ satisfies $[g]_{\src(\Qc)}=\Qc$ and $g\Af(\G)=h\Af(\G)$, thus  $[\delta(g)]_{\src(\Qc)}=\left(\Qc, h\Af(\G)\right)$.  
\end{proof}

Let us define a translation action $\K\acts (\H, p\circ \rg)$ as follows. For $\left(\Qc, h\Af(\G)\right)\in \K$ and $\gamma\in \H$ such that $p(\rg(\gamma))=\src(\Qc)$, we choose an element $g\in G$ such that $[\delta(g)]_{\src(\Qc)}=\left(\Qc, h\Af((\G)\right)$, and let $\left(\Qc, h\Af(\G)\right)\cdot \gamma=[\rho(g)]_{\rg(\gamma)}\gamma$. Let us check that it is well-defined. Assume that $g_1, g_2$ satisfy $[\delta(g_1)]_{\src(\Qc)}=[\delta(g_2)]_{\src(\Qc)}=(\Qc, h\Af(\G))$. Then $[\delta(g_1^{-1}g_2)]_{\src(\Qc)}=(\src(\Qc), 1\Af(\G))$, which tell us that $g_1^{-1}g_2\in \Af(\G)$ and $[g_1^{-1}g_2]_{\src(\Qc)}=\src(\Qc)$, i.e. $g_1^{-1}g_2\in \st^0_{\Af(\G)}(\src(\Qc))$. Since $p(\rg(\gamma))=\src(\Qc)$, Proposition \ref{p-map} tells us that $\st^0_{\Af(\G)}(\src(\Qc)) \le \st^{\rho}_G(\rg(\gamma))$, and therefore $g_1^{-1}g_2\in \st^{\rho}_G(\rg(\gamma))$. We deduce that $[\rho(g_1)^{-1}\rho(g_2)]_{\rg(\gamma)}=\rg(\gamma)$, and thus $[\rho(g_1)]_{\rg(\gamma)}=[\rho(g_2)]_{\rg(\gamma)}$. The fact that it is indeed an action boils down to the equivariance of the map $p$, and it is obvious from its definition that it commutes with the natural right action $(\K, \src)\curvearrowleft \K$. 
By Theorem \ref{t-non-comm-stone}, this translation action arises from a continuous morphism $\eta\colon \tK\to \tH$. The fact that $\eta\circ \delta =\rho$ is apparent from the formula \eqref{e-action-morphism} below Theorem \ref{t-non-comm-stone}. This concludes the proof of the claim and therefore the proof of the theorem. \qedhere

\end{proof}


\section{Rigidity of actions with slowly growing orbits} \label{s-fg}

In this section we study property $\fg_{f(n)}$ and establish it for the group $\Af(\G)$. For the definition and a general discussion on property $\fg_{f(n)}$ we refer to \S \ref{s-intro-fg}. 

 \subsection{Some equivalent characterisations of property $\fg_{f(n)}$.}

 The following elementary proposition clarifies some equivalent characterisations of property $\fg_{f(n)}$ for finitely generated groups.
 Fix a finitely generated group $G$ with a finite symmetric generating set $S$, which will always  be implicit in what follows.  All the relevant terminology  on Schreier graphs and graphs of actions can be found in  \S \ref{s-notations}. 
  \begin{prop} \label{p-fg}
Let $G$ be a finitely generated group, and $f\colon \N\to \N$ be a function. The following are equivalent.
\begin{enumerate}[label=(\roman*)]
\item \label{i-fg-I} The group $G$ has property $\fg_{f(n)}$.
\item \label{i-fg-Ibis} For every  discrete metric space  $\Gamma$ such that $\overline{b}_\Gamma \nsucceq f$ and every homomorphism $\rho\colon G\to W(\Gamma)$, the action of $\rho(G)$ on $\Gamma$ has a finite orbit.  

\item \label{i-fg-II} For every action $G\acts \Omega$  on a set (not necessarily transitive, not necessarily with infinite orbits), either the graph of the action $\Gamma:=\Gamma_{G\acts \Omega}$ satisfies  $\overline{b}_{\Gamma} \succeq f$, or the actions factors through a finite quotient of $G$.
\item \label{i-fg-III} For every  subgroup $H\le G$ of infinite index, the Schreier graph $\Gamma_H$ satisfies $\overline{b}_{\Gamma_H}\succeq f$. 
\item \label{i-fg-IV} For every subgroup $H\le G$ of infinite index which belongs to a URS of $G$, the Schreier graph $\Gamma_H$ satisfies $\overline{b}_{\Gamma_H}\succeq f$.  
\end{enumerate}

\end{prop}

In particular, condition \ref{i-fg-II} can be taken as an alternative and simpler definition of property $\fg_{f(n)}$ for finitely generated groups (the formulation in terms of the group $W(\Gamma)$ is however needed to deal with non-finitely generated groups). Part \ref{i-fg-Ibis} justifies why this property can be thought of as a ``fixed point property''. 

\begin{proof}
Clearly \ref{i-fg-I}$\Rightarrow$\ref{i-fg-Ibis}. To see that \ref{i-fg-I}$\Leftrightarrow $\ref{i-fg-II}, observe that  every action $G\acts \Omega$ defines a homomorphism $G\to W(\Gamma_{G\acts \Omega})$, and conversely for every uniformly discrete metric space $\Delta$ of bounded geometry, and every homomorphism $\rho\colon G\to W(\Delta)$ the graph $\Gamma_{G\acts \Delta}$ of the corresponding action is Lipschitz-embedded into $\Delta$, and therefore satisfies $\overline{b}_{\Gamma_{G\acts \Delta}}\preceq \overline{b}_\Delta$.  It is obvious that \ref{i-fg-II}$\Rightarrow$\ref{i-fg-III}$\Rightarrow$\ref{i-fg-IV}, as well as that \ref{i-fg-Ibis}$\Rightarrow$\ref{i-fg-III}. Let us explain why \ref{i-fg-IV}$\Rightarrow$ \ref{i-fg-III}. Let $H\le G$ be a subgroup of infinite index. The closure of its orbit in $\sub(G)$ must contain a uniformly recurrent subgroup $Y$ and thus we can find a sequence $g_n\in G$ and an element $K\in Y$ such that $g_nHg_n^{-1}\to K$. We have by assumption $\overline{b}_{\Gamma_K}\succeq f$. But since $g_nHg_n^{-1}\to K$, every ball of $\Gamma_K$ appears in $\Gamma_H$,  thus $\overline{b}_{\Gamma_H}\ge  \overline{b}_{\Gamma_K}\succeq f$. Let us now explain why \ref{i-fg-III}$\Rightarrow$ \ref{i-fg-II}. Let $G\acts \Omega$ be an action which does not factor through a finite quotient of $G$. If there exists $\omega \in \Omega$ with an infinite orbit,  its stabiliser $H=\st_G(\omega)$ has infinite index and we have $\overline{b}_{G\acts \Omega} \succeq \overline{b}_{\Gamma_H}\succeq f$.  If not, then we can choose a sequence $\omega_n\in \Omega$ such that the size of the orbit of $\omega_n$ tends to $\infty$. Upon extracting a subsequence we can assume that $\st(\omega_n)$ tends to some infinite index subgroup $H\in \sub(G)$. Then every ball in $\Gamma_H$ appears as a ball $\Gamma_{G\acts \Omega}$, and thus $\overline{b}_{G\acts \Omega}\succeq \overline{b}_{\gamma_H}\succeq f$. The proof is complete. \qedhere

\end{proof}

 \subsection{Growth of groupoids and property $\fg_{f(n)}$ for full groups}
 
 Let $\G$ be a compactly generated \'etale groupoid, and let $\T$ be a finite symmetric generating set of $\tG$. Given a graph $\Gamma$ we denote by $B_\Gamma(n , v)$ its ball of radius $n$ around a vertex $v\in \Gamma$.

\begin{defin} The \textbf{orbital growth} of $\G$ with respect to the generating set $\T$ is the function of $n\in \N$ given by
\begin{equation} \label{e:growth} {\beta}_\G(n, \T)=\max_{x\in X}|B_{\orb(\G, \mathcal{T})}(n, x)|=\max_{x\in X}\overline{b}_{\orb_x(\G, \mathcal{T})}(n).  \end{equation}

\end{defin}
\begin{remark}
The \textbf{growth}  $\widetilde{\beta}_\G(n, \T)$ is defined similarly by replacing the the orbital graphs with the Cayley graphs \cite{Nek-complexity} (we will only consider the orbital growth in this paper). 

\end{remark}
Since changing generating set results in a bi-Lipschitz equivalence of the orbital  and Cayley graphs, the growth type of this functions according to $\sim$ does not depend on the choice of $\T$ and will be denoted simply  $\beta_\G(n)$ and $ \widetilde{\beta}_\G(n)$.

 \begin{lemma} \label{l-growth-graph-groupoid} \label{l-growth-graph}
If $\G$ is minimal, then the orbital graph of every point $x\in X$ satisfies $\overline{b}_{\orb_x(\G, \T)} \sim \beta_\G $.\end{lemma}
\begin{proof}
The inequality $\overline{b}_{\orb_x}(n)\preceq \beta_\G(n)$ is obvious. Let us show the converse. Fix $n$,  let $x\in X$ realise the maximum in the definition of $\beta_\G(n)$, and let $z_1=z, \cdots, z_r$ be the distinct vertices in in the ball of radius $r$ around $z$, where $r=\beta_\G(n)$. This implies that there exists bisections $T_1,\ldots, T_r\in \cup_{i=0}^n \T^n$ such that $z_i=\tau(T_i)(x)$ for $i=1,\ldots, r$.  It follows that if $U$ is a sufficiently small neighbourhood of $z$ then $\tau(T_1)(U),\ldots ,\tau(T_r)(U)$ are pairwise disjoint. Hence for every $w\in U$ the points $\tau(T_1)(w),\ldots, \tau(T_r)(w)$ are pairwise distinct. Since $\G$ is minimal, we can choose $w\in U$ which is in the same orbit of $x$, and since the ball of radius $n$ in $\orb_x$ centred at $w$ has at least $r$ different points this shows that $\overline{b}_{\orb_x}(n)\geq r$. \qedhere
\end{proof}

The natural action of $\Ff(\G)$ on each $\G$-orbit is by permutation of bounded displacement with respect to the orbital graph structure, thus defines a homomorphism $\Ff(\G)\to W(\orb_x(\G, \mathcal{T}))$. If $\G$ is minimal and effective, the action on every orbit is faithful, and thus $\Ff(\G)$ is actually a subgroup of the wobbling group $W(\orb_x(\G, \mathcal{T}))$.
Thus the group $\Ff(\G)$ embeds in $W(\Gamma)$ for a graph $\Gamma$ such that $\overline{b}_\Gamma\sim \beta_\G$. This is  sharp:

\begin{thm} \label{t-wobbling}
Let $\G$ be a compactly generated minimal effective groupoid  over a Cantor space, and let $\beta_\G(n)$ be its orbital growth function. The group $\Af(\G)$ has property $\fg_{\beta_\G(n)}$.
\end{thm}


\begin{proof} 
When $\G$ is expansive (so that $\Af(\G)$ is finitely generated),  this theorem can be  easily deduced directly from our classification of the uniformly recurrent subgroups of $\Af(\G)$ (Corollary \ref{c-URS}) and from Proposition \ref{p-fg}~\ref{i-fg-IV}. The proof given below is a modification of this argument,  which applies also when $\Af(\G)$ is not finitely generated.

 Let $\Delta$ be a graph and consider a homomorphism ${\Af(\G)}\to W(\Delta)$. We assume that $\rho(\Af(\G))$ is non-trivial and show that $\overline{b}_{\Delta}\succeq \beta_\G$. 
 
Using Lemma \ref{l-qi-graph}, we can find a finitely generated subgroup $K\le {\Af(\G)}$ endowed with a finite generating set $S$ such that for every $x\in X$, the orbital graph $\Gamma_x(K, S)$ for the action $K\acts X$ is bi-Lipschitz equivalent to $\orb_x(\G, \T)$, where $\T$ is a finite symmetric generating set of $\tG$. We fix such $(K, S)$, and omit the generating sets $S$ and $\T$ from the notations in what follows. 
 Let be a finitely generated subgroup as in Lemma \ref{l-qi-graph} and let $S$ be a finite generating set of $K$.  For every vertex $v\in \Delta$ the orbital  graph $\sch_v(K)$ is Lipschitz-embedded into $\Delta$, with Lipschitz constant independent from the choice of $v$.  In particular we have $\overline{b}_{\Gamma_v(K)}(n)\preceq \overline{b}_\Delta(n)$, with constant independent of $v$. There are two cases two consider. First, assume that there exists a sequence of vertices $(v_n)$ such that $\st_{{\Af(\G)}}(v_n)$ tends to $\{1\}$ in $\sub({\Af(\G)})$. It follows that $\st_K(v_n)=\st_{{\Af(\G)}}(v_n)\cap K$ tends to $\{1\}$ in $\sub(K)$ and therefore the corresponding Schreier graphs tend to the Cayley graph of $K$ in the space of marked graphs. In this case, letting $b_K$ be the growth function of $K$, we deduce that ${b}_K\preceq \overline{b}_\Delta$. Since the Cayley graph of $K$ covers $\Gamma_x(K)$ for every $x\in X$, we obtain $\beta_\G\sim \overline{b}_{\Gamma_x(K)}\preceq b_K\preceq \overline{b}_\Delta$, whence the conclusion. 
  Second, assume that there is no such sequence of vertices. In particular, for every vertex $v\in \Delta$ the group $\st_{{\Af(\G)}}(v)$ is confined, and therefore we can apply Theorem \ref{t-Chabauty} and obtain the existence of a finite set $\Q\subset X$ such that $\st^0_{\Af({\G})}(Q)\le \st_{{\Af(\G)}}(v)\le \st_{{\Af(\G)}}(\Q)$. If   $\Q=\varnothing$ the vertex $v$ is fixed by ${\Af(\G)}=\st^0_{\Af(\G)}(\varnothing)$. If this happens for every vertex the homomorphism ${\Af(\G)}\to W(\Delta)$  is trivial. Otherwise choose $v$ such that $\Q\neq\varnothing$. Then, taking the intersection with $K$, we see that $\st_K(v)=K\cap \st_{\Af(\G)}(v)\cap K\le \st_{\Af(\G)}(Q)\cap K=\st_K(Q)$. We deduce that the orbital  graph $\sch_v(K, S)$ covers the Schreier graph $\sch_{\st_K(Q)}(K)$. Since $Q$ is finite, its point-wise fixator $\fix_K(Q)$ has finite index in $\st_K(Q)$ and therefore the Schreier graphs $\sch_{\st_K(Q)}(K)$ and $\sch_{\fix_K(Q)}(K)$ are quasi-isometric. Finally, for every $x\in Q$ we have $\fix_K(Q)\le \st_K(x)$, and therefore the graph $\sch_{\fix_K(Q)}(K)$ covers the Schreier graph $\orb_{\st_K(x)}(K)$. Using these facts, we obtain 
\[\overline{b}_\Delta(n)\succeq \overline{b}_{\sch_v(K)}(n)\succeq \overline{b}_{\sch_{\st_K(Q)}(K)}(n)\sim \overline{b}_{\sch_{\fix_K(Q)}(K)}(n)\succeq \overline{b}_{\orb_{\st_K(x)}(K)}\sim \beta_\G(n).\] \qedhere
\end{proof}
\begin{remark} \label{r-wobbling-full} We would like to point out that Theorem \ref{t-intro-fg}  can also be seen as a  corollary of the Extension Theorem \ref{t-intro-main}. 
In fact for every graph $\Gamma$ the group $W(\Gamma)$  can be interpreted as the full group of a  groupoid $\G_\Gamma$ over the Stone-\v{C}ech compactification $\beta\Gamma$, namely the \emph{translation groupoid} of  Skandalis, Tu, and Yu \cite{Ska-Tu-Yu}. Applying the Extension Theorem to a homomorphism $\Af(\G)\to \Ff(\G_\Gamma)=W(\Gamma)$, Theorem \ref{t-intro-fg} can be deduced. 
\end{remark}

 Theorem \ref{t-wobbling} has the following consequences.
 \begin{cor} \label{c-fg-fm}
 There exists a finitely generated group which has property $\fg$ but  admits a transitive faithful action on an infinite set with an invariant mean (in particular which does not have property $\fm$). 
 \end{cor}
 \begin{proof}
Let $H$ be any finitely generated amenable group of exponential growth (e.g. take $H=\Z/2 \Z \wr \Z$). Let $H\acts X$ be a minimal, free expansive action on the Cantor set, which exists by \cite{Gao-Jack-Sew}. Let $\G$ be the groupoid of germs of the action. Its orbital graphs are bi-Lipschitz equivalent to $H$ and thus $\beta_\G$ is exponential, and it follows that the group $\Af(\G)$ has property $\fg$. Then $G$ is finitely generated by Theorem \ref{t-fg}. The orbital graphs of its action on $X$ are quasi-isometric to $H$ by \ref{l-qi-graph}, and hence $\beta_\G(n)$ is exponential.  Therefore $\Af(\G)$ has property $\fg$ by Theorem \ref{t-wobbling}. However the action of $\Af(\G)$ on every orbit in $X$ has an amenable orbital graph, and therefore preserves an invariant mean.
\end{proof}
\begin{cor}
For every function $f$ which is the orbital growth of a minimal expansive groupoid over the Cantor set,  there exists a finitely generated group $G$ which has property $\fg_{f(n)}$, and the function $f$ is sharp.

\end{cor}
This phenomenon appears to be new for any function $f$ which is neither linear nor exponential.

\begin{remark} The problem of determining exactly which functions can be realised as orbital growth function of a minimal expansive groupoid is natural and does not seem to have been studied yet.  It seems likely that every subexponentially growing function (satisfying some mild condition) can be realised. 
 We do not study this problem here, and just   observe the following.

\begin{enumerate}[label=(\roman*)]
\item Every function which can be realised as growth function of a finitely generated \emph{group} can also be realised as the orbital growth function of a minimal expansive groupoid, as follows from the fact that every finitely generated group admits a minimal free subshift  \cite{Gao-Jack-Sew}. By results of Bartholdi and Erschler \cite{Ba-Er-givengrowth}, this includes every function which satisfies a mild regularity condition and grows faster than $\exp({n^\alpha})$ for an explicit constant $\alpha\asymp 0.7674\ldots$.

\item  In \cite{MB-full}, the author studied a class of minimal groupoids associated to bounded automata groups acting on rooted trees. For the so called \emph{mother groups}, these groupoids are expansive (see \cite{MB-full}). Only the case of trees of constant valency is considered in \cite{MB-full}. By considering the same construction but on the mother groups over trees of varying valencies (see \cite{Brieu-entropy, AV-speed}), it is possible to realise a vast class of functions with oscillating behaviour in the polynomial growth range. 
\end{enumerate}

\end{remark} 

Finally, we record the following consequence of Theorem \ref{t-wobbling}
\begin{cor}\label{c-growth-obstruction}
Let $\G$ be a compactly generated minimal effective \'etale groupoid over a Cantor space. Let $\H$ be a compactly generated effective groupoid over a compact space. If the orbital growth functions satisfy $\beta_\H \nsucceq \beta_\G$, then there is no non-trivial homomorphisms $\Af(\G)\to \Ff(\H)$ 

\end{cor}
\begin{proof}
Assume by contradiction that $\rho\colon \Af(\G)\to \Ff(\H)$ is a  non-trivial homomorphism. Since $\Af(\G)$ is simple, its image must act either trivially or faithfully on every $H$ orbit. Thus we can see $\rho$ as taking values in $W(\orb_y)$ for some orbital graph $\Gamma_y$ of $\H$. The conclusion follows from Theorem \ref{t-wobbling}. \qedhere

\end{proof}

\section{Complexity of subshifts as an obstruction to group embeddings}
\label{s-sec-complexity}
By Corollary \ref{c-growth-obstruction}, the  orbital growth of a \'etale groupoids produces an obstruction to the existence of embeddings between full groups.  This invariant captures only  a  limited part of the nature of the groupoid, namely the geometry of its leaves. For instance, for the groupoids of germs of  actions of $\Z$ this invariant always grows linearly, but this family of groupoids is rich.
For this reason we consider in this section an invariant of dynamical nature of  of \'etale groupoids,  namely the \emph{complexity function}, and show that it also produces an obstruction to the existence of embeddings.

\subsection{Complexity of groupoids and of subshifts}
\label{s-complexity}

 Let  $\G$ be an \'etale groupoid over a Cantor space $X$, that we assume to be \emph{expansive} (Definition \ref{d-expansive}). Fix $\mathcal{T}=\{T_1,\ldots, T_n\}\subset \tG$  a symmetric expansive generating set of compact open bisections. For $n\ge 1$ we denote by $\mathcal{Q}_\mathcal{T}(n)$ the finite partition of $X$  generated by the sets $\src(F)$ for $ F\in \bigcup_{j=1}^n \T j$, i.e. two points $x,y$ belong to the same element of $\mathcal{Q}_\T(n)$ if and only if for every $F\in \bigcup_{j=1}^n \T^j$ we have $x\in \src(F) \Leftrightarrow y \in \src(F)$.
 
 \begin{defin} The function 
\[\pi_\G(n, \mathcal{T})=|\mathcal{Q}_\mathcal{T}(n+1)|\]
is called the \textbf{complexity} of $\G$ (with respect to the expansive generating set $\T$). 
\end{defin}

 Observe  that  $x, y$ belong to the same element of $\Q_\T(n)$ if and only  for every path of length at most $n$ in the Cayley graph $\cay_x(\G, \T)$ starting at $x$, there is a path in  $\cay_y(\G, \T)$ starting at $y$ whose edges have the same labels, and vice versa.   This happens if and only if the balls of radius $n$ in the universal covers of  the Cayley graphs $\cay_x(\G, \T)$ and $\cay_y(\G, \T)$ centred at any preimages of the points $x, y$ are isomorphic as rooted, labelled graphs. In the sequel we will switch between these two equivalent definitions of $\Q_\T(n)$,  and we will refer to them as the \emph{dynamical} and the \emph{combinatorial} definitions of complexity, respectively.

\begin{lemma}
If $\T_1$ and $\T_2$ are expansive generating sets of $\G$, we have $\pi_\G(n,  \mathcal{T}_2)\sim \pi_\G(n,  \mathcal{T}_1)$
\end{lemma}
Thus we will simply write $\pi_\G(n)$ when we are interested only on the growth type of the complexity function, which is a well-defined invariant of $\G$. 
\begin{proof}
We use the dynamical definition of $\mathcal{Q}_\T(n)$. By the definition of expansivity, and by compactness, there exists $M>0$ be such that every element of $\mathcal{T}_2$ is a union of elements of $\cup_{r=0}^M\mathcal{T}_1^r$. Let $x ,y \in X$ be $\mathcal{Q}_{\mathcal{T}_1}(Mn)$-equivalent. Let $T_2\in \bigcup_{r\leq n}\mathcal{T}_2^r$ be such that $x\in \src(T_2)$. Then there exists $T_1\in \bigcup_{r\geq n} \mathcal{T}_1^{Mr}$ such that $T_1\subset T_2$ and $x\in \src(T_1)$. It follows that also $y\in \src(T_1)\subset \src(T_2)$. Since $T_2$ was arbitrary, this shows that $x$ and $y$ are $\mathcal{Q}_{\mathcal{T}_2}(n)$-equivalent, i.e. the partition $\mathcal{Q}_{\mathcal{T}_1}(Mn)$ is finer than the partition  $\mathcal{Q}_{\mathcal{T}_2}(n)$. This shows that $\pi_\G(n, \T_2)\preceq \pi_\G(n , \T_1)$. The converse inequality also holds by exchanging the roles of $\T_1$ and $\T_2$.   \qedhere \end{proof}

Let us clarify the link of this notion with the classical notion of complexity in symbolic dynamics. Let $G$ be a finitely generated group, endowed with at finite symmetric generating set $S$. For every $n\ge 1$ denote $B(n)$ the ball of radius $n$ in the Cayley graph of $G$.  Let $G\acts X\subset A^G$ be a subshift over a finite alphabet $A$. Its \textbf{complexity} (with respect to the generating set $S$) is the function $p_{G\acts X, S} (n)$ which measures the number of configurations $f\colon B(n)\to A$ which appear as restrictions of elements of $X$. We have the following.
\begin{prop}\label{p-complexity-action} \label{p-complexity-shift}
Let $G\acts X$ be a subshift over a finitely generated group, with finite generating set $S$. Let $\G$ be either the action groupoid, or the groupoid of germs of $G\acts X$. Then its complexity function satisfies 
\[\pi_\G(n)\sim p_{G\acts X, S}(n).\] 
\end{prop} 
Note that it follows that the growth type of the function $p_{G\acts X, S}$ is independent of $S$ and is an invariant of topological conjugacy of the subshift (this is well-known and can also be easily checked directly).  
\begin{proof}
We use the combinatorial definition of $\Q_\T(n)$. Assume first that $\G=G\times X$ is the action groupoid. For $a\in A$ let $C_a=\{x\in X\colon x(1_G)=a\}\subset X$ be the corresponding cylinder set. For $s\in S$ and $a\in A$ we consider the bisection $T_{s, a}=\{(s, x)\colon x\in C_a\}$.  The set $\mathcal{T}=\{T_{s, a}\colon s\in S, a\in A\}$ is an expansive generating set for $\G$ (see the proof of\cite[Prop. 5.7]{Nek-simple}). The Cayley graphs $\cay_x(\G, \T)$ coincide with the Cayley graph of the group $G$ (with respect to the generating set $S$) with edges labelled by $S\times A$, where each edge $(g, gs)$ is labelled  by $s\in S$ and by $a=x(g)\in A$. Thus its universal cover  coincides with the Cayley graph of the free group $F_S$, with every edge $(w, ws)$ labelled by $s$ and by $x(\bar{w})\in A$, where $\bar{w}\in G$ is the projection of $w\in F_S$ to  $G$. With this description it is obvious that the ball of radius $n$ in the universal cover of the graph $\cay_x(\G, \T)$ is uniquely determined by the ball of radius $n$ in $\cay_x(\G, \T)$ and vice versa. This shows that $\pi_\G(n, \T)=p_{G\acts X, S}(n)$ for every $n \ge 1$.

Let now $\G'$ be the groupoid of germs of the action. Note that the image of the natural representations $\G\to \tI(X)$ and $\tG'\to \tI(X)$ are the same, and the complexity function only depend on this image (this is clear from the dynamical description of   $\Q_T(n)$). Thus $\pi_{\G'}\sim \pi_\G$, concluding the proof. \qedhere
\end{proof}

\begin{prop}
Let $\G$ be an \'etale groupoid over a Cantor space $X$, and $U\subset X$ be a clopen set intersecting every $\G$-orbit. Then the groupoid $\G|U$ is expansive if and only if $\G$ is expansive. If this holds, their complexity functions satisfy $\pi_\G\sim \pi_{\G|U}$. 
\end{prop}
\begin{proof}
The first assertion is exactly \cite[Prop. 5.6]{Nek-simple}. We show that the complexity functions are equivalent. It is obvious that $\pi_{\G|U}\preceq \pi_\G$. Let us show the converse inequality.

Let $\Sc$ be an expansive generating set for $\G$. Let $\T$ be a finite set of compact open bisections of $\G$ such that $\src(T)\subset U $ for every $T\in \T$ and the sets $\rg(T), T\in \T$ cover $X$. Without loss of generality, we assume that $U\in \T$, that $\T\cup\T^{-1}\subset \Sc$ and that $X\in \Sc$. By the argument in the proof of  \cite[Prop. 5.6]{Nek-simple}, the set $\T^{-1} \Sc \T$ of all products $T_0^{-1}FT_1$ for $T_i\in \T, F\in \Sc$ is an expansive generating set of $\G|U$. Note that $\src(T_0^{-1}FT)\subset \src(T)$. Thus for every $n\ge 1$  the partition $\Q_{\T\Sc\T^{-1}}(n)$ of $U$ is finer than the partition generated by the sets $\src(T), T\in \T$, that is, for every atom $P$ of $\Q_{\T\Sc\T^{-1}}(n)$ and every $T\in \T$ we have either $P\subset \src(T)$ or $P\cap \src(T)=\varnothing$. For every $T\in \T$, let us denote by  $\Q^T_{\T^{-1}\Sc\T}(n)$ the set of atoms $P$ such that $P\subset \src(T)$. A similar reasoning, using that $\T^{-1}\subset \Sc$, shows that the partition $\Q_\Sc(n)$ is finer than the partition generated by $\rg(T)\colon T\in \T$. For $T\in \T$ let us denote $\Q_\Sc^T(n)$ the set of atoms $P\in \Q_\Sc^T(n)$ such that $Q\subset \rg(T)$. 
For $T\in \T$ consider the associated homeomorphism $\tau(T)\colon \src(T)\to \rg(T)$. We claim that for every $n\ge 1$, it descends a map $\tau(T)\colon \Q^T_{\T^{-1}\Sc\T}(n)\to \Q^T_\Sc(n)$, that is if $x, y\in \src(T)$ belong to the same atom of $ \Q^T_{\T^{-1}\Sc\T}(n)$, then $\tau(T)(x)$ and $\tau(T)(y)$ belong to the same atom of  $\Q^T_\Sc(n)$. To see this, assume by contradiction that this is not the case for $x, y\in \src(T)$. Hence (modulo exchanging $x$ and $y$) there exists $F_1,\cdots, F_n\in \Sc$ such that $\tau(T)(x)\in \src(F_n\cdots F_1)$ and $\tau(T)(y)\notin \src(F_n\cdots F_1)$. For $j=1,\cdots , n$ let $T_j\in \T$ be such that $\tau(F_j\cdots F_1T)(x)\in \rg(T_j)$. Then we have $\tau(T)(x)\in \src(T_n^{-1}F_nT_{n-1}T_n^{-1}F_{n-1}T_{n-2}\cdots T_1^{-1}F_1)\subset \src(F_n\cdots F_1)$. Thus we deduce that $x\in s\left((T_n^{-1}F_nT_{n})\cdots T_1^{-1}F_1T)\right)$ and $y\notin s\left((T_n^{-1}F_nT_{n})\cdots T_1^{-1}F_1T)\right)$. This contradicts the fact that $x$ and $y$ were in the same atom of $\Q_{\T^{-1}\Sc\T}(n)$. 

Thus we obtain a map $\Q^{T}_{\T^{-1}\Sc\T}(n)\to \Q^T_\Sc(n)$, which is obviously surjective (since the homeomorphism $\tau(T)$ is). This implies that for every $T\in \T$ and every $n\ge 1$ we have $|\Q^T_\Sc(n)| \le |\Q^T_{\T^{-1}\Sc\T}(n)|$. Thus, using that $\Q_\Sc(n)=\bigcup_{T\in \T} \Q^T_\Sc(n)$, we obtain
\begin{multline*} \pi_\G(n, \Sc)=\\ |\Q_\Sc(n)| \leq \sum_{T\in \T} |\Q^T_\Sc(n)| \leq   \sum_{T\in \T} |\Q^T_{\T^{-1}\Sc\T}(n)| \leq |\T| \cdot |\Q_{\T^{-1}\Sc\T}(n)| =\\ |\T|\cdot \pi_{\G|U}(n, \T^{-1}\Sc\T),\end{multline*}
showing that $\pi_\G\preceq \pi_{\G|U}$. \qedhere

\end{proof}

\begin{prop}\label{p-complexity-morphism}
Let $\G, \H$ be expansive groupoids over Cantor spaces $X, Y$. Assume that there exists a continuous injective morphism $\varphi\colon \tG\to \tH$. Then the complexity functions of $\G$ and $\H$ satisfy $\pi_\G\preceq \pi_\H$. 
\end{prop}

\begin{proof}
Let $\varphi\colon \tG\to \tH$ be a continuous injective morphism. By Proposition \ref{p-inj-surj}, the spatial component $q\colon Y\to X$ is surjective. 
Let $\T$ be an expansive generating set of $\G$. Let $\Sc$ be an expansive generating set of $\H$. Upon replacing $\Sc$ with $\Sc\cup \varphi(\T)$ we can assume that $\Sc$ contains $\varphi(\T)$. 
With this choice, the partition $\Q_\Sc(n)$ is finer than the partition $\Q_{\varphi(T)}(n)$ generated by the sets $\src(F)$ for $\colon F\in \bigcup_{r=1}^n \varphi(\T)^r$. By Proposition \ref{p-spatial-equivariant}, the latter is exactly the pullback of the partition $\Q_\T(n)$ under the map $q$. Thus, since $q$ is surjective, we have $|\Q_\T(n)|=|\Q_{\varphi(\T)}(n)|$. Therefore $\pi_\G(n, \T)=|\Q_\T(n)|=|\Q_{\varphi(\T)}(n)|\le |\Q_\Sc(n)|=\pi_\H(n, \Sc)$.  \qedhere
\end{proof}

\begin{remark}

Nekrashevych considers a different definition of complexity of \'etale groupoids in  \cite{Nek-complexity}, defined as the number  ${\pi}'_\G(n, \T)$ of elements of the partition of $X$ induced by isomorphism classes of the  balls of radius $n$ in the Cayley graphs $\cay_x(\G, \T)$  (without passing to the universal cover).  We changed this definition because the function $\pi_\G(n, \T)$ is better behaved for the purposes of this paper. Moreover,  expansivity of groupoids is more related to the universal covers of Cayley graphs then to the Cayley graphs themselves, as shown in  \cite{Nek-simple}. It is not difficult to see that  ${\pi}_\G(n, \T)\le {\pi}'_\G(n, \T)$, and the two functions coincide if $\G$ is the action groupoid of an expansive action of a finitely generated group (by Proposition \ref{p-complexity-action}).  It would be interesting to study  further the relation between these two notions of complexity.
\end{remark}

\subsection{Complexity as an obstruction to group embeddings}

The Extension Theorem implies the following. 
\begin{thm} \label{t-complexity-dichotomy}
Let $\G, \H$ be expansive \'etale groupoids over Cantor spaces $X, Y$, with $\G$ minimal and effective. Assume that  the complexity functions of $\G$ and $\H$ satisfy $\pi_\H\nsucceq \pi_\G$. Then for every non-trivial homomorphism $\rho\colon \Af(\G)\to \Ff(\H)$, there exists a point $y\in Y$ such that the germ map $\A(\G)\to \H_y, g\mapsto [\rho(g)]_y$ is injective. 
\end{thm}
\begin{proof}
Assume by contradiction that the conclusion does not hold. By Theorem \ref{t-main}, we obtain that the the support $Z$ of $\rho(\Af(\G))$ is clopen and there exists an extension $\widetilde{\rho}\colon \widetilde{\G^{[r]}}\to \widetilde{\H|Z}$ for some $r\ge 1$. As usual, by Remark \ref{r-comments}~\ref{i-comments-injective} we assume that the spatial component $q\colon Z\to X^{[r]}$ is surjective. Let $\T$ be a symmetric expansive generating set of $\tG$, and let $\T'=\{T^{[r]}\colon T\in \T\}\subset \widetilde{\G^{[r]}}$.  Note that $\T'$ is not necessarily an expansive generating set of $\widetilde{\G^{[r]}}$, but we still have a natural a sequence of partitions $\Q_n(\T')$ of $X^{[r]}$ associated to $\T'$ (we use the dynamical point of view). Note that in restriction to $X\subset X^{[r]}$, each partition  $\Q_n(\T')$ induces exactly $\Q_n(\T)$, and thus $|\Q_n(\T')|\ge |\Q_n(\T)|$. Since $Z$ is clopen, we can view $\widetilde{\rho}(\T')$ as a set of compact open bisections in $\tH$. Choose an expansive generating set $\Sc$ of $\tH$ which contains $\widetilde{\rho}(\T')$ (after replacing it if necessary with $\Sc\cup \widetilde{\rho}(\T')$). By Proposition \ref{p-spatial-equivariant}, for every $n$ the partition $\Q_n(\Sc)$ contains $q^{-1}(\Q_n(\T'))=\{q^{-1}(V)\colon V\in \Q_n(\T'))$, and thus we obtain $\pi_\H(n, \Sc)\geq |\Q_n(\T')| \ge \pi_\G(n, \T)$, contradicting the assumption. 
\end{proof}
When the groupoid $\G$ has a rapidly growing complexity function $\pi_\G$, this theorem can be interpreted as a rigidity property for actions of the group $\Af(\G)$ in the form of a dichotomy. Namely whenever $\H$ is an expansive groupoid such that $\Af(\G)$ embeds in $\Ff(\H)$, then  $\H$ must be either ``geometrically complicated'', meaning that it has a leaf $\H_y$ which contains a coarsely injected copy of the group $\Af(\G)$, or $\H$ must be ``dynamically complicated'', i.e. its complexity function cannot grow slower than $\pi_\G$. This phenomenon appears to be new in the realm of finitely generated groups. 
It is not known whether non-abelian free groups satisfy a similar dichotomy (see \cite[Question 1.6]{MB-Liouville} and \cite[p. 377]{Nek-complexity} for formulations of this question). 

Note also that for many classes of groupoids, the complexity function is an obstruction to embeddings between the corresponding full groups, therefore confirming the natural intuition that a slow growth of the complexity ``constraints'' the subgroup structure. For example, we have the following. 
\begin{cor} \label{c-complexity}
Let $\G, \H$ be expansive groupoids over Cantor spaces, with $\G$, minimal and effective, and assume that every leaf of $\H$ has finite asymptotic dimension. If the complexity functions satisfy $\pi_\H\nsucceq \pi_\G$, then there is no non-trivial homomorphism $\Af(\G)\to \Ff(\H)$.  
\end{cor}
\begin{proof}
Apply Theorem \ref{t-complexity-dichotomy} and repeat the beginning of the proof of Theorem \ref{t-asdim}. \qedhere
\end{proof}
Note that this applies in particular to the groupoid of germs of topologically free subshifts $G\acts X$ and $H\acts Y$ over finitely generated groups, provided $H$ has finite asymptotic dimension (in this case the complexity is simply the complexity of the corresponding subshifts, see Proposition \ref{p-complexity-shift}). In particular, this implies Theorem \ref{c-intro-complexity} in the introduction.


\section{Application to actions of abelian groups} \label{s-concrete1}
\label{s-concrete} \label{s-concrete-examples} \label{s-concrete example}

In this section we illustrate the Extension Theorem in practice by analysing  the structure of all possible homomorphisms between  topological full groups of Cantor minimal systems arising from $\Z$ actions, and to a family of groups of interval exchanges that arise as full groups of actions of abelian groups by ``Cantor rotations'' on the circle.

\subsection{Full groups of Cantor minimal systems} \label{s-Z}

 A \textbf{Cantor minimal system} is a dynamical system $(X, u)$ where $X$ is a Cantor space, and $u$ is a minimal homeomorphism of $X$. We denote $\G_u$ the groupoid of germs of the corresponding $\Z$-action.  We denote by $\Ff(X, u)$ the topological full group of $\G_u$, and by $\Df(X, u)$ its derived subgroup. Giordano, Putnam and Skau  \cite{G-P-S-full} proved that the isomorphism type of the group $\Ff(X, u)$ is a complete invariant of  \emph{flip-conjugacy} of the system $(X, u)$ (i.e. determines the pair $\{u, u^{-1}\}$ up to conjugacy).   Matui proved that $\Df(X, u)$ is simple \cite{Mat-simple}  (in particular in this case $\Df(X, u)=\Af(\G_u)$). Moreover it is finitely generated if and only if $(X, u)$ is conjugate to a subshift over a finite alphabet. By a result of Juschenko and Monod, the group $\Ff(X, u)$ is amenable \cite{Ju-Mo}. 

Although the isomorphism class of the  group  $\Ff(X, u)$ fully remembers the system $(X, u)$, the dependence on $(X, u)$ is not well-understood ``qualitatively''. In particular, there was no known result providing explicit obstructions on the possible subgroups of $\Ff(X, u)$ which depend non-trivially on $(X, u)$.  As an application of the Extension Theorem, we characterise all the possible embeddings between the groups $\Ff(X, u)$ when $(X, u)$ varies. This answers a question  asked by Cornulier in \cite[Question 2f]{Cor-Bou}.

The leaves of the groupoid $\G_u$ are quasi-isometric to $\Z$, and thus have asymptotic dimension equal to 1. Therefore, we have the following  immediate corollary of Theorem \ref{t-asdim}

\begin{thm} \label{t-Z-extends}
Let $(X_1, u_1)$ and $(X_2, u_2)$ be Cantor minimal systems. For every non-trivial homomorphism $\rho\colon \Df(X_1, u_1)$, the support $Z$ of $\rho(\Df(X_1, u_1))$ is a clopen subset of $X_2$ and $\rho$ extends to a continuous morphism of pseudogroups $\widetilde{\rho}\colon \tG_{u_1}\to \widetilde{\G_{u_2}|Z}$.\end{thm}

%
Let us state a more concrete interpretation of Theorem \ref{t-Z-extends}, which does not use of the language of pseudogroups.   The group $\Ff(X, u)$ is the group of all homeomorphisms of $X$ that  coincide locally with a power of $u$. Thus, every $g\in \Ff(X, u)$ is uniquely determined by a continuous function $k_g\colon X\to \Z$, called the \emph{orbit cocycle}, defined by the equality
\[g(x)=g^{k_g(x)}(x) \quad \forall x\in X.\]
One  ``obvious'' dynamical reason why the group $\Ff(X_1, u_1)$ should embed into the group $\Ff(X_2, u_2)$ is the existence of a \emph{factor map} from $(X_2, u_2)$ to $(X_1, u_1)$. Namely every factor map  $q\colon X_2\to X_1$  gives rise to an embedding of the group $\Ff(X_1, u_1)$ into $\Ff(X_2, u_2)$ by simply precomposing the orbit cocycles with the map $q$, i.e.
\[\rho(g)(x)=u_2^{k_g(q(x))}(x).\]
It is easy to check that this defines an embedding $\Ff(X_1, u_1)\to \Ff(X_2, u_2)$. Another straightforward construction of embeddings is as follows: for every $v\in \Ff(X_2, u_2)$, letting $\supp(v)\subset X_2$ denote its support (which is always clopen), the group $\Ff(\supp(v), v)$ is naturally a subgroup of $\Ff(X_2, u_2)$.

Theorem  \ref{t-Z-extends} says that all embeddings arise from these two examples.

\begin{thm} \label{t-Z-concrete}
Let $(X_1, u_1)$ and $(X_2, u_2)$ be Cantor minimal systems.  The following are equivalent:
\begin{enumerate}[label=(\roman*)]
\item \label{i-Z-elementary-2} There exists a group embedding $\Df(X_1, u_1)\hookrightarrow \Ff(X_2, u_2)$.
\item \label{i-Z-elementary-1} There exists a group embedding  $\Ff(X_1, u_1)\hookrightarrow \Ff(X_2, u_2)$. 
\item \label{i-Z-elementary-3} There exists an element $v\in \Ff(X_2, u_2)$ such that the system $(\supp(v), v)$ factors onto $(X_1, u_1)$. 
\end{enumerate}
More precisely, non-trivial group homomorphism $\rho\colon \Df(X_1, u_1)\to \Ff(X_2, u_2)$ are in one-to-one correspondence with pairs $(v, q)$ consisting of  an element $v\in \Ff(X_2, u_2)$ as in \ref{i-Z-elementary-3} and a factor map $q\colon \supp(v)\to X_2$. Such a pair $(v, q)$ gives rise to a homomorphism $\rho\colon \Df(X_1, u_1)\to \Ff(X_2, u_2)$ given by 
\[\rho(g)(x)=\left\{\begin{array}{rl} v^{k_g\circ q(x)} (x) & \text{ if }x\in \supp(v) \\x & \text{ otherwise }\end{array}\right.\]
and conversely every homomorphism $\Df(X_1, u_1)\to \Ff(X_2, u_2)$ arises in this way for a unique pair $(v, q)$. 

\end{thm}
\begin{proof}
\ref{i-Z-elementary-1}$\Rightarrow$\ref{i-Z-elementary-2} is obvious. To see that  \ref{i-Z-elementary-3}$\Rightarrow$\ref{i-Z-elementary-1}, it is enough to observe that for every $v\in \Ff(X_2, u_2)$ as in \ref{i-Z-elementary-3} and every factor map $q\colon \supp(v)\to X_2$, we obtain a group embedding $\Ff(X_1, u_1)\hookrightarrow \Ff(X_2, u_2)$. To see that \ref{i-Z-elementary-2}$\Rightarrow$\ref{i-Z-elementary-3}, let $\rho\colon \Df(X_1, u_1)\to \Ff(X_2, u_2)$ be a non-trivial homomorphism. By Theorem \ref{t-Z-extends}, it can be extended to a continuous morphism $\widetilde{\rho}\colon \tG_{u_1}\to \tG_{u_2}|Z$, where $Z$ is the (clopen) support of $\rho(\Df(X_1, u_1))$. Let $q\colon Z\to X_1$ be the spatial component of $\widetilde{\rho}$.  We view $u_1$ as a bisection of $\G_{u_1}$ (by identifying it with the set of its germs). Let $v=\widetilde{\rho}(u_1)$ (as usual, we identify $v$ with an element of $\Ff(X_2, u_2)$ by extending it to the identity outside $U$). Then $(Z,v)$ factors onto $(X_1, u_1)$ via the map $q$, and thus the element $v$ verifies \ref{i-Z-elementary-3}. It remains to show the ``More precisely...'' part of the statement. Let $g\in \Df(\G)$ be arbitrary and let $k_g\colon X_1\to X_2$ be its orbit cocycle. For every $n\in \Z$ let $A_n\subset X_1$ be the level set $\{x\in X_1\colon k_g(x)=n\}$ (note that $A_n$ is clopen,  empty for all but finitely many $n$'s, and that $X_1=\bigsqcup_{n\in \Z}A_n)$. The restriction of $g$ to $A_n$ is equal to $u_1^n$. This means that in  $\tG_1$, the element $g$ can be written as the disjoint union:
\[g=\bigsqcup_{n\in \Z} u_1^n A_n.\]
Applying $\widetilde{\rho}$, and using that it preserves disjoint unions, we obtain
\[\rho(g)=\bigsqcup \widetilde{\rho}(u_1)^n \widetilde{\rho}(A_n)=\bigsqcup_{n\in \Z} \widetilde{\rho}(u_1)^n q^{-1}(A_n)\]
note that $q^{-1}(A_n)=\{x\in Z \colon k_h(q(x))=n\}$. This concludes the proof \qedhere
\end{proof}

\subsection{Groups of interval exchanges} \label{s-iet}
Recall that an \textbf{interval exchange transformation} is a piecewise continuous permutation of $\R/\Z$ with finitely many points of discontinuities, which coincides with a translation in restriction of each of its intervals of continuity. The group of all interval exchange transformations is denoted $\iet$. The subgroup structure of $\iet $ is quite mysterious. For instance, it is still not known whether the group $\iet$ contains non-abelian free-subgroups (a question attributed to Katok in \cite{D-F-G}), and whether it contains non-amenable subgroups. See \cite{D-F-G, ExtAmen, D-F-G-solv}.

A natural family of subgroups of $\iet$ arises as full groups of actions of finitely generated abelian groups by ``Cantor rotations". Let $\Lambda$ be a finitely generated dense additive subgroup of $\R/\Z$, and consider the natural translation action $\Lambda \acts \R/\Z$. Let $\Lambda\acts X_\Lambda$  be  the Cantor space obtained from $\R/\Z$ by  ``doubling'' the orbit of $0 \in \R/\Z$, i.e. by replacing every point $\lambda\in \Lambda$ by two copies $\lambda_-, \lambda_+$  and endowing $X_\Lambda$ with the topology induced by the natural circular order on it.  We have a natural continuous surjection $\pi_\Lambda \colon X_\Lambda\to \R/\Z$, obtained by gluing back every pair $\lambda_-, \lambda_+$ to $\lambda$. The translation action of $\Lambda$ on $\R/\Z$ lifts to  a minimal action of $\Lambda$ on $X_\Lambda$, which factors onto the original action on $\R/\Z$ through the map $\pi_\Lambda$.  In the important special case when $\Lambda=\langle \lambda \rangle$ is generated by a single irrational rotation, the system $(X_\Lambda, \lambda)$ is called the \textbf{Sturmian subshift} of angle $\lambda$.  We will denote by $\G_\Lambda$ the groupoid of germs of the action of $\Lambda\acts X_\Lambda$, and by  $\iet(\Lambda)$  its topological full group. A result of Matui implies that the  derived subgroup $\iet(\Lambda)'$ is simple  \cite{Mat-SFT}  (thus it coincides with the group $\Af(\G_\Lambda)$), moreover it is finitely generated  by \cite{Nek-simple}

The group $\iet(\Lambda)$ is easily seen to be isomorphic to the subgroup of $\iet$  of all interval exchanges whose points of discontinuity belong to $\Lambda$, and that in restriction to every interval of continuity are given by translation by an element of $\Lambda$ (see \cite{Cor-Bou, ExtAmen} for a more detailed explanation). We will freely switch between these two point of views on $\iet(\Lambda)$. 

Every {finitely generated} subgroup of $\iet$ is a subgroup of $\iet(\Lambda)$ for the group $\Lambda$ generated by its points of discontinuities.  Thus, it is natural to study to what extent the possible subgroups of the group $\iet(\Lambda)$ depend on the group $\Lambda$. 

%

The next theorem classifies the pairs $(\Lambda, \Delta)$ for which the group $\iet(\Lambda)$ admits an embedding into the group $\iet(\Delta)$.  
Given a subgroup $\Lambda<\R/\Z$ we denote $\widehat{\Lambda}
\le \R$ its preimage under the quotient map $\R\to \R/\Z$.

\begin{thm} \label{t-sturmian-embedding}
Let $\Lambda, \Delta$ be finitely generated dense subgroups of $\R/\Z$. The following are equivalent. 

\begin{enumerate}[label=(\roman*)]
\item \label{i-sturmian-2} there exists a group embedding   $\iet(\Lambda)'\hookrightarrow \iet(\Delta)$
\item \label{i-sturmian-1} there exists a group  embedding $\iet(\Lambda) \hookrightarrow \iet(\Delta)$.
\item  \label{i-sturmian-3} there  exists $0<c\le 1$ such that the linear map $\R\to \R, x\mapsto cx$  sends  $\widehat{\Lambda}$  into $\widehat{\Delta}$. 
 \end{enumerate}
 Moreover the image of every embedding $\rho\colon \iet(
 \Lambda)'\hookrightarrow  \iet(\Delta)$   fixes pointwise a non-empty interval unless $\Lambda\le \Delta$. \end{thm}

Note that the last sentence implies that $\iet(\Lambda)'$ and $\iet(\Delta)'$ are isomorphic if and only if $\Lambda=\Delta$. This can also be proven using the Isomorphism Theorem.

We will deduce Theorem \ref{t-sturmian-embedding}  from the following  theorem. 
\begin{thm}\label{t-iet-extends}
Let $\Lambda, \Delta$ be finitely generated dense subgroups of $\R/\Z$. For  every non-trivial homomorphism $\rho\colon \iet(\Lambda)'\to \iet(\Delta)$, the support of $\rho(\iet(\Lambda)')$ is a clopen subset $Z$ of $X_\Delta$, and $\rho$ extends to a continuous morphism  of pseudogroups $\widetilde{\rho}\colon \widetilde{\G_\Lambda}\to \widetilde{\G_\Delta|Z}$. 
\end{thm}

In order to deduce this from the Extension Theorem, we will need the following elementary property of (the Cantor version of ) interval exchange transformations. We include a proof for the convenience of the reader.

\begin{lemma} \label{l-iet-minimal}
Fix a dense finitely generated subgroup $\Lambda<\R/\Z$. Let $f\in \iet(\Lambda)$ and let $W\subset X_\Lambda$ be its support. Assume that $f$ does not have any periodic point in $W$. Then there exists a finite subset $\Sigma$ of $W$  such that the union of the $f$-orbits of points in $\Sigma$ is dense in $W$. \end{lemma}
\begin{proof}
We  view $f$ both as a homeomorphism of $X_\Lambda$ and as an interval exchange transformation acting on $\R/\Z$. The set $W$ is  clopen. Its projection $\pi_\Lambda(W)\subset \R/\Z$ is a disjoint union of finitely many closed intervals, and  is invariant under $f$. Let $\Sigma\subset W$ be the (finite) set of points that project onto some discontinuity point of $f$. We first claim that the union of orbits of points in $\Sigma$ is dense in $W$. Towards a contradiction, let $C\subset W$ be its closure and assume that $C\neq W$.  Then $\pi_\Lambda(C)$ is a compact invariant subset of $\pi_\Lambda(W)$, and $f$ has no discontinuity point in its complement, thus it permutes isometrically the connected components of the complement. Since moreover it preserves measure of every connected component, each one of them has a finite $f$-orbit, and thus we deduce that $f$ must have periodic points. This contradicts our assumption, thereby showing the claim. \end{proof}

\begin{proof}[Proof of Theorem \ref{t-iet-extends}]
The Cayley graphs of the groupoid $\G_\Delta$ are quasi-isometric to the Cayley graph of the group $\Delta$, which is a finitely generated abelian group, and thus has finite asymptotic dimension.  
Using Theorem \ref{t-asdim} we obtain that the support $Z\subset X_\Delta$ of ${\rho}(\iet(\Lambda)')$ is clopen and that $\rho$ extends to a continuous morphism $\widetilde{\rho}\colon \tG_\Lambda^{[r]}\to \tG_\Delta|Z$ for some $r \ge 1$ such that the map $q\colon Z\to X_\Lambda^{[r]}$ is surjective (see Remark \ref{r-comments}~\ref{i-comments-surjective}. We have to show that $r=1$. 

Fix $\lambda\in  \Lambda$ irrational. Since $\Lambda$ is naturally a subgroup of $\Ff({\G_\Lambda^{[r]}})$ for every $r\ge 1$ we can view $\lambda$ it as an element of $\tG_\Lambda^{[r]}$. Consider its image $\widetilde{\rho}(\lambda)\in \Ff(\G_\Delta|Z)$. 

The system $(U, \widetilde{\rho}(\lambda))$ factors onto  the system $(X_\Lambda^{[r]}, \lambda)$ via the map $q$. The latter factors onto $((\R/\Z)^{[r]}, \lambda)$, via the natural surjection $\pi_\Lambda\colon X_\Lambda^{[r]}\to (\R/\Z)^{[r]}$.   Thus, we have a factor map $\pi_\Lambda\circ q \colon Z\to (\R/\Z)^{[r]}$ from  the system $(Z, \widetilde{\rho}(\lambda))$ to $((\R/\Z)^{[r]}, \lambda)$. Since we chose $\lambda$ irrational, this system does not have periodic orbits, and thus neither does $(Z, \widetilde{\rho}(\lambda))$. Applying Lemma \ref{l-iet-minimal} to $f=\widetilde{\rho}(\lambda)$ and $W=Z$, there exists a finite subset $\Sigma$ such that the union of orbits of points in $\Sigma$ is dense in $Z$. 
However,  if $r\ge 2$, the space $(\R/\Z)^{[r]}$ can be written as the disjoint union of uncountably many non-empty closed  subsets
\[(\R/\Z)^{[r]}=\bigsqcup_{c\in [0, 1/2]} \Big\{\{x_1,\ldots, x_r\}\subset \R/\Z\colon \operatorname{diam}(\{x_1,\ldots, x_r\})=c\Big\}\]
where the diameter refers to the arc-distance on  $\R/\Z$. Since rotations are isometries, each of these sets is invariant under the system $\left((\R/\Z)^{[r]}, \lambda\right)$. Thus the preimage of these sets provide a partition of $Z$ into uncountably many $\widetilde{\rho}(\lambda)$-invariant closed non-empty subsets. This contradicts the existence of the finite set $\Sigma$ above unless $r=1$.   \qedhere


\end{proof}

To deduce Theorem \ref{t-sturmian-embedding} from Theorem \ref{t-iet-extends}, we will look at the homology groups of $\G_\Lambda$ and $\G_\Delta$ (the definition of homology of \'etale groupoid is recalled in \S \ref{s-homology}). Let us recall how to compute of the homology group $H_0(\G_\Lambda)$. It is well known that the action $\Lambda \acts X_\Lambda$ is uniquely ergodic,  i.e. admits a unique $\Lambda$-invariant probability measure  $\mu_\Lambda$. The measure $\mu_\Lambda$ is  characterised as the unique measure on $X_\Lambda$ which  projects to the Lebesgue measure on $\R/\Z$ (note that the set $\{\lambda_\pm  \colon  \lambda\in \Lambda\}$ is countable and null). The measure $\mu_\Lambda$ induces a homomorphism from the homology group $H_0(\G_\Lambda)$ to $\R$,  given by
\begin{align*} \iota_{\mu_\Lambda}\colon H_0(\G_\Lambda) & \longrightarrow  \R \\
  {[}\sum n_iU_i{]} & \mapsto  \sum n_i \mu_\Lambda(U_i) .\end{align*}
It  is well-known and not difficult to see that this map is injective and its image is equal to  $\widehat{\Lambda}$. Thus, it establishes an isomorphism between $H_0(\G_\Lambda)$ and $\widehat{\Lambda}$.

We will not use the fact that the map above is injective, but only the simpler observation that its image is equal to $\widehat{\Lambda}$. Let us briefly recall why this is true.  For every clopen set $U\subset X_\Lambda$, the image $\pi_\Lambda(U)<\R/\Z$ is a disjoint union of closed intervals $[a_i, b_i]$ with  $a_i, b_i\in \Lambda$. The measure of such an interval is clearly an element of  $ \widehat{\Lambda}$, and thus $\iota_{\mu_\Lambda}(H_0(\G_\Lambda))\subset \widetilde{\Lambda}$.  Let $\lambda_1,\ldots, \lambda_d$ be generators of $\Lambda$, and $\widehat{\lambda_1}, \cdots, \widehat{\lambda_d}\in \widehat{\Lambda}$ be their unique preimages in $(0, 1)$. These are the measures of the arcs $[0, \lambda_i]$ and thus belong to the image of $\iota_{\mu_\Lambda}$. Together with $1=\mu_\Lambda(X_\Lambda)$ they generate $\widehat{\Lambda}$, which is thus equal to the image of $\iota_{\mu_\Lambda}$.

We are now ready to prove Theorem \ref{t-sturmian-embedding}.

%
%

\begin{proof}[Proof of Theorem \ref{t-sturmian-embedding}]
It is obvious that \ref{i-sturmian-1}$\Rightarrow$\ref{i-sturmian-2}. To check that \ref{i-sturmian-3}$\Rightarrow$\ref{i-sturmian-1}, we view $\iet(\Lambda)$ and $\iet(\Delta)$ as groups of interval exchanges.   Assume that $c\in (0, 1]$ is such that $c\widehat{\Lambda}\le \widehat{\Delta}$. Since $1\in \widehat{\Lambda}$, we have $c\in \widehat{\Delta}$. Let $\bar{c}$ be the projection of $c$ to $\R/\Z$. If we rescale the length by $c$ to identify the arc $[0, \bar{c})$ with a unit circle we see that $\iet(\Lambda)$ is isomorphic to the subgroup of $\iet(\Delta)$ consisting of elements that act non trivially only on the arc $[0, \bar{c})$, and whose discontinuity points are projection of points in $c\widehat{\Lambda}\cap [0, c)$. 

 Let us prove that \ref{i-sturmian-2}$\Rightarrow$\ref{i-sturmian-3}, which is the main content of the theorem. Let $\rho\colon \iet(\Lambda)'\hookrightarrow \iet(\Delta)$ be  a group embedding. By Theorem \ref{t-iet-extends}, $\rho$ extends to a continuous pseudogroup morphism $\widetilde{\rho}\colon \widetilde{\G_\Lambda}\to \widetilde{\G_\Delta|Z}$ for some clopen subset $Z\subset X_\Delta$ equal to the support of $\rho(\iet(\Lambda)')$. We consider the morphism induced in homology $\widetilde{\rho}_\ast \colon H_0(\G_\Lambda)\to H_0(\G_\Delta|Z)$ (see Proposition \ref{p-morphism-homology}). Note that in degree zero it is simply given by   $\widetilde{\rho}_*([\sum n_i U_i])=[\sum n_i q^{-1}(U_i)]$, where $q\colon Z\to X_\Delta$ is the spatial component of $\widetilde{\rho}$.  Recall also that homology is invariant under taking restrictions  onto clopen subsets intersecting every orbits \cite{Mat-hom}, and thus we can see  $\widetilde{\rho}_\ast$ as taking values in $H_0(\G_\Delta)$.
 
We claim that the following square is commutative:
\[\begin{tikzcd}
H_0(\G_\Lambda) \arrow[r, "\widetilde{\rho}_\ast"] \arrow[d, "\iota_{\mu_\Lambda}"]
& H_0(\G_\Delta) \arrow[d, "\iota_{\mu_\Delta}" ] \\
\R \arrow[r,  "x\mapsto \mu_\Delta(Z) \cdot x"']
& \R \end{tikzcd} 
 \]
 To see that the diagram is commutative, observe that the pushforward $q_\ast (\mu_\Delta|_Z)$ is a positive invariant measure for the action $\Lambda \acts X_\Lambda$ with total mass $\mu_\Delta(Z)$.  By unique ergodicity, we must have $q_*(\mu_\Delta|_Z)=\mu_\Delta(Z)\mu_\lambda$. Thus for every class $[\sum n_i U_i]\in H_0(\G_\Lambda)$ we have
\begin{multline*}\iota_{\mu_\Delta}\circ \widetilde{\rho}_*([\sum n_i U_i])=\sum n_i \mu_\Delta(q^{-1}(U_i))=\sum n_i q_*\mu_\Delta(U_i)=\\ \mu_\Delta(Z)\sum n_i\mu_\Lambda(U_i)=\mu_\Delta(Z)\cdot \iota_{\mu_{\Delta}}([\sum n_i U_i]).\end{multline*}

Since the above diagram is commutative, it follows that the linear map $x\mapsto \mu_\Delta(Z) \cdot x$ sends $\iota_{\mu_\Lambda}(H_0(\G_\Lambda)=\widehat{\Lambda}$ inside $\iota_{\mu_\Delta}(H_0(\G_\Delta))=\widehat{\Delta}$; moreover if $Z=X_\Delta$ then $\mu_\Delta(Z)=1$ and thus $\widehat{\Lambda}\subset \widehat{\Delta}$. This concludes the proof. \qedhere

 \end{proof}

\subsection{An application based on the complexity function} To conclude this section, recall the well-known fact that the actions $\Lambda \acts X_\Lambda$ by Cantor rotations  are conjugate to subshifts whose complexity function is bounded above by a polynomial (see e.g. \cite[Lemma 5.13]{ExtAmen}). This fact makes it more difficult to construct embeddings of a given group in $\iet$, but this has never so far been used to prove obstructions to embeddings. Corollary \ref{c-complexity} implies the following.

\begin{cor}
Let $\G$ be a minimal effective groupoid over a Cantor space. Assume that $\G$ is expansive and that its orbital growth function $\beta_\G(n)$ and its complexity function $\pi_\G(n)$ do not both grow polynomially. Then there is no non-trivial homomorphism $\Af(\G)\to \iet$.  
\end{cor} 
\begin{proof}
Since $\G$ is expansive, the group $\Af(\G)$ is finitely generated \cite{Nek-simple}, and thus, every non-trivial homomorphism $\rho\colon \Af(\G)\to \iet$ takes values in $\iet(\Lambda)$ for some $\Lambda$. Therefore the conclusion follows from  Corollary \ref{c-complexity} and from the  well-known fact that the actions $\Lambda \acts X_\Lambda$ have polynomial complexity  \cite[Lemma 5.13]{ExtAmen}. \qedhere
\end{proof}

\section{One-sided SFT's and Higman-Thompson's groups} \label{s-SFT} 

In this subsection we use the Extension Theorem  to provide a concrete description of  all homomorphisms between the full groups of groupoids associated to \emph{one sided shifts of finite type}. Their full groups were studied in detail by Matui in \cite{Mat-SFT}, and  include the family  Higman-Thompson's groups $V_{n, r}$. 

\subsection{Basic definitions}
Let $\Sigma=(\mathcal{V}, \mathcal{E})$ be a directed graph, with finite sets of vertices $\mathcal{V}$ and of edges $\mathcal{E}$. We will always assume that $\Sigma$ is \textbf{irreducible}, i.e. for any vertices $x, y\in \mathcal{V}$ there is a directed path going from $x$ to $y$, and that $\Sigma$ is not a  cycle (hence there are at least two such paths for any pair of vertices). The starting and ending vertex of an edge $e\in \mathcal{E}$ are denoted $\mathsf{o}(e)$ and $\mathsf{t}(e)$. 

 We let $X_\Sigma$ be the set of all one sided infinite directed paths in $\Sigma$. Note that by a \emph{path} we mean its set of edges. The set $X_\Sigma$ is endowed with the topology induced from the product topology on $\mathcal{E}^\N$, which makes it homeomorphic to the Cantor set. The map 
\[S \colon X_\Sigma \to X_\Sigma,\quad e_0e_1e_2\cdots\mapsto e_1e_2\cdots,\]
is called the (one sided) \textbf{shift of finite type} associated to $\Sigma$. 

The map $S$ is not globally invertible, but it is invertible locally. More precisely, for every edge $e\in \mathcal{E}$ let $C_e\subset X_\Sigma$ be the clopen subset consisting of paths that start with $e$, and for every vertex $x\in \mathcal{V}$ let $D_x$ be the clopen set consisting of paths that start at $x$. Let 
\[S_e\colon C_e\to D_{\mathsf{t}(e)}\] 
be the restriction of $S_e$ to $C_e$, and note that it is a homeomorphism. We let $\tG_\Sigma$ be the pseudogroup generated by the maps $S_e, e\in \mathcal{E}$ (that is, the smallest pseudogroup over $X_\Sigma$ which contains the maps $S_e$), and $\G_\Sigma$ be the corresponding groupoid of germs. In what follows, we identify the maps $S_e$ with the sets of their germs, thus think of them both as homeomorphisms and as bisections of $\G_\Sigma$.

See \cite{Mat-SFT} for an extensive study of the groups $\Ff(\G_\Sigma)$.

\begin{example}[Higman-Thompson groups]\label{e-V} For $d\ge 2$, denote by $\Omega_d$ the a graph consisting of one vertex and $d$  loops, that we label by $\{1,\ldots , d\}$. The corresponding map $S\colon X_\Sigma\to X_\Sigma$  the full one-sided shift on $\{1, d\}^\N$. The group $\Ff(\G_{\Omega_d})$ is the group of all homeomorphism of $\{1,d\}^\N$ which are locally given by a prefix change. This is a well studied group, namely the  \emph{Higman-Thompson group} $V_{d, 1}$. The group $V_{2,1}$  is more simply known as \emph{Thompson's group}  $V$. In a similar way, the extended family of Higman-Thompson's groups $V_{d, r}$ arise as topological full groups of one-sided shifts of finite type, see \cite[Sec. 6.7.1]{Mat-SFT} (we do not the definition of the other groups $V_{d, r}$, as we will only discuss this example in the case $r=1$).  Each of these groups contain as subgroups the  Higman Thompson's groups $F_{d, r}$ and $T_{d, r}$ acting on the circle and the Cantor set respectively. 
\end{example}

\subsection{Combinatorial representation of elements}
In this subsection we recall a more combinatorial way to think of elements of the pseudogroup $\tG_\Sigma$ from \cite{Mat-Mats-Higman-Thompson} (close in sprit to the classical description of elements of Thompson's groups  as pairs of trees).

Fix an irreducible connected graph $\Sigma$. Given a finite directed path $w=e_1\cdots e_k$ in $\Sigma$, we denote $C_{w}\subset X_\Sigma$ the clopen subset of paths that start with $w$. We  set $S_w=S_{e_k}\cdots S_{e_1}$.   Let $v$ be another path with the same endpoint as $w$ (but not necessarily with the same starting point, nor with the same length). 
 The element $S_{v}^{-1}S_w$ is verifies  $\src(S_{v}^{-1}S_w)=C_w, \rg(S_{v}^{-1}S_w)=C_w$, and acts as a ``prefix replacement'':
\[S_{v}^{-1}S_w\colon C_w\to C_{v}, \quad w e_{k+1}e_{k+2}\cdots \mapsto v e_{k+1}e_{k+2}\cdots.\]

Two finite paths $v_1, v_2$ in $\Sigma$ are said to be \textbf{unrelated} if $C_{v_1}\cap C_{v_2}=\varnothing$, i.e. if neither of them is a prefix of the other.

\begin{defin}
A \textbf{table} is a (possibly infinite) collection of pairs of finite paths $\{(v_i, w_i)\}_{i\in I}$ with the following properties
\begin{enumerate}[label=(\roman*)]
\item  For every $i$ the path $w_i$ and $v_i$ end in the same vertex of $\Sigma$. 
\item For every $i\neq j$ the paths $v_i$ and $v_j$ are unrelated, and the paths $w_i$ and $w_j$ are unrelated.
\end{enumerate}  \end{defin}
The terminology is borrowed from \cite{Nek-Cuntz, Mat-Mats-Higman-Thompson}.

Every table $\{(v_i, w_i)\}_{i\in I}$ defines an element $F\in \tG_\Sigma$  with  $S(F)=\bigsqcup_{i} C_{w_i}$ and $\rg(F)=\bigsqcup_i C_{v_i}$, given by:

\begin{equation} \label{e-table} F=\bigsqcup_{i\in I} S_{v_i}^{-1} S_{w_i}\end{equation}
Conversely, every element of $\tG_\Sigma$ has this form for some table $\{(v_i, w_i)\}_{i\in I}$, and we say that $F$ is represented by the table.   By definition, the empty element of $\tG_{\Sigma}$ is represented by the empty table.

Note that $F\in \tG_\Sigma$ is represented by a finite table if and only if its domain and range are clopen  (i.e. if $F$ is a compact open bisection of $\G_\Sigma$).  An element $F\in \tG_\Sigma$ belongs to the full group $\Ff(\G_{\Sigma})$ if and only if it can be represented by a finite table $\{(w_i, v_i)\}_{i\in I}$ such that $\bigsqcup_{_i}C_{v_i}=\bigsqcup_{i} C_{w_i}=X_\Sigma$.

Two different tables can define the same element of $\tG_{\Sigma}$.   A \textbf{simple expansion} of a table  $\{(w_i, v_i)\}_{i\in I}$ is a table obtained by replacing  $(w_i, v_i)$ by the collection $\{(w_ie, v_ie)\}_{e\in \mathsf{t}(w_i)}$ for every $i$ in some subset $J\subset I$.  Two tables are said to be \textbf{equivalent} if and only if there exists a third table which is obtained from both through a series of simple expansions.
  The following is proven by a simple modification of the proof of \cite[Lemma 5.3]{Mat-Mats-Higman-Thompson}.
  
  \begin{lemma}\label{l-table-equivalent}
  Two tables define the same element of $\tG_{\Sigma}$ if and only if they are equivalent. 
  \end{lemma} 

The product of two elements $F, T\in \tG_{\Sigma}$ can be easily computed in terms of tables. 
First, let us define a partial binary operation on the set of pairs of finite paths. If $(v, w)$ and $(v', w')$ are pairs of path with $\mathsf{t}(v)=\mathsf{t}(w)$ and $\mathsf{t}(v')=\mathsf{t}(w'),$ we set

\[(v', w')\cdot (v, w)=\left\{\begin{array}{lr} (v'u, w) & \text{ if } w'=vu \\
(v', wu) & \text{ if } v=w'u  \\
\text{undefined} & \text{ otherwise }
\end{array}\right.\]

Note that $(v', w')\cdot (v, w)$ is defined if and only if the cylinders $C_{w'}$ and $C_v$ are not disjoint (and therefore one of them is contained into the other). 
The reason for this definition is that it describes products of elements of the form $S^{-1}_{v}S_w$. Namely  we have $(S^{-1}_{v'}S_{w'})(S^{-1}_{v}S_w)=S^{-1}_{v''}S_{w''}$ if $(v', w')\cdot (v, w) =(v'', w'')$, and $(S^{-1}_{v'}S_{w'})(S^{-1}_{v}S_w)=\varnothing$ if $(v', w')\cdot (v, w)$ is not defined. 
Since product in a pseudogroup behaves distributively with respect to disjoint unions, we immediately deduce: 

\begin{lemma} \label{l-product} \label{l-product-table}
Let $F, F'\in  \tG_\Sigma$ be represented  by tables $\{(v_i, w_i)\}_{i\in I}$ and $\{(v'_j, w'_j)\}_{j\in J}$. Then $F F'$ is represented by the table  
\[\{(v_i, w_i)\cdot (v'_j, w'_j)\}_{i, j},\]
 where $(i,j)\in I\times J$ varies over pairs of indices such that $(v_i, w_i)\cdot (v'_j, w'_j)$ is defined. 
 \end{lemma}


%

Note that it follows from this discussion that  the bisections of the form $S_{\ell'}^{-1}S_\ell$ form a basis of the topology of $\G_\Sigma$. Since all bisections of this form are products of $S_e, e\in \mathcal{E}\}$ and their inverses, it follows that the set $\Sc=\{S_e\colon e\in \mathcal{E}\}$ is an expansive generating set of  $\tG_\Sigma$.

\subsection{Description of homomorphisms} The structure of group isomorphisms and group embeddings between this family of groups has been studied by many authors. The case of isomorphisms is well-understood.  For the family of Higman-Thompson groups $V_{d, r}$ this can be traced back to Higman, who showed in particular that $V_{d, 1}$ and $V_{d', 1}$ are isomorphic if and only if $d=d'$, and that if $V_{d, r}$ and $V_{d', r'}$ are isomorphic, then $d=d'$ and $\mathsf{gcd}(d, r)=\mathsf{gcd}(d', r')$. Pardo then showed the converse to this statement  in \cite{Pardo}. Matui and Matsumoto gave a complete explicit description of the irreducible graphs $\Sigma_1, \Sigma_2$ that give rise to isomorphic groupoids $\G_{\Sigma_1}, \G_{\Sigma_2}$ \cite{Mat-Mats}, and Matui \cite{Mat-SFT} used his isomorphism theorem to deduce from this the classification of the groups $\Ff(\G_{\Sigma})$ up to isomorphism.  Bleak, Cameron, Maissel, Navas and Olukoya recently used Rubin's theorem to give a description of the  automorphism group of $V$ in terms of bi-synchronising automata \cite{aut-V}. 

Considerably less is understood about endomorphisms and embeddings within this family of groups. It is  well-known and easy to see that all the groups in the family  $V_{d, r}$ can be embedded into each other,  and that they all admit  injective, non-surjective endomorphisms. Several authors have pointed out different construction of embeddings and endomorphisms \cite{Bu-Cl-St-dist, Gen-V, Don-endV}, arising from the several points of view that have been developed on these groups. In \cite[Prop. 5.14]{Mat-prod}, Matui provides a way to construct homomorphisms between groups $\Ff(\G_\Sigma)$, and uses it  to show that they can all be embedded in Thompson's group $V=V_{2,1}$ \cite[Cor. 5.16]{Mat-prod}. In fact, a modification of his idea shows that all the groups $\Ff(\G_\Sigma)$ embed into each other (see Corollary \ref{c-SFT-biembed} below).
However no classification result of the possible homomorphisms is known. 

 In the reminder of this subsection, we will revisit Matui's construction in the language of this paper. We will then apply the Extension Theorem to show that, conversely, all embeddings between groups of the form $\Ff(\G_\Sigma)$ arise from this construction (modulo taking the restriction to the derived subgroup). 

\medskip

Fix  a graph $\Sigma=(\mathcal{E}, \mathcal{V})$, as above.  Observe that the set of bisections $\{S_e\}_{e\in \mathcal{E}}$ verify the following:
\[\bigsqcup_{e\in \mathcal{E}} \src(F_e)= X_\Sigma, \quad  \rg(S_e)=\bigsqcup_{\mathsf{o}(e)=\mathsf{t}(f)} \src(S_f).\]

 Our first goal is to show that in order to define a continuous morphisms ${\tG_\Sigma}$ to another pseudogroup $\tH$, it is necessary and sufficient to find a family of elements of $\tH$ satisfying the same combinatorial relations.   To this end, we give the following terminology:
%
\begin{defin}
Let  $\tH$ be a pseudogroup over a  space $Y$.  A $\Sigma$-\textbf{system} in $\tH$  is a family $\{T_e\}_{e\in \mathcal{E}}$ of non-zero elements of $\tH$ with the following properties:  
\begin{enumerate}[label=(\roman*)]
\item the sets $\src(T_e), e\in \mathcal{E}$ are clopen and  pairwise disjoint;
\item for every $e\in \mathcal{E}$ we have $\rg(T_e)=\bigsqcup_{\mathsf{o}(e)=\mathsf{t}(f)} \src(T_f)$.
\end{enumerate}
The set $Z=\bigsqcup_{e\in \mathcal{E}} \src(S_e)$ is called the \textbf{domain} of  the $\Sigma$-system.

\end{defin}

Let $\{T_e\}_{e\in \mathcal{E}}$ be a $\Sigma$-system with domain $Z$. For every $y\in Z$, we let $e_y\in \mathcal{E}$ be the unique edge such that $y\in \src(T_e)$.
 Define a continuous map $T\colon Z\to Z$ by $T(y)=T_{e_y}(y)$.
Further define a  map $q\colon Z\to X_{\Sigma}$ by coding  the dynamical system $(T, Z)$, namely 
\[q(y)=e_ye_{T(y)} e_{T^2(y)}\cdots .\]  
It is readily checked that this map is continuous, and verifies $q\circ T=S\circ q$, where $S\colon X_{\Sigma}\to X_{\Sigma}$ is the shift. The map $q$ will be called the \textbf{coding map} associated to the $\Sigma$-system.

\begin{prop} \label{p-morphism-cuntz} \label{p-morphism-Cuntz} \label{p-Cuntz}
Fix a graph $\Sigma$ as above. Let $\tH$ be a pseudogroup over a space $Y$, and let $Z\subset Y$ be a clopen subset of $Y$. For every continuous morphism $\tG_\Sigma\to \widetilde{\H|Z}$, the collection $\{\varphi(S_e)\}_{e\in \mathcal{E}}$ is a $\Sigma$-system in $\tH$ with domain $Z$.

 Conversely let  $\{T_e\}_{e\in \mathcal{E}}$ be a  $\Sigma$-system in $\tH$, with domain $Z$.  Then the association $S_e\mapsto T_e, e\in \mathcal{E}$ extends to a unique continuous morphism of pseudogroups $\varphi \colon \tG_{\Sigma}\to \widetilde{\H|Z}$, where $Z$ is the domain of the $\Sigma$-system. Moreover, the spatial component $q\colon Z\to X_{\Sigma}$ of $\varphi$ coincides with the coding map of $\{T_e\}_{e\in \mathcal{E}}$.

\end{prop}

\begin{remark}
The reader familiar with Cuntz-Krieger $C^*$-algebras may  notice a tight connection. Namely a family $\{T_e\}_{e\in\mathcal{E}}$ is a $\Sigma$-system if and only if the indicator functions $1_{T_e}$ satisfy the Cuntz-Krieger relations in $C^*_{red}(\H)$. Hence a $\Sigma$-system gives rise to a unital embedding of the Cuntz-Krieger $C^*$-algebra $C^*_{red}(\G_{\Sigma})$ into $C^{*}_{red}(\H|Z)$, which is turn gives rise  to an embedding of $\Ff(\G_{\Sigma})$ into $\Ff(\H|Z)\le \Ff(\H)$. This argument appears implicitly in Matui's proof of \cite[Prop. 5.14]{Mat-prod}.  Below we give instead a direct argument (not using $C^*$-algebras). 
 \end{remark}

\begin{proof}
The first part of the statement is a straightforward consequence of Proposition \ref{p-spatial-equivariant}. Let us prove the converse.

 For every path $w=e_1\cdots e_k$ in $\Sigma$, set $T_w=T_{e_k}\cdots T_{e_1}$. It is clear that if $\varphi$ exists, it must be unique. Namely for every $F\in \tG_\Sigma$,  applying $\varphi$ to  \eqref{e-table} we obtain
 \begin{equation} \varphi(F)=\bigsqcup_{i\in I}T_{v_i}^{-1}T_{w_i}, \label{e-phi-unique} \end{equation}
where $\{(v_i, w_i)\}_{i\in I}$ is a table representing $F$. Thus, we want to show that the above formula defines indeed a continuous morphism.

  Let $q\colon U\to X_{\Sigma}$ be the coding map of the $\Sigma$-system. The following lemma is obvious from the definition of the map $q$. 
  \begin{lemma} \label{l-preimage}
  For every path $w$ we have $\src(T_w)=q^{-1}(C_w)$, and $\rg(T_w)=q^{-1}(D_{\mathsf{t}(w)})$. 

  \end{lemma}
  
 The lemma implies that whenever $v$ and $w$ are paths with $\mathsf{t}(v)=\mathsf{t}(w)$, we have $\src(T_v^{-1}T_w)=q^{-1}(C_w)$ and $\rg(T_v^{-1}T_w)=q^{-1}(C_v)$. Thus, if $\{(v_i, w_i)\}_{i\in I}$ is a table, then for every $i\neq j$ we have $\src(T_{v_i^{-1}}T_{w_i})\cap \src(T_{v_j^{-1}}T_{w_j})=q^{-1}(C_{w_i}\cap C_{w_j})=\varnothing$ and similarly $\rg(T_{v_i^{-1}}T_{w_i})\cap \rg(T_{v_j^{-1}}T_{w_j})=\varnothing$. It follows that the right hand side of \eqref{e-phi-unique} is indeed an element of $\tH$. 
 Let us check that it is well-defined, i.e. that it does not depend on the choice of a table representing $F$. Note that  for any path $w=w_1w_2$ we have $T_w=T_{w_2} T_{w_1}$. Thus for any paths $w, v$ with $\mathsf{t}(w)=\mathsf{t}(v)$ we have
 \begin{multline*}\bigsqcup_{e\in \mathsf{t}(v)} T^{-1}_{ve}T_{we}=\bigsqcup_{e\in \mathsf{t}(v)} T_v^{-1} T_e^{-1}T_e T_w=T_{v^{-1}}\left( \bigsqcup_{e\in \mathsf{t}(v)} \src(T_e)\right) T_w=\\T_{v^{-1}}\left( \bigsqcup_{e\in \mathsf{t}(v)} q^{-1}(C_e)\right) T_w=T_{v^{-1}}q^{-1}(D_{\src(e)}) T_w=T_{v^{-1}} \rg(T_w)T_w=T_{v}^{-1}T_w.\end{multline*}
 
 Therefore if a table $\{(v'_i, w'_i)\}_{i\in I}$ is obtained from $\{(v_i, w_i)\}_{i\in I}$ by a simple expansion, then $\bigsqcup_{i\in I}T_{v_i}^{-1}T_{w_i}=\bigsqcup_{i\in I}T_{v'_i}^{-1}T_{w'_i}$. It follows from Lemma \ref{l-table-equivalent} that $\varphi(F)$ does not depend on the choice of a table representing $F$. 
 
 Let us check that $\varphi$ is a semigroup homomorphism. We use the description of the product in $\tG_\Sigma$ in terms of tables given by Lemma \ref{l-product-table}. First let $(v, w)$ and $(v', w')$ be pairs of paths such that $\mathsf{t}(v)=\mathsf{t}(w)$ and  $\mathsf{t}(v')=\mathsf{t}(w')$. Assume that $(v', w')\cdot (v, w)$ is defined and equals $(v'', w'')$, and let us show that $(T_{v'}^{-1}T_{w'})(T_{v}^{-1}T_{w})=(T_{v''}^{-1}T_{w''})$. Assume, say, that $w'=vu$, so that $(v'', w'')=(v, wu)$ (the other case is done similarly). Then, using  Lemma \ref{l-preimage},  we have 
 \begin{multline*}(T_{v'}^{-1}T_{w'})(T_{v}^{-1}T_{w})=T_{v'}^{-1}T_{vu}T_{v}^{-1}T_{w}=T_{v'}^{-1}T_u(T_{v}T_{v}^{-1})T_{w}=
 T_{v'}^{-1}T_u \rg(T_v)T_{w}=\\ T_{v'}^{-1}T_u q^{-1}(D_{\mathsf{t}(v)})T_{w}= T_{v'}^{-1}T_u \rg(T_w)T_{w}=T_{v'}^{-1}T_u T_{w}=T_{v'}^{-1}T_{wu}=T_{v''}^{-1}T_{w''}.
 \end{multline*}
 If $(v',w')\cdot (v, w)$ is not defined then the cylinders $C_{w'}$ and $C_{v}$ are disjoint, and therefore so are $\src(T_{v'}^{-1}T_{w'})=q^{-1}(C_{w'})$ and $\rg(T_v^{-1}T_w)=q^{-1}(C_v)$. Thus, in this case we have $(T_{v'}^{-1}T_{w'})(T_{v}^{-1}T_{w})=\varnothing$. We deduce that $\varphi$ is a semi-group homomorphism, by Lemma \ref{l-product-table}. 
 
It is clear that $\varphi$ preserves unions of compatible families, and that it sends $\varnothing$ to $\varnothing$. Moreover, it sends $X_\Sigma$ to $U$, since $X_\Sigma =\bigsqcup_{e\in \mathcal{E}} S_e^{-1} S_e\mapsto \bigsqcup_{e\in \mathcal{E}} T_e^{-1}T_e=\bigsqcup_{e\in \mathcal{E}} \src(T_e)= U$. Thus $\varphi$ is a continuous morphisms of pseudogroups. 

The fact that the map $q$ coincides with the spatial component of $\varphi$ is clear from Lemma \ref{l-preimage}. \qedhere

\end{proof}
\begin{cor}
If $\tH$ contains a $\Sigma$-system, then there is a group embedding $\Ff(\G_\Sigma)\to \Ff(\H)$. 

\end{cor}

We now want to apply our Extension Theorem to show that all embeddings between groups in this family arise as described. 

To do so, we use the following fact that we have  learned  from V. Nekrashevych.

\begin{lemma}\label{l-tree}
For every $x\in X_\Sigma$ the Cayley graph $\cay_x(\G_\Sigma, \{ S_e\colon e\in \mathcal{E}\})$ is a tree. \end{lemma} 
\begin{proof}
Set $\Sc=\{ S_e\colon e\in \mathcal{E}\}$. Let $x\in X_\Sigma$ and assume that $x$ belongs to a cycle in $\cay_x(\G_\Sigma, \Sc)$ which is not homotopically trivial. Let $S_{e_1}^{\varepsilon_1},\ldots, S_{e_n}^{\varepsilon_n}$ be the word  read on the labels of edges of the cycle, where $\epsilon_i\in \{-1, 1\}$. We can assume that the cycle contains no backtracking, that is $e_{i}=e_{i+1}\Rightarrow \varepsilon_i=\varepsilon_{i+1}$. Then fact that these elements label a cycle in  $\cay_x(\G_\Sigma, \Sc)$ means that the partial homeomorphism $S_{e_n}^{\varepsilon_n}\cdots S_{e_1}^{\varepsilon_1}$ has trivial germ at $x$. Thus, it coincides with the identity on some  neighbourhood $U$ of $x$ contained in  $\src(S_{e_n}^{\varepsilon_n}\cdots S_{e_1}^{\varepsilon_1})$. Let  $y\in U$ be a point which is not an eventually periodic sequence (note that such sequences are dense). Let $y_0=y$, and $y_{i}=S_{e_i}^{\varepsilon_i}\cdots S_{e_1}^{\varepsilon_1}(y)$. Note that $y_{i+1}$ is obtained from $y_i$ by either appending or removing $e_{i+1}$, depending on the value of $\varepsilon_i$. Since there is no cancellation in this process and $y$ is not eventually periodic, we deduce that $y_n\neq y$. This contradicts the fact that $S_{e_n}^{\varepsilon_n}\cdots S_{e_1}^{\varepsilon_1}$ is the identity on $U$. \qedhere\end{proof}

In particular, the Cayley graphs of $\G_\Sigma$ have finite asymptotic dimension, equal to one. Therefore Theorem \ref{t-asdim} can be applied, an yields the following. 

\begin{thm} \label{t-SFT}
Let $\G_1$ be any minimal effective \'etale groupoid over a Cantor space $X_1$, and let $\G_{\Sigma_2}$ be a groupoid associated to irreducible one sided shifts of finite type. Then for every non-trivial group homomorphism $\rho \colon \Af(\G_1)\to \Ff(\G_{\Sigma_2})$, the support of $\rho(\Af(\G_1))$ is a clopen subset $U\subset X_{\Sigma_2}$ and $\rho$ extends to a continuous morphism of pseudogroups $\widetilde{\rho}\colon \tG_1\to \tG_{\Sigma_2}|U$.

\end{thm}

Thus, when $\G_1=\G_{\Sigma_1}$ is also associated to a one-sided SFT,  we finally arrive at the following explicit classification for homomorphisms.
\begin{thm} \label{t-SFT-explicit}
Let $\tG_{\Sigma_1}$ and $\tG_{\Sigma_2}$ be groupoids associated to irreducible one-sided-shifts of finite type. 
Then non-trivial group homomorphisms $\rho\colon \Df(\G_{\Sigma_1})\to \Ff(\G_{\Sigma_2})$ are in natural one-to-one correspondence with  $\Sigma_1$-systems in $\tG_{\Sigma_2}$. 

This correspondence is explicitly described as follows: a $\Sigma_1$-system $\{T_e\}_{e\in \mathcal{E}_1}\subset \widetilde{\G_{\Sigma_2}}$ with domain $U$ gives rise to a homomorphism $\rho\colon\Df(\G_{\Sigma_1})\to  \Ff(\G_{\Sigma_2}|Z)$  given by
\[\rho(g)=\bigsqcup_{i=1}^n T_{v_i}^{-1} T_{w_i}\]
where $\{(v_i, w_i)\}_{i=1}^n$ is a table which defines $g$, and for every finite path $w=e_1\cdots e_k$ we set $T_w=T_{e_k}\cdots T_{e_1}$.  
In particular,  the $\rho$-action of $\Df(\G_{\Sigma_1})$ on $Z$ factors onto $X_{\Sigma_1}$ via the coding map $q\colon Z\to X_{\Sigma_1}$ associated to  $\{T_e\}_{e\in \mathcal{E}}$.
\end{thm}

%
\begin{proof}
Combine Theorem \ref{t-SFT} with Proposition \ref{p-Cuntz}. 

\end{proof}

\begin{example}[Endomorphisms of Thompson's $V$] It is well-known that Thompson's group $V$ admits many embeddings into itself. Theorem \ref{t-SFT-explicit} provides a way to parametrise them all explicitly in terms of families of binary words satisfying certain conditions, as follows. 

We see $V=\Ff(\G_{\Omega_2})$ as in Example \ref{e-V} (since $V$ is simple, we do not need to pass to the derived subgroup). Let  $T_0, T_1\in \tG_{\Omega_2}$ be compact open bisections  represented by finite tables $\{(v_i^{(0)}, w_i^{(0)})\}_{i=1}^n$ and $\{(v_i^{(1)}, w_i^{(1)})\}_{i=1}^m$. The pair $T_0, T_1$ is an is an $\Omega_2$-system if and only if $\src(T_0)\cap \src(T_1)=\varnothing$ and $\src(T_0)\sqcup \src(T_1)=\rg(T_0)=\rg(T_1)$. In terms of tables is equivalent to
\[C_{w_1^{(0)}}\sqcup \cdots C_{w_n^{(0)}}\sqcup C_{w_1^{(1)}}\sqcup\cdots \sqcup C_{w_m^{(1)}}=C_{v_1^{(0)}}\sqcup \cdots C_{v_n^{(0)}}= C_{v_1^{(1)}}\sqcup\cdots \sqcup C_{v_m^{(1)}}.\]
These translate to elementary conditions on the binary words $v_i^{(j)}, w_i^{(j)}$. 
Thus, we obtain an explicit parametrisation of all  non-trivial  endomorphisms $V\to V$ by pairs of finite tables $\{(v_i^{(0)}, w_i^{(0)})\}_{i=1}^n$ and $\{(v_i^{(1)}, w_i^{(1)})\}_{i=1}^m$ which satisfy the above conditions (considered modulo equivalence of tables).

 This correspondence is explicit, in the sense that if $\rho\colon V\to V$ is the homomorphism associated to a pair  $\{(v_i^{(0)}, w_i^{(0)})\}_{i=1}^n, \{(v_i^{(1)}, w_i^{(1)})\}_{i=1}^m$, then for every $g\in V$, a table representing $\rho(g)$ can be computed explicitly from a table representing $g$ and the tables  $\{(v_i^{(0)}, w_i^{(0)})\}_{i=1}^n, \{(v_i^{(1)}, w_i^{(1)})\}_{i=1}^m$, by expanding the formula in the statement of the theorem via elementary (but perhaps long!) table computations. 
 
 Note also that the coding map $q\colon Z\to X_{\Omega_2}$ can be computed explicitly in terms of the tables, and provides us with a factor map of the the action of $V$ on $Z$ induced by the endomorphism with its natural action on $X_{\Omega_2}$. 
 
Clearly there is nothing special about Thompson's group $V$ in  this discussion, and  similar considerations hold  for general homomorphisms $\Df(\G_{\Sigma_1})\to \Ff(\G_{\Sigma_2})$.

\end{example}

\subsection{Homological obstructions to embeddings}

To give a concrete consequence of Theorem \ref{t-SFT-explicit}, let us show how the homology of \'etale groupoids can be used to provide constraints on the possible embeddings between full groups of one-sided SFT's.

Let us first recall the description of the homology group $H_0(\G_\Sigma)$. Let $M_\Sigma$ be the adjacency matrix of the graph $\Sigma$,  that is the $|\mathcal{V}\times \mathcal{V}|$ matrix $M_\Sigma=(m_{x,y})_{x, y\in \mathcal{V}}$ where $m_{xy}=|\{e\in \mathcal{E}\colon \mathsf{o}(e)=x, \mathsf{t}(e)=y\}|$. 
The homology group $H_0(\G_\Sigma)$ is given by 
\[H_0(\G_\Sigma)\simeq \operatorname{Coker}_\Z(\operatorname{Id}-M_\Sigma^t)= \mathbb{Z}^{\mathcal{V}}/(\operatorname{Id}-M_\Sigma^t).\]

This isomorphism is described as follows. Define a homomorphisms $\Z^{\mathcal{V}}\to H_0(\G_\Sigma, \Z)$, mapping each basis vector $e_x, x \in \mathcal{V}$  to the  class  $[D_x]$. This map is surjective, and its kernel is precisely the image of the matrix $\operatorname{Id}-M_\Sigma^t$, see  \cite[\S 4.2]{Mat-hom}.

  Recall that an \'etale groupoid $\H$ over a space $Y$ is said to be \emph{purely infinite} if for every open set $U\subset Y$ there exist $F_1, F_2\in \tH$ such that $\src(F_1)=S(F_2)=U$,  $\rg(F_1)\subset, \rg(F_2)\subset U$, and $\rg(F_1)\cap \rg(F_2)=\varnothing$. When $\H$ is purely infinite, Matui has established a homological  criterion to ensure the existence of a $\Sigma$-system.
\begin{prop}[Matui] \label{p-Matui-embedding} Fix an irreducible graph $\Sigma$. Let $\H$ be a minimal purely infinite \'etale groupoid over a Cantor space $Y$. For every clopen subset $U\subset Y$, the following are equivalent. 
\begin{enumerate}[label=(\roman*)]
\item \label{p-homology-i} There exists a $\Sigma$-system in $\tH$ with domain $U$.  
\item \label{p-homology-ii} There exists a homomorphism $H_0(\G_{\Sigma})\to H_0(\H)$ which maps $[{X_{\Sigma}}]$ to $[U]$.
\end{enumerate}

\end{prop}
\begin{proof}
The implication \ref{p-homology-i}$\Rightarrow$\ref{p-homology-ii} holds for every $\H$, as already explained above. The implication \ref{p-homology-ii}$\Rightarrow$\ref{p-homology-i} is proven at the beginning of the proof of \cite[Prop. 5.14]{Mat-prod}.    \end{proof}

\begin{cor}
If $\H$ is purely infinite, then $\tH$ always admits a $\Sigma$-system.
In particular, the group $\Ff(\G_\Sigma)$ can be embedded in $\Ff(\H)$. 
\end{cor}
\begin{proof}
It is enough to show that there exists a non-empty clopen set $U\subset Y$ with $[U]=0$ in $H_0(\H)$, since then the zero map $H_0(\G_\Sigma)\to \H_0(\H)$ will map $[{X_{\Sigma}}]$ to $[U]$. Let  $V\subset Y$ be an arbitrary clopen set. Using that  $\H$ is purely infinite, let $F\in \tH$ be such that $\src(F)=V$ and $\rg(F)$ is strictly contained in $V$. Then we have $[{V\setminus \rg(F)}]=[1_V]-[{ \rg(F)}]=0$ and hence the clopen set $U=V\setminus \rg(F)$ satisfies the desired conclusion.
The fact that $\Ff(\G_\Sigma)$ embeds into $\Ff(\H)$ therefore follows from Proposition \ref{p-morphism-Cuntz}
\end{proof}
Since groupoid associated to irreducible shifts of finite type are themselves purely infinite, this has the following consequence. 
\begin{cor} \label{c-SFT-biembed}
For any irreducible connected graphs $\Sigma_1$ and $\Sigma_2$, the groups $\Ff(\Sigma_1)$ and $\Ff(\Sigma_2)$ can be embedded into each other. 
\end{cor}

The combination of our results with Matui's Proposition \ref{p-Matui-embedding} implies that although all the groups $\Ff(\G_{\Sigma})$ admit many embeddings into each other,   the possible support of the image is completely characterised as follows.  
\begin{thm}\label{t-SFT-support}
Let $\tG_{\Sigma_1}, \tG_{\Sigma_2}$ be groupoids associated to one sided shifts of finite type, and let $Z\subset X_{\Sigma_2}$ be a clopen subset.
Then the following are equivalent.
\begin{enumerate}[label=(\roman*)]
\item \label{i-SFT-support-embed} There exists a group embedding $\rho\colon \Df(\G_{\Sigma_1})\to \Ff(\G_{\Sigma_2})$ such that $\rho(\Df(\G_{\Sigma_1}))$ has support equal to $Z$.

\item \label{i-SFT-support-homology} There exists a homomorphism $\varphi \colon \mathsf{H}_0(\G_{\Sigma_1})\to \mathsf{H}_0(\G_{\Sigma_2})$ such that $\varphi([{X_{\Sigma_1}}])=[Z]$.

\end{enumerate}

\end{thm}
\begin{proof}
The implication \ref{i-SFT-support-homology}$\Rightarrow$ \ref{i-SFT-support-embed} is due to Matui, see Proposition \ref{p-Matui-embedding}. The converse follows from Theorem \ref{t-SFT} and from the functoriality of homology with respect to continuous morphisms of pseudogroups (see \S \ref{s-homology}). \qedhere

\end{proof}

This shows that for many pairs $(\Sigma_1, \Sigma_2)$, all embeddings $\Df(\G_{\Sigma_1})\to \Ff(\G_{\Sigma_2})$ must give rise to an action with global fixed points. For example, we have the following for the  Higman-Thompson groups.

\begin{cor} \label{c-V-support}
The following are equivalent:
\begin{enumerate}[label=(\roman*)]
\item there exist an embedding  $\rho\colon V_{n, 1}' \to V_{m, 1}$ such that the action of $\rho(V_{n, 1}')$ has no global fixed points; 
\item $m-1$ divides $n-1$. 
\end{enumerate}
\end{cor}
\begin{proof}
The homology $H_0(\G_{\Omega_n})$ is isomorphic to $\Z/(n-1)\Z$, with generator given by $[{X_{\Omega_n}}]$. Thus, the conclusion follows from Theorem \ref{t-SFT-support}.
\end{proof}

In particular, for every $n>2$, every embedding $V\to V_{n, 1}$ must give rise to an action with global fixed points. On the other hand, all the groups $V_{n, 1}'$ admit an embedding into $V$ without any global fixed point. 

\subsection{Products of SFT's and Brin's groups $nV$} \label{s-product-SFT}
Matui studies in \cite{Mat-prod} \'etale groupoids arising from products of such groupoids, i.e. \'etale groupoids of the form $\G=\G_{\Sigma_1}\times \cdots \times \G_{\Sigma_n}$, where each $\G_{\Sigma_n}$ is the groupoid associated to an irreducible one-sided shift of finite type.  As an illustration of Corollary \ref{c-asdim-obstruction}, we have the following. 
\begin{cor} \label{c-SFT-prod}
Let $\G, \H$  be respectively the products of $n$ and $m$ \'etale groupoids arising from irreducible shifts of finite type. If $n>m$, then every homomorphism $\Ff(\G)\to \Ff(\H)$ has abelian image. 
\end{cor}
\begin{proof}
The derived subgroup $\Df(\G)$ is simple by \cite[Th. 4.2]{Mat-SFT} since $\G$ is purely infinite (see \cite[Sec. 5.2]{Mat-prod}), and therefore we have $\Df(\G)=\Af(\G)$. Moreover it follows from Lemma \ref{l-tree} that every Cayley graph of $\G$ is a product of $n$ trees and thus it has asymptotic dimension $n$, while every Cayley graph of $\H$ has asymptotic dimension $m$.  The conclusion follows from Corollary \ref{c-asdim-obstruction}. \qedhere
\end{proof}

A special case of topological full groups of products of one sided shifts of finite type is the family of higher dimensional Higman-Thompson's groups $nV$, introduced by Brin \cite{Brin-nV}. The group $nV$ coincides with the topological full group of the product $\G=\G_\Sigma\times \cdots \times \G_\Sigma$, where $\G_\Sigma$  is as in Example \ref{e-V}. 

It was shown in \cite{Brin-nV} that the group $2V$ is  not isomorphic to the group $V$ as a consequence of Rubin's theorem \cite{Rubin}. Later this was extended in \cite{Bl-La-nV} to show that $nV$ and $mV$ are isomorphic if and only if $n=m$, also using Rubin's theorem. Corollary \ref{c-SFT-prod} implies the following. 
\begin{cor}\label{c-nV}
The group $nV$ can be embedded in $mV$ if and only if $n\le m$. 
\end{cor}

\bibliographystyle{alpha}
\bibliography{FullGroups.bib}

\newcommand{\etalchar}[1]{$^{#1}$}
\begin{thebibliography}{JMBMdlS18}

\bibitem[AV17]{AV-speed}
Gideon Amir and B\'{a}lint Vir\'{a}g.
\newblock Speed exponents of random walks on groups.
\newblock {\em Int. Math. Res. Not. IMRN}, (9):2567--2598, 2017.

\bibitem[BBF15]{BBF-asdim}
Mladen Bestvina, Ken Bromberg, and Koji Fujiwara.
\newblock Constructing group actions on quasi-trees and applications to mapping
  class groups.
\newblock {\em Publ. Math. Inst. Hautes \'{E}tudes Sci.}, 122:1--64, 2015.

\bibitem[BCM{\etalchar{+}}16]{aut-V}
Collin Bleak, Peter Cameron, Yonah Maissel, Andr\'{e}s Navas, and Feyishayo
  Olukoya.
\newblock The further chameleon groups of {R}ichard {T}hompson and {G}raham
  {H}igman: {A}utomorphisms via dynamics for the {H}igman groups ${G}_{n, r}$.
\newblock {\em arXiv:1605.09302}, 2016.

\bibitem[BCS01]{Bu-Cl-St-dist}
J.~Burillo, S.~Cleary, and M.~I. Stein.
\newblock Metrics and embeddings of generalizations of {T}hompson's group
  {$F$}.
\newblock {\em Trans. Amer. Math. Soc.}, 353(4):1677--1689, 2001.

\bibitem[BD06]{Be-Dra-poly}
G.~C. Bell and A.~N. Dranishnikov.
\newblock A {H}urewicz-type theorem for asymptotic dimension and applications
  to geometric group theory.
\newblock {\em Trans. Amer. Math. Soc.}, 358(11):4749--4764, 2006.

\bibitem[BE14]{Ba-Er-givengrowth}
Laurent Bartholdi and Anna Erschler.
\newblock Groups of given intermediate word growth.
\newblock {\em Ann. Inst. Fourier (Grenoble)}, 64(5):2003--2036, 2014.

\bibitem[BG00]{Ba-Gri-spec}
L.~Bartholdi and R.~I. Grigorchuk.
\newblock On the spectrum of {H}ecke type operators related to some fractal
  groups.
\newblock {\em Tr. Mat. Inst. Steklova}, 231(Din. Sist., Avtom. i Beskon.
  Gruppy):5--45, 2000.

\bibitem[BGN03]{Ba-Gr-Nek-fra}
Laurent Bartholdi, Rostislav Grigorchuk, and Volodymyr Nekrashevych.
\newblock From fractal groups to fractal sets.
\newblock In {\em Fractals in {G}raz 2001}, Trends Math., pages 25--118.
  Birkh\"{a}user, Basel, 2003.

\bibitem[BJS10]{tiling1}
Jean Bellissard, Antoine Julien, and Jean Savinien.
\newblock Tiling groupoids and {B}ratteli diagrams.
\newblock {\em Ann. Henri Poincar\'e}, 11(1-2):69--99, 2010.

\bibitem[BL10]{Bl-La-nV}
Collin Bleak and Daniel Lanoue.
\newblock A family of non-isomorphism results.
\newblock {\em Geom. Dedicata}, 146:21--26, 2010.

\bibitem[Bri04]{Brin-nV}
Matthew~G. Brin.
\newblock Higher dimensional {T}hompson groups.
\newblock {\em Geom. Dedicata}, 108:163--192, 2004.

\bibitem[Bri13]{Brieu-entropy}
J\'{e}r\'{e}mie Brieussel.
\newblock Behaviors of entropy on finitely generated groups.
\newblock {\em Ann. Probab.}, 41(6):4116--4161, 2013.

\bibitem[BST12]{Be-Sc-Ti-asdim}
Itai Benjamini, Oded Schramm, and \'{A}d\'{a}m Tim\'{a}r.
\newblock On the separation profile of infinite graphs.
\newblock {\em Groups Geom. Dyn.}, 6(4):639--658, 2012.

\bibitem[CFP96]{C-F-P-th}
J.~W. Cannon, W.~J. Floyd, and W.~R. Parry.
\newblock Introductory notes on {R}ichard {T}hompson's groups.
\newblock {\em Enseign. Math. (2)}, 42(3-4):215--256, 1996.

\bibitem[CG04]{Ca-Go-asdim}
Gunnar Carlsson and Boris Goldfarb.
\newblock On homological coherence of discrete groups.
\newblock {\em J. Algebra}, 276(2):502--514, 2004.

\bibitem[Cha18]{Cha-F}
Maksym Chaudkhari.
\newblock Confined subgroups of thompson's group f and its embeddings into
  wobbling groups.
\newblock {\em arXiv:1809.05146}, 2018.

\bibitem[Cor14]{Cor-Bou}
Yves Cornulier.
\newblock {G}roupes pleins-topologiques [d'apr\`es {M}atui, {J}uschenko,
  {M}onod,...].
\newblock {\em Ast\'erisque}, (361):Exp. No. 1064, 2014.
\newblock S{\'e}minaire Bourbaki. Vol. 2012/2013.

\bibitem[Cor15]{Cor-FM}
Yves Cornulier.
\newblock Irreducible lattices, invariant means, and commensurating actions.
\newblock {\em Math. Z.}, 279(1-2):1--26, 2015.

\bibitem[CSGdlH99]{Ce-Gri-Ha}
T.~Ceccherini-Silberstain, R.~I. Grigorchuk, and P.~de~la Harpe.
\newblock Amenability and paradoxical decompositions for pseudogroups and
  discrete metric spaces.
\newblock {\em Tr. Mat. Inst. Steklova}, 224(Algebra. Topol. Differ. Uravn. i
  ikh Prilozh.):68--111, 1999.

\bibitem[DFG13]{D-F-G}
Fran\c{c}ois Dahmani, Koji Fujiwara, and Vincent Guirardel.
\newblock Free groups of interval exchange transformations are rare.
\newblock {\em Groups Geom. Dyn.}, 7(4):883--910, 2013.

\bibitem[DFG17]{D-F-G-solv}
Fran\c{c}ois Dahmani, Koji Fujiwara, and Vincent Guirardel.
\newblock Solvable groups of interval exchange transformations.
\newblock {\em arXiv:1701.00377}, 2017.

\bibitem[DO07]{Don-endV}
Casey Donoven and Feyishayo Olukoya.
\newblock Conjugate subgroups and overgroups of {$V_n$}.
\newblock {\em arXiv:1710.00913}, 2007.

\bibitem[Dye59]{Dye}
H.~A. Dye.
\newblock On groups of measure preserving transformation. {I}.
\newblock {\em Amer. J. Math.}, 81:119--159, 1959.

\bibitem[EM13]{E-M}
G\'{a}bor Elek and Nicolas Monod.
\newblock On the topological full group of a minimal {C}antor
  {$\bold{Z}^2$}-system.
\newblock {\em Proc. Amer. Math. Soc.}, 141(10):3549--3552, 2013.

\bibitem[Gen17]{Gen-V}
Anthony Genevois.
\newblock Embeddings into thompson's groups from quasi-median geometry.
\newblock {\em arXiv:1709.03888}, 2017.

\bibitem[GJS09]{Gao-Jack-Sew}
Su~Gao, Steve Jackson, and Brandon Seward.
\newblock A coloring property for countable groups.
\newblock {\em Math. Proc. Cambridge Philos. Soc.}, 147(3):579--592, 2009.

\bibitem[Gla76]{Gla-book}
Shmuel Glasner.
\newblock {\em Proximal flows}.
\newblock Lecture Notes in Mathematics, Vol. 517. Springer-Verlag, Berlin-New
  York, 1976.

\bibitem[GM07]{Gla-Mo}
Y.~Glasner and N.~Monod.
\newblock Amenable actions, free products and a fixed point property.
\newblock {\em Bull. Lond. Math. Soc.}, 39(1):138--150, 2007.

\bibitem[GN05]{Gri-Nek-amen}
R.~Grigorchuk and V.~Nekrashevych.
\newblock Amenable actions of nonamenable groups.
\newblock {\em Zap. Nauchn. Sem. S.-Peterburg. Otdel. Mat. Inst. Steklov.
  (POMI)}, 326(Teor. Predst. Din. Sist. Komb. i Algoritm. Metody. 13):85--96,
  281, 2005.

\bibitem[GPS99]{G-P-S-full}
Thierry Giordano, Ian~F. Putnam, and Christian~F. Skau.
\newblock Full groups of {C}antor minimal systems.
\newblock {\em Israel J. Math.}, 111:285--320, 1999.

\bibitem[Gri00]{Gri-branch}
R.~I. Grigorchuk.
\newblock Just infinite branch groups.
\newblock In {\em New horizons in pro-{$p$} groups}, volume 184 of {\em Progr.
  Math.}, pages 121--179. Birkh\"{a}user Boston, Boston, MA, 2000.

\bibitem[Gro93a]{Gro-asymptotic}
M.~Gromov.
\newblock Asymptotic invariants of infinite groups.
\newblock In {\em Geometric group theory, {V}ol.\ 2 ({S}ussex, 1991)}, volume
  182 of {\em London Math. Soc. Lecture Note Ser.}, pages 1--295. Cambridge
  Univ. Press, Cambridge, 1993.

\bibitem[Gro93b]{Gro}
M.~Gromov.
\newblock Asymptotic invariants of infinite groups.
\newblock In {\em Geometric group theory, {V}ol. 2 ({S}ussex, 1991)}, volume
  182 of {\em London Math. Soc. Lecture Note Ser.}, pages 1--295. Cambridge
  Univ. Press, Cambridge, 1993.

\bibitem[GW15]{Glas-Wei}
E.~Glasner and B.~Weiss.
\newblock Uniformly recurrent subgroups.
\newblock In {\em Recent trends in ergodic theory and dynamical systems},
  volume 631 of {\em Contemp. Math.}, pages 63--75. Amer. Math. Soc.,
  Providence, RI, 2015.

\bibitem[Hae02]{Hae}
Andr\'e Haefliger.
\newblock Foliations and compactly generated pseudogroups.
\newblock In {\em Foliations: geometry and dynamics ({W}arsaw, 2000)}, pages
  275--295. World Sci. Publ., River Edge, NJ, 2002.

\bibitem[Hur15]{Hurt-diff}
Sebastian Hurtado.
\newblock Continuity of discrete homomorphisms of diffeomorphism groups.
\newblock {\em Geom. Topol.}, 19(4):2117--2154, 2015.

\bibitem[HZ97]{Har-Zal}
B.~Hartley and A.~E. Zalesski\u\i.
\newblock Confined subgroups of simple locally finite groups and ideals of
  their group rings.
\newblock {\em J. London Math. Soc. (2)}, 55(2):210--230, 1997.

\bibitem[JdlS15]{Ju-Sa-Wobbling}
Kate Juschenko and Mikael de~la Salle.
\newblock Invariant means for the wobbling group.
\newblock {\em Bull. Belg. Math. Soc. Simon Stevin}, 22(2):281--290, 2015.

\bibitem[Ji04]{Ji-asdim}
Lizhen Ji.
\newblock Asymptotic dimension and the integral {$K$}-theoretic {N}ovikov
  conjecture for arithmetic groups.
\newblock {\em J. Differential Geom.}, 68(3):535--544, 2004.

\bibitem[JM13]{Ju-Mo}
Kate Juschenko and Nicolas Monod.
\newblock Cantor systems, piecewise translations and simple amenable groups.
\newblock {\em Ann. of Math. (2)}, 178(2):775--787, 2013.

\bibitem[JMBMdlS18]{ExtAmen}
Kate Juschenko, Nicol\'{a}s Matte~Bon, Nicolas Monod, and Mikael de~la Salle.
\newblock Extensive amenability and an application to interval exchanges.
\newblock {\em Ergodic Theory Dynam. Systems}, 38(1):195--219, 2018.

\bibitem[JNdlS16]{J-N-S}
Kate Juschenko, Volodymyr Nekrashevych, and Mikael de~la Salle.
\newblock Extensions of amenable groups by recurrent groupoids.
\newblock {\em Invent. Math.}, 206(3):837--867, 2016.

\bibitem[Kel97]{tiling2}
Johannes Kellendonk.
\newblock Topological equivalence of tilings.
\newblock {\em J. Math. Phys.}, 38(4):1823--1842, 1997.
\newblock Quantum problems in condensed matter physics.

\bibitem[KP00]{tiling3}
Johannes Kellendonk and Ian~F. Putnam.
\newblock Tilings, {$C^*$}-algebras, and {$K$}-theory.
\newblock In {\em Directions in mathematical quasicrystals}, volume~13 of {\em
  CRM Monogr. Ser.}, pages 177--206. Amer. Math. Soc., Providence, RI, 2000.

\bibitem[Law98]{Law-book}
Mark~V. Lawson.
\newblock {\em Inverse semigroups}.
\newblock World Scientific Publishing Co., Inc., River Edge, NJ, 1998.
\newblock The theory of partial symmetries.

\bibitem[LB18]{LB-urs-lattice}
A.~Le~Boudec.
\newblock Amenable uniformly recurrent subgroups and lattice embeddings.
\newblock {\em in preparation}, 2018.

\bibitem[LBMB16]{LB-MB-subdyn}
A.~Le~Boudec and N.~Matte~Bon.
\newblock Subgroup dynamics and {C}$^\ast$-simplicity of groups of
  homeomorphisms.
\newblock {\em Ann. Sci. Ecole Norm. Sup. (to appear)}, 2016.

\bibitem[LL13]{Law-Le-Stone}
Mark~V. Lawson and Daniel~H. Lenz.
\newblock Pseudogroups and their \'{e}tale groupoids.
\newblock {\em Adv. Math.}, 244:117--170, 2013.

\bibitem[LN07]{Lav-Nek-LDA}
Y.~Lavrenyuk and V.~Nekrashevych.
\newblock On classification of inductive limits of direct products of
  alternating groups.
\newblock {\em J. Lond. Math. Soc. (2)}, 75(1):146--162, 2007.

\bibitem[LP02]{Lei-Pug-finlin}
Felix Leinen and Orazio Puglisi.
\newblock Confined subgroups in periodic simple finitary linear groups.
\newblock {\em Israel J. Math.}, 128:285--324, 2002.

\bibitem[LP03]{Lei-Pug-diag}
Felix Leinen and Orazio Puglisi.
\newblock Diagonal limits of finite alternating groups: confined subgroups,
  ideals, and positive definite functions.
\newblock {\em Illinois J. Math.}, 47(1-2):345--360, 2003.
\newblock Special issue in honor of Reinhold Baer (1902--1979).

\bibitem[LP05]{Lei-Pu-LDA}
Felix Leinen and Orazio Puglisi.
\newblock Some results concerning simple locally finite groups of 1-type.
\newblock {\em J. Algebra}, 287(1):32--51, 2005.

\bibitem[Man15]{Mann-diff}
Kathryn Mann.
\newblock Homomorphisms between diffeomorphism groups.
\newblock {\em Ergodic Theory Dynam. Systems}, 35(1):192--214, 2015.

\bibitem[Man16]{Mann-hom}
Kathryn Mann.
\newblock Automatic continuity for homeomorphism groups and applications.
\newblock {\em Geom. Topol.}, 20(5):3033--3056, 2016.
\newblock With an appendix by Fr\'{e}d\'{e}ric Le Roux and Mann.

\bibitem[Mat06]{Mat-simple}
Hiroki Matui.
\newblock Some remarks on topological full groups of {C}antor minimal systems.
\newblock {\em Internat. J. Math.}, 17(2):231--251, 2006.

\bibitem[Mat12]{Mat-hom}
Hiroki Matui.
\newblock Homology and topological full groups of \'etale groupoids on totally
  disconnected spaces.
\newblock {\em Proc. Lond. Math. Soc. (3)}, 104(1):27--56, 2012.

\bibitem[Mat15]{Mat-SFT}
Hiroki Matui.
\newblock Topological full groups of one-sided shifts of finite type.
\newblock {\em J. Reine Angew. Math.}, 705:35--84, 2015.

\bibitem[Mat16a]{Mat-prod}
Hiroki Matui.
\newblock \'etale groupoids arising from products of shifts of finite type.
\newblock {\em Adv. Math.}, 303:502--548, 2016.

\bibitem[{Mat}16b]{Mat-survey}
Hiroki {Matui}.
\newblock {Topological full groups of \'etale groupoids.}
\newblock In {\em {Operator algebras and applications. The Abel symposium 2015
  took place on the ship Finnmarken, part of the Coastal Express Line (the
  Norwegian Hurtigruten), from Bergen to the Lofoten Islands, Norway, August
  7--11, 2015}}, pages 197--224. Cham: Springer, 2016.

\bibitem[MB14]{MB-Liouville}
Nicol\'as Matte~Bon.
\newblock Subshifts with slow complexity and simple groups with the {L}iouville
  property.
\newblock {\em Geom. Funct. Anal.}, 24(5):1637--1659, 2014.

\bibitem[MB16]{MB-full}
N.~Matte~Bon.
\newblock Full groups of bounded automaton groups.
\newblock {\em J. Fractal Geometry (to appear)}, 2016.

\bibitem[Med11]{Med-isom}
K.~Medynets.
\newblock Reconstruction of orbits of {C}antor systems from full groups.
\newblock {\em Bull. Lond. Math. Soc.}, 43(6):1104--1110, 2011.

\bibitem[Mil16a]{Mil-hom2}
Emmanuel Militon.
\newblock Actions of groups of homeomorphisms on one-manifolds.
\newblock {\em Groups Geom. Dyn.}, 10(1):45--63, 2016.

\bibitem[Mil16b]{Mil-hom1}
Emmanuel Militon.
\newblock Actions of the group of homeomorphisms of the circle on surfaces.
\newblock {\em Fund. Math.}, 233(2):143--172, 2016.

\bibitem[MM14]{Mat-Mats}
Kengo Matsumoto and Hiroki Matui.
\newblock Continuous orbit equivalence of topological {M}arkov shifts and
  {C}untz-{K}rieger algebras.
\newblock {\em Kyoto J. Math.}, 54(4):863--877, 2014.

\bibitem[MM17]{Mat-Mats-Higman-Thompson}
Kengo Matsumoto and Hiroki Matui.
\newblock Full groups of {C}untz-{K}rieger algebras and {H}igman-{T}hompson
  groups.
\newblock {\em Groups Geom. Dyn.}, 11(2):499--531, 2017.

\bibitem[Nek04]{Nek-Cuntz}
Volodymyr~V. Nekrashevych.
\newblock Cuntz-{P}imsner algebras of group actions.
\newblock {\em J. Operator Theory}, 52(2):223--249, 2004.

\bibitem[Nek10]{Nek-free}
V.~Nekrashevych.
\newblock Free subgroups in groups acting on rooted trees.
\newblock {\em Groups Geom. Dyn.}, 4(4):847--862, 2010.

\bibitem[Nek13]{Nek-fp}
V.~Nekrashevych.
\newblock Finitely presented groups associated with expanding maps.
\newblock {\em arXiv:1312.5654v1}, 2013.

\bibitem[Nek15a]{Nek-simple}
V.~Nekrashevych.
\newblock Simple groups of dynamical origin.
\newblock {\em Ergodic Theory Dynam. Syst (to appear)}, 2015.

\bibitem[Nek15b]{Nek-hyp}
Volodymyr~V. Nekrashevych.
\newblock Hyperbolic groupoids and duality.
\newblock {\em Mem. Amer. Math. Soc.}, 237(1122):v+105, 2015.

\bibitem[Nek16a]{Nek-frag}
V.~Nekrashevych.
\newblock Palindromic subshifts and simple periodic groups of intermediate
  growth.
\newblock {\em Ann. Math. (2) (to appear), arXiv:1601.01033}, 2016.

\bibitem[Nek16b]{Nek-complexity}
Volodymyr Nekrashevych.
\newblock Growth of \'etale groupoids and simple algebras.
\newblock {\em Internat. J. Algebra Comput.}, 26(2):375--397, 2016.

\bibitem[Neu54]{Neumann}
B.~H. Neumann.
\newblock Groups covered by permutable subsets.
\newblock {\em J. London Math. Soc.}, 29:236--248, 1954.

\bibitem[Ols79]{Ol-Tarski1}
A.~Ju. Olshanskii.
\newblock An infinite simple torsion-free {N}oetherian group.
\newblock {\em Izv. Akad. Nauk SSSR Ser. Mat.}, 43(6):1328--1393, 1979.

\bibitem[Ols80]{Ol-Tarski2}
A.~Ju. Olshanskii.
\newblock An infinite group with subgroups of prime orders.
\newblock {\em Izv. Akad. Nauk SSSR Ser. Mat.}, 44(2):309--321, 479, 1980.

\bibitem[Osi05]{Os-asdim}
D.~Osin.
\newblock Asymptotic dimension of relatively hyperbolic groups.
\newblock {\em Int. Math. Res. Not.}, (35):2143--2161, 2005.

\bibitem[Par11]{Pardo}
E.~Pardo.
\newblock The isomorphism problem for {H}igman-{T}hompson groups.
\newblock {\em J. Algebra}, 344:172--183, 2011.

\bibitem[Pat99]{Pat-book}
Alan L.~T. Paterson.
\newblock {\em Groupoids, inverse semigroups, and their operator algebras},
  volume 170 of {\em Progress in Mathematics}.
\newblock Birkh\"{a}user Boston, Inc., Boston, MA, 1999.

\bibitem[Res16]{Resende}
Pedro Resende.
\newblock Lectures on \'{e}tale groupoids, inverse semigroup, and quantales.
\newblock {\em Lectures Notes for the GAMAP IP Meeting, Antwerp, 4-18
  September.}, 2016.
\newblock Available at
  https://www.math.tecnico.ulisboa.pt/~pmr/poci55958/gncg51gamap-version2.pdf.

\bibitem[Roe03]{Roe-coarse}
John Roe.
\newblock {\em Lectures on coarse geometry}, volume~31 of {\em University
  Lecture Series}.
\newblock American Mathematical Society, Providence, RI, 2003.

\bibitem[Roe05]{Roe-hyp}
John Roe.
\newblock Hyperbolic groups have finite asymptotic dimension.
\newblock {\em Proc. Amer. Math. Soc.}, 133(9):2489--2490, 2005.

\bibitem[Rub89]{Rubin}
Matatyahu Rubin.
\newblock On the reconstruction of topological spaces from their groups of
  homeomorphisms.
\newblock {\em Trans. Amer. Math. Soc.}, 312(2):487--538, 1989.

\bibitem[STY02]{Ska-Tu-Yu}
G.~Skandalis, J.~L. Tu, and G.~Yu.
\newblock The coarse {B}aum-{C}onnes conjecture and groupoids.
\newblock {\em Topology}, 41(4):807--834, 2002.

\bibitem[SZ93]{Se-Za-alt}
S.~K. Sehgal and A.~E. Zalesski\u\i.
\newblock Induced modules and some arithmetic invariants of the finitary
  symmetric groups.
\newblock {\em Nova J. Algebra Geom.}, 2(1):89--105, 1993.

\bibitem[Tho17]{Thomas}
S.~Thomas.
\newblock Uniformly recurrent subgroups of simple locally finite groups.
\newblock {\em Preprint}, 2017.

\bibitem[vD90]{Dou}
Eric~K. van Douwen.
\newblock Measures invariant under actions of {$F_2$}.
\newblock {\em Topology Appl.}, 34(1):53--68, 1990.

\bibitem[Wri12]{Wr-asdim}
Nick Wright.
\newblock Finite asymptotic dimension for {${\rm CAT}(0)$} cube complexes.
\newblock {\em Geom. Topol.}, 16(1):527--554, 2012.

\end{thebibliography}

\end{document}